\documentclass[a4paper]{amsart}

\title[Dense and comeager conjugacy classes]
  {Dense and comeager conjugacy classes in zero-dimensional dynamics}

\author{Michal Doucha}
\address{
  Institute of Mathematics\\
  Czech Academy of Sciences\\
  \v Zitn\'a 25\\
  115 67 Praha 1\\
  Czechia}
\email{doucha@math.cas.cz}

\author{Julien Melleray}
\address{
  Universit\'e Claude Bernard Lyon 1 \\
  Institut Camille Jordan \\
  Lyon \\
  France}
\email{melleray@math.univ-lyon1.fr}

\author{Todor Tsankov}
\address{
  Universit\'e Claude Bernard Lyon 1 \\
  Institut Camille Jordan \\
  Lyon \\
  France}
\email{tsankov@math.univ-lyon1.fr}

\thanks{M.~Doucha was supported by the GA\v{C}R project 25-15366S and by the Czech Academy of Sciences (RVO 67985840).
T.~Tsankov was partially supported by Institut Universitaire de France.}

\subjclass[2020]{Primary 37B05, 37B10; Secondary 03C25}
\keywords{Cantor dynamics, generic actions, symbolic dynamics, sofic subshifts, free groups, hyperbolic groups, non-finitely generated groups}

\usepackage[T1]{fontenc}
\usepackage[osf]{mathpazo}

\usepackage{eucal}
\usepackage{stmaryrd}
\usepackage{microtype}

\usepackage{hyperref}
\usepackage[initials, shortalphabetic]{amsrefs}

\usepackage{my-macros}

\usepackage{enumitem}
\setlist[enumerate,1]{label=(\roman*), font=\normalfont}

\usepackage{tikz-cd}
\usepackage{todonotes}
\usepackage{subcaption}

\numberwithin{equation}{section}

\newcommand{\Tr}{\textup{tr}}
\newcommand{\Min}{\textup{min}}
\newcommand{\Part}{\mathrm{part}}
\newcommand{\Mpmp}{\textup{min-pmp}}
\newcommand{\Mp}{\textup{pmp}}
\newcommand{\Per}{\textup{per}}
\newcommand{\Pertr}{\textup{per-tr}}
\newcommand{\ec}{\textup{ec}}

\newcommand{\edge}{\xrightarrow}

\DeclareMathOperator{\Cay}{Cay}
\DeclareMathOperator{\Ret}{Ret}

\begin{document}

\begin{abstract}
Given a countable group $G$, we initiate a systematic study of the Polish spaces of all minimal and topologically transitive actions of $G$ on the Cantor space by homeomorphisms, with a focus on the existence of comeager conjugacy classes in these spaces. We develop a general model-theoretic framework to study this and related questions, recovering on the way many existing results from the literature.

A substantial part of the paper is devoted to actions of free groups.
We show that in that case, there is a comeager conjugacy class in the space of minimal actions, as well as in the space of minimal, probability measure-preserving actions. The first one is the Fraïssé limit of all sofic minimal subshifts and the second, the universal profinite action.
The case of the integers was already treated by Hochman and there the two actions coincide with the universal odometer.
In the non-abelian case, they are substantially different and new techniques are required.

In the opposite direction, if $G$ is an amenable group which is not finitely generated, we show that there is no comeager conjugacy class in the space of all actions, and if $G$ is locally finite, also in the space of minimal actions.

Finally, we study the question of existence of a dense conjugacy class in the space of topologically transitive actions. We show that if $G$ is hyperbolic or virtually polycyclic, then such a dense conjugacy class exists iff $G$ is virtually cyclic, suggesting that the case of the integers may be exceptional.
\end{abstract}

\maketitle

\setcounter{tocdepth}{1}
\tableofcontents

%%%%%%%%%%%%%%%%%%%%%%%%%%%%%%%%%%%%%%%%%%%%%%%%%%

\section{Introduction}
\label{sec:introduction}

Understanding generic behavior of dynamical systems is a problem with long history that, in ergodic theory, goes back at least to Oxtoby and Ulam~\cite{Oxtoby1941}, and was later extensively studied by Halmos.  One way of formalizing the question is to consider a Polish space of dynamical systems of interest and ask about the properties of the ``typical system'' in the sense of Baire category. Often spaces of this type satisfy a topological zero--one law: for every isomorphism-invariant, sufficiently definable property, either it or its negation is generic.

In this paper, we are interested in actions of countable groups on the Cantor space, which we denote by $\Omega$. For a fixed group $G$, it is natural to parametrize all such actions by homomorphisms from $G$ to the Polish group $\Homeo(\Omega)$; these form a Polish space that we denote by $\Xi(G)$. The group $\Homeo(\Omega)$ acts on $\Xi(G)$ by conjugation and we call its orbits \df{conjugacy classes}. Two systems are in the same conjugacy class iff they are isomorphic in the usual dynamical sense.
The space $\Xi(G)$ always has a dense conjugacy class (see \autoref{p:dense-conj-class-Xi}) and thus the topological zero--one law applies. In particular, every conjugacy class is meager or comeager and there are two possibilities for the global structure of $\Xi(G)$, depending on the group $G$: either there is a comeager conjugacy class or all classes are meager. The existence of a comeager conjugacy class is a very strong property which implies that the generic dynamical behavior for the group $G$ is determined by a single system. This is quite unusual in other dynamical contexts such as ergodic theory.

The first rather striking result of this type in topological dynamics was proved by Kechris and Rosendal in \cite{Kechris2007a}. They established the existence of a comeager conjugacy class in $\Xi(\Z)$  (see also \cite{Akin2008} for a more concrete construction). This was generalized to the finitely generated free groups $\bF_d$ by Kwiatkowska~\cite{Kwiatkowska2012}. Doucha~\cite{Doucha2022p} further extended Kwiatkowska's result to free products of finite and cyclic groups. In the opposite direction, Hochman~\cite{Hochman2012} showed that there is no comeager conjugacy class in $\Xi(\Z^d)$ for $d \geq 2$ and this result was extended in \cite{Doucha2022p} to finitely generated nilpotent groups which are not virtually cyclic.

From a dynamical perspective, the systems in the comeager conjugacy classes in the cases discussed above are rather degenerate: for example, they are never topologically transitive. In order to find more interesting generic behavior, it is therefore natural to look for generic conjugacy classes in the spaces $\Xi_\Min(G)$ and $\Xi_\Tr(G)$ of minimal and topologically transitive actions, respectively. Here, the only results in the literature, due to Hochman~\cite{Hochman2008}, are for $G = \Z$: there is a generic conjugacy class in $\Xi_\Min(\Z)$, which is isomorphic to the universal odometer and $\Xi_\Min(\Z)$ is dense in $\Xi_\Tr(\Z)$. The tools used to prove these results are specific to $\Z$ and do not generalize to other groups.

It turns out that for studying the existence of a comeager conjugacy class in this and related situations, the most general and flexible tools come from logic. Fraïssé theory is a branch of model theory that studies generic structures and variations of these techniques have already been used by Kechris and Rosendal~\cite{Kechris2007a}. In its most basic form, Fraïssé's theorem and its extension proved by Ivanov~\cite{Ivanov1999}, gives a necessary and sufficient condition for the existence of a comeager isomorphism class, based on the family of allowed finitely generated substructures, when this family is countable: it has to be hereditary and satisfy the joint embedding and weak amalgamation properties. Group actions on the Cantor space fall naturally in the Fraïssé framework if one uses Stone duality to represent them as actions on a Boolean algebra (which is a countable structure). Then finitely generated substructures correspond to finitely generated Boolean $G$-algebras (or, dually, \df{subshifts}), which are the main objects of study of symbolic dynamics. So, for example, if one wants to study the space $\Xi(G)$, one considers the collection of all subshifts, for $\Xi_\Min(G)$, the collection of minimal subshifts, etc. The correspondence between the topology on the space of subshifts and the one on $\Xi(G)$ was already observed by Hochman in \cite{Hochman2008}.

The classes of all subshifts and minimal subshifts are hereditary and satisfy joint embedding and (full) amalgamation but, crucially, fail to be countable, which prevents the application of classical Fraïssé theory.
They do, however, come with a Polish topology, which can be viewed either as the Stone topology on the space of \df{quantifier-free types} or, more traditionally, as the Vietoris topology on the space of subshifts.
This is reminiscent of the situation in model theory, where there is also a notion of genericity coming from the omitting types theorem: a countable model is generic iff it is \df{atomic}, i.e., it realizes only isolated types, and such a model exists iff the isolated types are dense.
However, this theory is based on types with quantifiers and a different topology.
The main novelty of our abstract approach is to combine the two and develop a theory for \df{topological Fraïssé classes} of quantifier-free types.
Isolated types are replaced by the \df{projectively isolated} ones (see \autoref{df:proj-isolated}), a notion which generalizes the projectively isolated subshifts from \cite{Doucha2022p}.
With this in hand, we prove a general model-theoretic criterion: a topological Fraïssé class of structures admits a comeager isomorphism class iff projectively isolated types are dense (see \autoref{th:proj-atomic}). Moreover, we show that when a comeager conjugacy class exists, the corresponding structure can be represented as the usual Fraïssé limit of the projectively isolated quantifier-free types (note that these always form a countable class).
It is also worth noting that we do not assume compactness (because for our main application, the space $\Xi_\Min(G)$ is not compact), and this leads to a more delicate proof and makes apparent a surprising topological condition that is not present in previous work (condition \autoref{i:tF:cond-pi} in \autoref{df:top-Fraisse}).

The choice to put these results in an appendix does not reflect their relative importance but rather the fact that they require more model-theoretic background.
Our goal is to make the paper more accessible to readers interested in the dynamical content and willing to take some model theory on faith.
While it would be possible to translate the proofs into a more familiar dynamical language, they cannot be substantially simplified, and we believe that the benefit of the general applicability of the criterion outweighs the cost of abstraction. Another advantage of the abstract approach is that it is not sensitive to the parametrization space and can be adapted for example to the space of ``generalized subshifts'' used by Hochman in \cite{Hochman2008}.

The model-theoretic results apply in particular to the space $\Xi_\Min(G)$, but also to $\Xi(G)$ (recovering the result of \cite{Doucha2022p}), to the space of marked groups (recovering a result of \cite{Goldbring2023}; see \autoref{sec:appendix-groups}), and to the space of $G$-sets (recovering a result of \cite{Glasner2016}; see \autoref{sec:appendix-g-sets}).

Applying the criterion to $\Xi_\Min(G)$, we obtain the following (see \autoref{c:min-actions-comeager-ex}).
\begin{theorem}
  \label{intro:th:min-actions-comeager-ex}
  Let $G$ be a countable, infinite group. Then the following are equivalent:
  \begin{enumerate}
  \item There is a comeager conjugacy class in $\Xi_\Min(G)$.
  \item Minimal subshifts which are projectively isolated in $\cS_\Min(A^G)$ are dense in $\cS_\Min(A^G)$ for every $A$.
  \end{enumerate}
\end{theorem}

\subsection*{Minimal actions of the free group}
\label{sec:intro-new-results-free}

A sufficient condition for a minimal subshift to be projectively isolated in $\cS_\Min(A^G)$ is for it to be \df{sofic}, i.e., a factor of a subshift of finite type (or SFT, for short); see \autoref{sec:preliminaries} for the precise definitions. We do not know whether this condition is also necessary (see \autoref{q:proj-isolated-sofic}). The first of our main results for free groups is the following (see \autoref{th:minimal-sofic-dense} and \autoref{c:min-Z}; the case $G = \Z$ is  due to \cite{Bezuglyi2006} and \cite{Hochman2008}).

\begin{theorem}
  \label{intro:th:minimal-sofic-dense}
  Let $G$ be a finitely generated free group and let $A$ be a finite alphabet. Then the set of sofic minimal subshifts is dense in $\cS_\Min(A^G)$.
\end{theorem}
While the existence of a comeager conjugacy class in $\Xi_\Min(\Z)$ is established more easily than that in $\Xi(\Z)$, for non-abelian free groups, the situation is reversed.
For $\Z$, the sofic minimal subshifts are the finite subshifts, and it is relatively straightforward to show that they are dense in $\cS_\Min(A^\Z)$.
This fails dramatically for non-abelian free groups.
For example, the boundary action $\bF_d \actson \partial \bF_d$ can be realized as an SFT, which must be isolated because it is minimal.
We construct somewhat similar SFTs that factor inside any open subset of $\cS_\Min(A^{\bF_d})$.

Combining \autoref{intro:th:minimal-sofic-dense} with \autoref{intro:th:min-actions-comeager-ex}, we obtain the following.
\begin{cor}
  \label{intro:c:minimal-free-comeager}
  The space $\Xi_\Min(\bF_d)$, for $d\geq 1$, has a comeager conjugacy class, elements of which are isomorphic to the Fraïssé limit of the sofic minimal subshifts.
\end{cor}

Our main tool for constructing sofic minimal subshifts is a certain family of graphs, inspired by the Rauzy graphs used for studying SFTs for $\Z$, where the edges are labeled by the generators of the free group, and which we use for describing open neighborhoods in the space of subshifts. We call these graphs again \df{Rauzy graphs}. We start by characterizing the Rauzy graphs whose associated neighborhood contains a minimal subshift (see \autoref{th:minimal-Rauzy}). This characterization is not obvious and several natural attempts fail (see \autoref{ex:examples-Rauzy-graphs}).
Then for each of these graphs, we construct a sofic minimal subshift in the corresponding neighborhood. On the way, we also recover Kwiatkowska's result on the existence of a generic class in $\Xi(\bF_d)$ (see \autoref{c:comeager-conj-class-free-group}).

We remark that these techniques cannot work for $\Z^d$ when $d\geq 2$: it follows from the work of Hochman~\cite{Hochman2012} that sofic minimal subshifts are not dense in this case and the following problem remains open.
\begin{question}
  \label{intro:q:intro:min-Zd}
  Does there exist a comeager conjugacy class in $\Xi_\Min(\Z^d)$ for $d \geq 2$?
\end{question}

Next we turn to finding the correct generalization of Hochman's result that the generic minimal $\Z$-action is isomorphic to the universal odometer. We recall that the \df{universal odometer} (or the \df{universal profinite action} of $\Z$) is the action of $\Z$ on its profinite completion by translation (the \df{profinite completion} of a group is the compact group obtained by taking the inverse limit of all of its finite quotients).
There is a natural generalization of this action to the free group (namely, the profinite completion of $\bF_d$) but there is an obvious obstruction to its being generic: profinite actions are \df{pmp} (i.e., preserve a Borel probability measure), and the subset of minimal pmp actions is closed in $\Xi_\Min(G)$ and it is proper if $G$ is not amenable. It turns out that this is the only obstruction (see \autoref{th:min-pmp-free-group}).

\begin{theorem}
  \label{intro:th:min-pmp-free-group}
  The following hold:
  \begin{itemize}
  \item Finite $\bF_d$-subshifts are dense in the space of all minimal, pmp $\bF_d$-subshifts.

  \item The space of minimal, pmp actions of $\bF_d$ has a comeager conjugacy class whose elements are isomorphic to the universal profinite action of $\bF_d$.
  \end{itemize}
\end{theorem}

\subsection*{Non-finitely generated groups}

Our next results concern actions of groups that are not finitely generated, which exhibit a number of interesting dynamical properties. For example, the generic behavior of their dynamical systems is markedly different from the finitely generated case. A simple instance of this is that a generic element of $\Xi(G)$ is topologically transitive iff $G$ is not finitely generated (see \autoref{p:non-fg-gen-n-transitive}).
It is thus natural to ask whether the generic action in $\Xi(G)$ is minimal.
We have the following characterization for amenable groups (see \autoref{c:character-amen-non-fg}).

\begin{theorem}
  \label{intro:th:character-amen-non-fg}
  Let $G$ be a countably infinite amenable group. Then the following are equivalent:
  \begin{itemize}
  \item $G$ is locally finite;
  \item $\Xi_\Min(G)$ is comeager in $\Xi(G)$.
  \end{itemize}
\end{theorem}

For locally finite groups, we also obtain more detailed information about the generic properties of the action and its invariant measures (see \autoref{th:generic_behavior_locally_finite}).

The next theorem is the basis of our results concerning the existence of a generic conjugacy class in $\Xi(G)$ for non-finitely generated groups.
\begin{theorem}
  \label{intro:th:projectivly_isolated_iff_minimal_sofic}
Let $G$ be a countable group which is not finitely generated. Then a subshift is projectively isolated in the space of all subshifts iff it is minimal and sofic.
\end{theorem}

From this and \autoref{intro:th:min-actions-comeager-ex}, we obtain the following.
\begin{cor}
  \label{intro:c:a_generic_must_be_minimal}
Let $G$ be a countable group which is not finitely generated. If there exists a comeager conjugacy class in $\Xi(G)$, then the generic element of $\Xi(G)$ is minimal.
\end{cor}

Combining the above results and the fact that locally finite groups admit no sofic, minimal subshifts (see \autoref{p:no_minimal_sofic_locally_finite_group}), we get the following partial answer to a question from \cite{Doucha2022p}.
\begin{cor}
  \label{c:intro:comeager-not-fg}
Let $G$ be a countable amenable group which is not finitely generated. Then there is no comeager conjugacy class in $\Xi(G)$. If $G$ is locally finite, there is also no comeager conjugacy class in $\Xi_\Min(G)$.
\end{cor}

\subsection*{The space $\Xi_\Tr(G)$ of topologically transitive actions}
\label{sec:space-xi_trg-trans}

For $\Xi_\Tr(G)$, even the more basic question of the existence of a dense conjugacy class is quite delicate and interesting. First, for $G = \Z$, it follows from \cite{Hochman2008} that $\Xi_\Min(\Z)$ is dense in $\Xi_\Tr(\Z)$, so $\Xi_\Tr(\Z)$ admits even a comeager conjugacy class. It is easy to generalize this to virtually cyclic groups (see \autoref{p:dense-conj-class-finite-index}) but these are all the finitely generated examples we know.

\begin{question}
  \label{intro:q:transitive-virtually-cyclic}
  Let $G$ be a finitely generated group which is not virtually cyclic.
  \begin{enumerate}
  \item Is it possible that $\Xi_\Min(G)$ is dense is $\Xi_\Tr(G)$?
  \item Can $\Xi_\Tr(G)$ have a dense conjugacy class?
  \end{enumerate}
\end{question}

In \cite{Hochman2012}, Hochman claims that there is a dense conjugacy class in $\Xi_\Tr(G)$ for $G = \Z^d$ and provides a general argument that does not depend on the structure of the group. However, this argument is incorrect and we have the following theorem. See \autoref{sec:transitive-actions} for more details.

\begin{theorem}
  \label{th:intro:dense-transitive}
  Let $G$ be a group which is hyperbolic or virtually polycyclic. Then $\Xi_\Tr(G)$ has a dense conjugacy class iff $G$ is virtually cyclic.
\end{theorem}

The proof of this theorem proceeds through a general criterion, again based on subshifts, for the (non-)existence of a dense conjugacy class in $\Xi_\Tr(G)$ that we are able to verify in the above cases using a relativized version of the \df{special symbol property} introduced by Dahmani and Yaman in \cite{Dahmani2008}.
The proof for hyperbolic groups is based on finite automata.

\subsection*{Acknowledgments}
\label{sec:acknowledgments}

We are grateful to Alekos Kechris for asking questions which prompted some of this work, to Yves Cornulier for some advice in group theory, and to Andrew Marks, Denis Osin, and Ville Salo for useful discussions.

%%%%%%%%%%%%%%%%%%%%%%%%%%%%%%%%%%%%%%%%%%%%%%%%%%

\section{Preliminaries}
\label{sec:preliminaries}

\subsection{Minimal and transitive actions}
\label{sec:minim-trans-acti}

In this section, we fix some terminology and notation.

Let $G$ be a countable group and let $\Omega$ denote the Cantor space. We denote by $\Xi(G)$ the space of all actions $G \actson \Omega$ by homeomorphisms, that is,
\begin{equation}
  \label{eq:defn-Xi}
  \Xi(G) = \Hom \big(G, \Homeo(\Omega) \big).
\end{equation}
We equip $\Homeo(\Omega)$ with the uniform convergence topology and note that $\Xi \sub \Homeo(\Omega)^G$ is a closed subset, so a Polish space.

If $Z$ is a compact, zero-dimensional space, we denote by $\cB(Z)$ the Boolean algebra of clopen subsets of $Z$ and if $\cB$ is a Boolean algebra, we denote by $\tS(\cB)$ the space of ultrafilters of $\cB$. By Stone duality, an element $f \in \Homeo(\Omega)$ can be viewed as an automorphism of $\cB(\Omega)$, and the topology on $\Homeo(\Omega)$ is given by pointwise convergence on $\cB(\Omega)$ (seen as a discrete space). This gives a convenient way to view the topology on $\Xi(G)$: a subbasis consisting of clopen subsets is given by the collection of sets of the form
\begin{equation}
  \label{eq:subbasis-Xi-definition-topo-Homeo}
  \set{\xi \in \Xi(G) : a = \xi(g) \cdot b }, \quad \text{ for } g \in G, a, b \in \cB(\Omega).
\end{equation}
Another useful subbasis is given by the clopen subsets
\begin{equation}
  \label{eq:subbasis-Xi}
  \set{\xi \in \Xi(G) : a \cap \xi(g) \cdot b = \emptyset}, \quad \text{ for } g \in G, a, b \in \cB(\Omega).
\end{equation}
To see this, note that $a = \xi(g) \cdot b$ iff $a \cap (\xi(g) \cdot \neg b) = \emptyset$ and $\neg a \cap \xi(g) \cdot b = \emptyset$. If $G$ is generated by a subset $S$, then in \autoref{eq:subbasis-Xi-definition-topo-Homeo} and \autoref{eq:subbasis-Xi}, it is enough to consider elements $g \in S$.
\begin{defn}
  \label{df:Gflow}
  Let $G$ be a group. A \df{$G$-flow} is an action $G \actson Z$ on a non-empty, compact, Hausdorff space $Z$ by homeomorphisms. A \df{subflow} of $Z$ is a subset which is a $G$-flow, i.e., a subset which is non-empty, closed, and $G$-invariant.
\end{defn}

If $G \actson Z_1$ and $G \actson Z_2$ are $G$-flows, a \df{factor map} $\pi \colon Z_1 \to Z_2$ is a $G$-equivariant, continuous, surjective map. In this situation, we will say that $Z_2$ is a \df{factor} of $Z_1$ or that $Z_1$ is an \df{extension} of $Z_2$.

The space $\Xi(G)$ is equipped with a continuous action of $\Homeo(\Omega)$ by conjugation:
\begin{equation*}
  (f \cdot \xi)(g) = f \xi(g) f^{-1}, \quad \text{ for } \xi \in \Xi(G), f \in \Homeo(\Omega), g \in G.
\end{equation*}
An orbit of this action is called a \df{conjugacy class} and we say that two actions are \df{conjugate} if they are in the same class. The following two facts are well-known.

\begin{prop}
  \label{p:factor-closure-conj-class}
Let $G$ be a countable group, and $\xi_1$, $\xi_2 \in \Xi(G)$ be such that $\xi_2$ is a factor of $\xi_1$. Then $\xi_2$ belongs to the closure of the conjugacy class of $\xi_1$.
\end{prop}
\begin{proof}
Choose a neighborhood of $\xi_2$, which we may assume to be of the form
\[ O=\set{ \xi \in \Xi(G) : \forall g \in K \ \forall a \in A \ \xi(g)\cdot a=\xi_2(g)\cdot a}, \]
where $K$ is a finite subset of $G$ and $A$ is a clopen partition of $\Omega$. Let $\pi \colon \Omega \to \Omega$ witness that $\xi_2$ is a factor of $\xi_1$. Then $\pi^{-1}$ induces an isomorphism from the (finite) Boolean subalgebra of $\cB(\Omega)$ generated by the sets $\xi_2(g) \cdot a$ $(a \in A)$ onto the subalgebra generated by the sets $\xi_1(g) \cdot a$ ($a \in A$). By homogeneity of $\cB(\Omega)$ and Stone duality, there exists $f \in \Homeo(\Omega)$ such that $f  \big(\pi^{-1}(\xi_2(g)\cdot a) \big)=\xi_1(g)\cdot a$ for all $a \in A$ and all $g \in K$. Then $f \cdot \xi_1 \in O$.
\end{proof}

Recall that a subset of a topological space is \df{Baire measurable} if it differs from an open set by a meager set. In particular, all Borel sets are Baire measurable. We refer the reader to \cite{Kechris1995} for basic facts about Baire category (note however that in that book, the terminology \emph{sets with the Baire property} is used to denote what we call here Baire measurable sets).
\begin{prop}
  \label{p:dense-conj-class-Xi}
  Let $G$ be a countable group. Then $\Xi(G)$ admits a dense conjugacy class. In particular, every conjugacy invariant, Baire measurable subset of $\Xi(G)$ is meager or comeager.
\end{prop}
\begin{proof}
  If $\set{\xi_i : i \in \N}$ is a countable dense subset of $\Xi(G)$, then the conjugacy class of the action $\prod_i \xi_i$ is dense by~\autoref{p:factor-closure-conj-class}. The second claim follows from \cite{Kechris1995}*{8.46}.
\end{proof}

\begin{defn}
  \label{df:syndetic}
  Let $G \actson Z$ be a $G$-set. A subset $L \sub Z$ is called \df{syndetic} if finitely many $G$-translates of $L$ cover $Z$. It is called \df{thick} if it intersects every syndetic set, equivalently, if every finite set of translates of $L$ has non-empty intersection. A subset $L \sub G$ is called \df{left syndetic} if it is syndetic for the action $G \actson G$ by left translation and it is called \df{left thick} if it is thick for this action. Similarly, \df{right syndetic} and \df{right thick} for the action by right translation.
\end{defn}

\begin{defn}
  \label{df:min-trans}
  A $G$-flow $G \actson Z$ is called \df{topologically transitive} (or simply \df{transitive}), if for all non-empty, open $U, V \sub Z$, there is $g \in G$ such that $g \cdot U \cap V \neq \emptyset$. It is called \df{minimal} if every non-empty, open $U \sub Z$ is syndetic. Equivalently, $G \actson Z$ is minimal if it does not have any proper subflows.

  For a $G$-flow $G \actson Z$ and $H \leq G$, we say that a point $z \in Z$ is \df{$H$-fixed} if $H \cdot z = z$. We say that $z$ is \df{periodic} if it is $H$-fixed for some $H \leq G$ of finite index. $Z$ is \df{periodic} if it is finite.
\end{defn}

We note that the two properties of being transitive or being minimal are preserved under factors.
 It is clear that every minimal action is transitive. An action is minimal iff every orbit is dense in $Z$ and if $Z$ is metrizable, an action is transitive iff it admits a dense orbit. If $G \actson Z$ is minimal, then $Z$ is either finite or perfect, but this is not true for transitive actions: it is possible to have an infinite dense orbit consisting of isolated points.

If $Z = \Omega$, in \autoref{df:min-trans}, we can take $U$ and $V$ to be elements of $\cB(\Omega)$. We will denote by $\Xi_\Tr(G)$ and $\Xi_\Min(G)$ the set of $\xi \in \Xi(G)$, which are transitive and minimal, respectively.

\begin{prop}
  \label{p:trans-min-Gdelta}
 Let $G$ be a countable group. The sets $\Xi_\Tr(G)$ and $\Xi_\Min(G)$ are $G_\delta$ subsets of $\Xi(G)$ and, therefore, Polish spaces. Moreover, they are invariant under the conjugation action of $\Homeo(\Omega)$.
\end{prop}
\begin{proof}
  For $\Xi_\Tr$, this follows from the countability of $\cB(\Omega)$ and the fact that for fixed $a, b \in \cB(\Omega)$ and $g \in G$, the condition $\xi(g) \cdot a \cap b \neq \emptyset$ defines an open subset of $\Xi(G)$ (see \autoref{eq:subbasis-Xi}). Similarly, for $\Xi_\Min$, use the fact that for a fixed $a \in \cB(\Omega)$, the existence of some finite $F \sub G$ such that $\bigcup_{g \in F} \xi(g)a = \Omega$ is an open condition.
\end{proof}

Let $\mu$ be a Borel probability measure (or just a measure, for short) on $\Omega$. We note that, by regularity, $\mu$ is determined by its values on $\cB(\Omega)$ and it follows from Carathéodory's theorem that any finitely additive measure on $\cB(\Omega)$ extends to a Borel measure on $\Omega$. The measure $\mu$ is called \df{invariant} under an action $G \actson \Omega$ if $\mu(g \cdot a) = \mu(a)$ for all $g \in G$ and $a \in \cB(\Omega)$ (and thus, for all Borel sets). We will denote by $\Xi_\Mp(G)$ the elements of $\Xi(G)$ which admit an invariant  measure, and by $\Xi_{\Mpmp}(G)$ the intersection $\Xi_\Min(G) \cap \Xi_\Mp(G)$. Note that the property of having an invariant measure is preserved under factors.

\begin{prop}
  \label{p:pmp-closed}
 Let $G$ be a countable group. The set $\Xi_{\Mp}(G)$ is closed in $\Xi(G)$. Therefore $\Xi_{\Mpmp}(G)$ is a $G_\delta$ subset of $\Xi(G)$.
\end{prop}
\begin{proof}
  Let $(\xi_n)_{n \in \N}$ be a sequence of elements of $\Xi(G)$ which converges to some $\xi \in \Xi(G)$; let also $\mu_n$ be an invariant measure for $\xi_n$ for every $n$. Then any limit point $\mu$ of the $\mu_n$ in the compact space of probability measures on $\Omega$ is $\xi$-invariant.
\end{proof}

\subsection{Subshifts}
\label{sec:subshifts}

Let $G \actson Z$ be a $G$-flow. We will denote by $\cS(Z)$ the set of all subflows of $Z$. It is a closed subspace of the space of all closed subsets of $Z$, equipped with the Vietoris topology, and thus a compact space (metrizable if $Z$ is metrizable). We recall that the Vietoris topology is generated by sets of the form $\set{K \in \cS(Z) \colon K \subseteq U}$ and $\set{K \in \cS(Z) \colon K \cap U \ne \emptyset}$ with $U$ an open subset of $Z$.

We will denote by $\cS_\Tr(Z)$, $\cS_\Min(Z)$, and $\cS_\Per(Z)$ the collections of transitive, minimal, and periodic subflows, respectively. We also let $\cS_\Pertr(Z) = \cS_\Per(Z) \cap \cS_\Tr(Z)$.

\begin{prop}
  \label{p:trmin-subflows-Gdelta}
  Let $G \actson Z$ be a $G$-flow with $G$ countable and $Z$ metrizable. Then $\cS_\Tr(Z)$ and $\cS_\Min(Z)$ are $G_\delta$ subsets of $\cS(Z)$.
\end{prop}
\begin{proof}
  Let $\set{U_n : n \in \N}$ be a basis for $Z$.
  For $X \in \cS(Z)$, we have that $X$ is transitive iff for all $i, j$, $X \cap U_i = \emptyset$ or $X \cap U_j = \emptyset$ or there is $g \in G$ such that $X \cap U_i \cap g \cdot U_j \neq \emptyset$. For fixed $i,j$ this is a $G_\delta$ condition since in a metrizable space closed subsets are $G_\delta$, and a finite union of $G_\delta$ subsets is a $G_\delta$ subset.
   Similarly, $X$ is minimal iff for every $i$, $U_i \cap X = \emptyset$ or $X \sub G \cdot U_i$. Both conditions are $G_\delta$ (one is closed and the other is open).
\end{proof}

\begin{prop}
  \label{p:basis-negative}
  Let $G \actson Z$ be a $G$-flow.
  The collection
  \begin{equation*}
    \set[\big]{\set{X \in \cS_\Min(Z) : X \sub U} : U \sub Z \text{ open}}
  \end{equation*}
  forms a basis of open sets of $\cS_\Min(Z)$. If $Z$ is zero-dimensional, one can take the sets $U$ to be clopen.
\end{prop}
\begin{proof}
  Let $X_0 \in \cS_\Min(Z)$. A basic open neighborhood of $X_0$ in the Vietoris topology is given by the collection of all $X \in \cS_\Min(Z)$ such that $X \sub U$ and $X \cap V_i \neq \emptyset$, $i = 1, \ldots, n$ for some non-empty, open $U, V_1, \ldots, V_n \sub Z$. We claim that the set
  \begin{equation}
    \label{eq:defn-basis-X0}
    \set{X \in \cS_\Min(Z) : X \sub U \cap G \cdot V_1 \cap \cdots \cap G \cdot V_n}
  \end{equation}
  contains $X_0$ and is contained in this open neighborhood. First, by minimality of $X_0$, $X_0 \sub G \cdot V_i$ for every $i$. Second, if $X \in \cS_\Min(Z)$ and $X \sub G \cdot V_i$, it is clear that $X \cap V_i \neq \emptyset$.

  For the last statement, note that by compactness and $0$-dimensionality, if $X$ is closed, $U$ is open and $X \subseteq U$ then there exists a clopen subset $V$ such that $X \subseteq V \subseteq U$. So we can replace $U \cap G \cdot V_1 \cap \cdots \cap G \cdot V_n$ in \autoref{eq:defn-basis-X0} by a clopen subset.
\end{proof}

Let $A$ be a finite set. The group $G$ acts on the compact space $A^G$ by left shift:
\begin{equation*}
  (g \cdot x)(h) = x(g^{-1}h).
\end{equation*}
A \df{subshift} is an element of $\cS(A^G)$. We note that for any $x \in A^G$, the subshift $\overline{G\cdot x}$ is minimal iff for any finite $F \subseteq G$ the set $\{g : (g \cdot x)|_{F}=x|_{F}\}$ is left syndetic.

Let $F \sub G$ be finite. An \df{$F$-pattern} is a function $p \colon F \to A$ and $F$ is called the \df{support} of $p$. Each $F$-pattern $p$ defines a clopen cylinder
\begin{equation}
  \label{eq:def-cyl-pattern}
  C_p \coloneqq \set{x \in A^G : x|_F = p}
\end{equation}
and two clopen subsets of $\cS(G)$
\begin{align*}
  \cU_p^+ &\coloneqq \set{X \in \cS(A^G) : X \cap C_p \neq \emptyset} \text{ and }\\
  \cU_p^- &\coloneqq \cS(A^G) \sminus \cU_p^+ = \set{X \in \cS(A^G) : X \cap C_p = \emptyset}.
\end{align*}
We say that the pattern $p$ \df{occurs} in a subshift $X$ if $X \in \cU_p^+$. The pattern $p$ \df{occurs} in a point $x \in A^G$ if there is $g \in G$ such that $g \cdot x \in C_p$.
The collection
\begin{equation*}
  \set{\cU_p^+, \cU_p^- : p \text{ is a pattern}}
\end{equation*}
forms a subbasis for the topology of $\cS(A^G)$. If $P$ is a finite set of patterns, we denote $\cU^+_P = \bigcap_{p \in P} \cU^+_p$ and $\cU^-_P = \bigcap_{p \in P} \cU^-_p$.
A subshift $X \in \cS(A^G)$ is of \df{finite type} (or \df{SFT}, for short) if it is of the form
\begin{equation}
  \label{eq:SFT}
  X = A^G \sminus \bigcup_{p \in P} \bigcup_{g \in G} g \cdot C_p,
\end{equation}
where $P$ is a finite collection of \df{forbidden patterns}. In other words, $X$ is the largest subshift in which none of the forbidden patterns occur. It follows that if $X$ is the SFT given by \autoref{eq:SFT}, then the set
\begin{equation*}
  \cS(X) = \cU^-_P
\end{equation*}
is open in $\cS(A^G)$. It is also clear that SFTs are dense in $\cS(A^G)$. If $X$ is defined as in \autoref{eq:SFT} and $F$ is finite and contains $\dom(p)$ for every $p \in P$, then we say that $F$ is a \emph{defining window} for $X$.

We record the following corollary of \autoref{p:basis-negative}.
\begin{cor}
  \label{c:SFT-basis-minimal}
  The collection
  \begin{equation*}
    \set{\cS(X) \cap \cS_\Min(A^G) : X \text{ is an SFT}}
  \end{equation*}
  forms an open basis for the topology of $\cS_\Min(A^G)$.
\end{cor}

A subshift $X \in \cS(A^G)$ is \df{sofic} if there exists a finite alphabet $B$, an SFT $Y \in \cS(B^G)$ and a factor map $Y \to X$. It is well known that, by enlarging the alphabet $B$, one can always assume that this factor map is induced by a surjection from $B$ to $A$ (hence extends to a factor map $B^G \to A^G$).

A subshift $X \in \cS(A^G)$ is \df{isolated} in $\cS(A^G)$ if it is an isolated point in the topology of $\cS(A^G)$.
It is \df{projectively isolated} in $\cS(A^G)$ (see \cite{Doucha2022p}) if there exists a finite alphabet $B$, a factor map $\Phi \colon B^G \to A^G$, and a non-empty open subset $\cU \sub \cS(B^G)$ such that $\Phi[Y] = X$ for every $Y \in U$. (Here and in what follows, if $f \colon Z \to W$ is a function and $Y \sub Z$, we denote by $f[Y]$ the image of $Y$ by $f$, that is, the set $\set{f(y) : y \in Y}$.)
 Without loss of generality, we may assume that $\Phi$ comes from a map $\phi \colon B \to A$. As SFTs are dense, it follows that every isolated subshift is an SFT and every projectively isolated subshift is sofic.

We have the following easy sufficient condition for a subshift to be isolated in $\cS(A^G)$.
\begin{lemma}
  \label{l:isolated-point-SFT}
  Let $X \in \cS(A^G)$ be an SFT that has an isolated point with a dense orbit. Then $X$ is isolated in $\cS(A^G)$.
\end{lemma}
\begin{proof}
  Let $x_0 \in X$ be isolated and let $U \sub A^G$ be open with $U \cap X = \set{x_0}$. Then
  \begin{equation*}
    \set{X} = \set{Y \in \cS(X) : Y \cap U \neq \emptyset}. \qedhere
  \end{equation*}
\end{proof}
Similarly, if $X$ is sofic and has an isolated point with a dense orbit, then $X$ is projectively isolated in $\cS(A^G)$.

Another operation on subshifts that we will need is the \df{disjoint union}: if $X \in \cS(A^G)$ and $Y \in \cS(B^G)$, then $X \sqcup Y$ is a naturally a subshift of $(A \sqcup B)^G$. Here $A \sqcup B$ means that if $A$ and $B$ are not disjoint, we take the union of two disjoint copies.
\begin{lemma}
  \label{l:disjoint-union}
  Suppose that $G$ is finitely generated and $X \in \cS(A^G)$ and $Y \in \cS(B^G)$ are subshifts. If $X$ and $Y$ are SFT/sofic/isolated/projectively isolated, then so is $X \sqcup Y$.

If $X$, $Y$ are sofic/projectively isolated in $\cS(A^G)$ then so is $X \cup Y$.
\end{lemma}
\begin{proof}
  First we note that $A^G \sqcup B^G$ is an SFT (as a subshift of $(A \sqcup B)^G$). Indeed, if $S$ is a finite, symmetric generating set of $G$ not containing $1_G$, it is enough to forbid the patterns $p \colon S \cup \set{1_G} \to A \sqcup B$ with $p(1_G) \in A$ and $p(s) \in B$ for $s \in S$. Now the statements for disjoint unions of SFT and sofic subshifts are immediate, and the other two follow after observing that $\cS(A^G \sqcup B^G)$ is an open subset of $\cS((A \sqcup B)^G)$.

  The second statement follows from the first and the fact that there is a natural factor map $X \sqcup Y \to X \cup Y$.
\end{proof}

Now let $A$ be a \df{(clopen) partition} of $\Omega$, that is, a finite set of elements of $\cB(\Omega)$ which are pairwise disjoint and whose union is $\Omega$. If all of the elements of $A$ are non-empty, we will say that the partition is \df{non-degenerate}. If $A$ is a partition of $\Omega$ and $\omega \in \Omega$, we will denote by $A(\omega)$ the element $a \in A$ such that $\omega \in a$.
For a fixed $\xi \in \Xi$, we can define a map $\Pi^A_\xi \colon \Omega \to A^G$ by
\begin{equation}
  \label{eq:defin-PiAxi}
  \Pi^A_\xi(\omega)(g) = A(\xi(g)^{-1} \cdot \omega) \quad \text{ for } \omega \in \Omega, g \in G
\end{equation}
and we note that $\Pi^A_\xi$ is continuous and $G$-equivariant.
Using this, we define the map $\pi_A \colon \Xi(G) \to \cS(A^G)$ by:
\begin{equation}
  \label{eq:defn-piA}
  \pi_A(\xi) = \Pi^A_\xi[\Omega]
\end{equation}
and note that it is continuous. Moreover, the topology on $\Xi(G)$ is induced by the maps $\pi_A$ as $A$ varies over non-degenerate partitions of $\Omega$.

Every $X \in \cS(A^G)$ admits a canonical generating partition $\hat{A}(X)$ given by
\begin{equation*}
  \hat{A}(X) \coloneqq \set[\big]{\set{x \in X : x(1_G) = a} : a \in A},
\end{equation*}
which is in a bijection with a subset of $A$.

%%%%%%%%%%%%%%%%%%%%%%%%%%%%%%%%%%%%%%%%%%%%%%%%%%

\section{Fraïssé theory for Cantor actions}
\label{sec:Fraisse-cantor-actions}

As was first observed by Ivanov \cite{Ivanov1999} and more systematically exploited by Kechris and Rosendal in \cite{Kechris2007a}, Fraïssé theory provides a convenient framework for studying the existence of dense and comeager conjugacy classes. In \autoref{sec:append-g_delta-isom}, we prove a general criterion for the existence of a comeager conjugacy class and here we provide the necessary translations to apply this result in the context of Cantor dynamics.

First, we use Stone duality to convert dynamical systems into algebraic objects.
\begin{defn}
  \label{df:Boolean-algebra}
  A \df{Boolean $G$-algebra} is a Boolean algebra $\cA$ equipped with an action of $G$ by automorphisms. A Boolean $G$-algebra $\cA$ is \df{finitely generated} if there is a finite subset $A \sub \cA$ that generates $\cA$ as a Boolean $G$-algebra.
\end{defn}

It is easy to see that if a Boolean $G$-algebra is finitely generated, then the generating set $A$ can be taken to be a partition of $\bOne$. In that case, every element of $\cA$ can be written as a finite union of elements of the form
  \begin{equation*}
    \bigcap_{g \in F} g \cdot p(g),
  \end{equation*}
  where $p \colon F \to A$ is a pattern.

  Using the Stone functors $Z \mapsto \cB(Z)$, $\cA \mapsto \tS(\cA)$, we see that zero-dimensional $G$-flows correspond to Boolean $G$-algebras and subshifts of $A^G$ correspond to Boolean $G$-algebras generated by the finite partition $\hat A$. More precisely, $\cB(A^G)$ is the free Boolean $G$-algebra generated by the partition $\hat A$ and subshifts of $A^G$ correspond to its quotients. A subshift is of finite type if the corresponding ideal in $\cB(A^G)$ is finitely generated (over $G$). Factor maps of $G$-flows correspond to embeddings of Boolean $G$-algebras.

  Next we explain how to see the space of subshifts as a type space in the sense of \autoref{sec:append-g_delta-isom} to which we refer for the relevant definitions. Let $\cL$ be the language of Boolean algebras $\set{\bZero, \bOne, \cap, \cup, \neg}$ and let $\cL_G$ be $\cL$ augmented with unary function symbols for all elements of $G$. Let $\cF$ be the hereditary $\cL$-class of Boolean algebras and let $\cF(G)$ be the $\cL_G$-class of all $G$-Boolean algebras. For a finite set of variables $x$, it is easy to see that $\tS_x(\cF)$ is finite and $\tS_x(\cF(G))$ is compact (the theory of $G$-Boolean algebras is first-order, universal). We let $A = 2^x$ and
  \begin{equation*}
    \tS_A^\Part(\cF(G)) = \set{p \in \tS_A(\cF(G)) : p \models \bigsqcup_{\alpha \in A} \alpha = \bOne},
  \end{equation*}
  that is, the set of types where the variables form a partition of $\bOne$. For $\alpha \in A$, we let $t_\alpha(x)$ be the $\cL$-term $\bigcap_{v \in x} v^{\alpha(v)}$, where $v^1 = v$ and $v^0 = \neg v$.
  There is a homeomorphism $\Phi \colon \tS_x(\cF(G)) \to \tS_A^\Part(\cF(G))$ given by
  \begin{equation*}
    \Phi(q) \models \phi(A) \iff q \models \phi \big((t_\alpha(x) : \alpha \in A)\big) \quad \text{ for } q \in \tS_x(\cF(G)),
  \end{equation*}
  and another homeomorphism $\Psi \colon \cS(A^G) \to \tS_A^\Part(\cF(G))$ given by
  \begin{equation*}
    \Psi(X) \models \bigcap_{g \in F} g \cdot p(g) \neq \bZero \iff X \cap \cU_p \neq \emptyset
  \end{equation*}
  for $X \in \cS(A^G)$ and for every pattern $p \colon F \to A$ with $F \sub G$ finite. Moreover, using these identifications, if $y$ is a finite set of variables disjoint from $x$ and $B = 2^{x \cup y}$, then the projection map $\pi_x \colon \tS_{xy}(\cF(G)) \to \tS_x(\cF(G))$ can be represented as the projection $\cS(B^G) \to \cS(A^G)$ given by the map $B \to A$, $\alpha \mapsto \alpha|_x$. This shows that for all intents and purposes, we may identify the type space $\tS_x(\cF(G))$ and the space of subshifts $\cS(A^G)$. In particular, it makes sense to say that a certain class of subshifts is a topological Fraïssé class in the sense of \autoref{df:top-Fraisse}.

Next we consider coding models of $\cF(G)$, i.e., Boolean $G$-algebras. We use a coding by expansions. Let $u$ be a countable set of variables and let $c \in \cB(\Omega)^u$ be a tuple that bijectively enumerates $\cB(\Omega)$. Following \autoref{ex:expansions}, we let
\begin{equation*}
  \Xi(G) = \set{\xi \in \Xi_0(\cF(G)) : \xi|_\cL = \tp c}.
\end{equation*}
There is an obvious identification between the space defined above and the one defined by \autoref{eq:defn-Xi}.

\begin{prop}
  \label{p:all-subshift-Fraisse}
  Let $G$ be a countable group. Then the class $\cF(G)$ of $G$-subshifts is a topological Fraïssé class and $\Xi(G)$ is an admissible coding for models of $\cF(G)$ in the sense of \autoref{df:admissible-coding}.
\end{prop}
\begin{proof}
  To check amalgamation, let $X, Y_1, Y_2$ be $G$-subshifts and let $\pi_1 \colon Y_1 \to X$, $\pi_2 \colon Y_2 \to X$ be factor maps. Then
  \begin{equation*}
    Y_1 \times_X Y_2 = \set{(y_1, y_2) \in Y_1 \times Y_2 : \pi_1(y_1) = \pi_2(y_2)}
  \end{equation*}
  is a $G$-subshift which amalgamates $Y_1$ and $Y_2$ over $X$. The other two items of \autoref{df:top-Fraisse} are satisfied because each $\cS(A^G)$ is compact.

  That the coding $\Xi(G)$ is admissible follows from \autoref{ex:expansions}.
\end{proof}

We also note that the notion of a projectively isolated type (as in \autoref{df:proj-isolated}) coincides with the notion of a projectively isolated subshift discussed in \autoref{sec:subshifts}. From \autoref{c:proj-isolated-equivalence}, we obtain the following.

\begin{cor}[\cite{Doucha2022p}*{Theorem~3.1}]
  \label{c:all-actions-comeager-ex}
  Let $G$ be a countable group. Then the following are equivalent:
  \begin{enumerate}
  \item $\Xi(G)$ admits a comeager conjugacy class.
  \item Projectively isolated subshifts in $\cS(A^G)$ are dense in $\cS(A^G)$ for every $A$.
  \end{enumerate}
\end{cor}

Next we turn to the class of minimal subshifts that we denote by $\cF_\Min(G)$. We also let $\Xi_\Min(G) = \Xi(G) \cap \Xi_0(\cF_\Min(G))$ and again, there is an obvious identification with the previous definition.
\begin{prop}
  \label{p:minimal-subshift-Fraisse}
  Let $G$ be a countable group. Then the class $\cF_\Min(G)$ is a topological Fraïssé class. If $G$ is infinite, then $\Xi_\Min(G)$ is an admissible coding for models of $\cF_\Min(G)$.
\end{prop}
\begin{proof}
  We verify the conditions of \autoref{df:top-Fraisse}.

  \autoref{i:tF:amalg} The proof is the same as in \autoref{p:all-subshift-Fraisse}, except that now we have to take a minimal subshift of $Y_1 \times_X Y_2$.

  \autoref{i:tF:Gdelta} This follows from \autoref{p:trmin-subflows-Gdelta}.

  \autoref{i:tF:cond-pi} Let $A, B$ be finite alphabets and fix some surjective map $B \to A$. Let $\Pi \colon B^G \to A^G$ denote the corresponding projection. We define $\pi \colon \cS(B^G) \to \cS(A^G)$ by $\pi(X) = \Pi[X]$. Let $\cS_\Min(Z)$, where $Z \in \cS(B^G)$ is an SFT, be a basic open set in $\cS_\Min(B^G)$ (see \autoref{c:SFT-basis-minimal}). We claim that
  \begin{equation*}
    \pi[\cS_\Min(Z)] = \cS_\Min(\pi(Z)),
  \end{equation*}
  which is a closed set. The $\sub$ inclusion is obvious. For the other, let $X \sub \pi(Z)$ be a minimal subshift and let $Y$ be any minimal subshift of $Z \cap \Pi^{-1}(X)$. Then $\pi(Y) = X$.

  To verify that $\Xi_\Min(G)$ is an admissible coding, according to \autoref{ex:expansions}, we only have to check that every minimal $G$-subshift $X$ can be realized as a factor of a minimal action $G \actson \Omega$. If $X$ is infinite, then the underlying space of $X$ is homeomorphic to $\Omega$, so we may assume that $X$ is finite. Let $G \actson Z$ be any minimal action on a Cantor space. Then any minimal subset of $Z \times X$ is homeomorphic to the Cantor set and has $X$ as a factor.
\end{proof}

Applying \autoref{c:proj-isolated-equivalence} to this situation, we obtain the following.
\begin{cor}
  \label{c:min-actions-comeager-ex}
  Let $G$ be a countable, infinite group. Then the following are equivalent:
  \begin{enumerate}
  \item There is a comeager conjugacy class in $\Xi_\Min(G)$.
  \item Minimal subshifts which are projectively isolated in $\cS_\Min(A^G)$ are dense in $\cS_\Min(A^G)$ for every $A$.
  \end{enumerate}
\end{cor}

We have the following characterization of minimal subshifts which are projectively isolated in $\cS_\Min(A^G)$.
\begin{prop}
  \label{p:sofic-minimal-proj-isol}
  Let $G$ be a countable group and let $A$ be a finite alphabet. Then the following are equivalent:
  \begin{enumerate}
  \item $X \in \cS_\Min(A^G)$ is projectively isolated in $\cS_\Min(A^G)$.
  \item There exists a sofic $Y \in \cS(A^G)$ such that $X$ is the unique minimal subshift of $Y$.
  \end{enumerate}
  In particular, every sofic minimal subshift is projectively isolated in $\cS_\Min(A^G)$.
\end{prop}
\begin{proof}
  \begin{cycprf}
  \item[\impnext] Suppose that $X$ is projectively isolated. Then there exists a finite alphabet $B$, a map $\phi \colon B \to A$, and a non-empty open set $\cU \sub \cS(B^G)$ such that, denoting by $\Phi$ the $G$-map induced by $\phi$, $\Phi[Z] = X$ for every $Z \in \cU \cap \cS_\Min(B^G)$. By \autoref{c:SFT-basis-minimal}, we may assume that $\cU = \cS(Z_0)$ for some SFT $Z_0$. Then we can take $Y = \Phi[Z_0]$. It is clear that $X \sub Y$ and if $X'$ is another minimal subshift of $Y$, then any minimal subshift of $\Phi^{-1}(X') \cap Z_0$ does not map to $X$, contradicting the choice of $Z_0$.

  \item[\impfirst] Let $Z \in \cS(B^G)$ be an SFT and let $\Phi \colon B^G \to A^G$ be a factor map such that $\Phi[Z] = Y$. Now $\Phi$ must map every minimal subshift from the open set $\cS(Z)$ to a minimal subshift of $Y$, i.e., to $X$.
  \end{cycprf}
\end{proof}

\begin{question}
  \label{q:proj-isolated-sofic}
  Do there exist a countable group $G$ and a subshift in $\cS_\Min(A^G)$ which is projectively isolated in $\cS_\Min(A^G)$ but not sofic?
\end{question}
A negative answer to this question for $\Z^d$ would also imply a negative answer to \autoref{intro:q:intro:min-Zd}.

We finally consider the classes of periodic subshifts and periodic, transitive subshifts, that we will denote by $\cF_\Per(G)$ and $\cF_\Pertr(G)$, respectively.
We also denote $\Xi_\Per(G) = \Xi(G) \cap \Xi_0(\cF_\Per(G))$ and $\Xi_\Pertr(G) = \Xi(G) \cap \Xi_0(\cF_\Pertr(G))$. Actions in $\Xi_\Per(G)$ are often called \df{profinite} actions and they are exactly the equicontinuous actions on $\Omega$. A profinite action is in $\Xi_\Pertr(G)$ iff it is transitive iff it is minimal (an equicontinuous action preserves a metric, so if one orbit is dense, then all orbits are). We note that $\Xi_\Per(G)$ is a $G_\delta$ set in $\Xi(G)$.
\begin{prop}
  \label{p:periodic-subshifts}
  Let $G$ be a finitely generated group. Then:
  \begin{enumerate}
  \item \label{i:per:isolated} Each $X \in \cS_\Per(A^G)$ is isolated in $\cS(A^G)$.
  \item \label{i:per:Fraisse} $\cF_\Per(G)$ and $\cF_\Pertr(G)$ are topological Fraïssé classes.
  \item \label{i:per:admissible-coding-per} $\Xi_\Per(G)$ is an admissible coding for models of $\cF_\Per(G)$.
  \item \label{i:per:admissible-coding-pertr} If $G$ has arbitrarily large finite quotients, then $\Xi_\Pertr(G)$ is an admissible coding for models of $\cF_\Pertr(G)$.
  \end{enumerate}
\end{prop}
\begin{proof}
  \autoref{i:per:isolated} Let $X \in \cS_\Per(A^G)$. There is a finite index subgroup $G_0 \leq G$ which fixes the generating partition $\hat A(X)$ of $\cB(X)$. As $G$ is finitely generated, $G_0$ is also finitely generated and ``$G_0$ fixes $\hat{A}(X)$'' is an open condition on $X$. Now it remains to notice that $A^G$ contains only finitely many $G_0$-fixed points, so there are only finitely many subshifts in this open set (including $X$).

  \autoref{i:per:Fraisse} It follows from \autoref{i:per:isolated} that each $\cS_\Per(A^G)$ is open and discrete, so we only have to check amalgamation. Note that if $Y_1$, $Y_2$ are periodic subshifts then $Y_1 \times_X Y_2$ is also periodic. This allows us to amalgamate periodic subshifts.

Similarly, to amalgamate two periodic transitive subshifts $Y_1$, $Y_2$, simply consider an orbit of $Y_1\times_X Y_2$.

  \autoref{i:per:admissible-coding-per} This is clear as every periodic flow $G \actson Z$ is a factor of the flow $G \actson Z \times \Omega$, where the action on the second coordinate is trivial.

  \autoref{i:per:admissible-coding-pertr} Let $\Gamma$ be the profinite completion of $G$. By our assumptions, $\Gamma$ is infinite and metrizable, and therefore, homeomorphic to $\Omega$. Moreover, every transitive periodic $G$-flow is a factor of the left translation action $G \actson \Gamma$.
\end{proof}
The next result is a direct application of \autoref{th:proj-atomic}.
\begin{cor}
  \label{c:profinite-actions}
  Let $G$ be a finitely generated group and let $\Gamma$ be its profinite completion. Then the isomorphism class of the action $G \actson \Gamma \times \Omega$, where the action on the first coordinate is left translation and on the second,   is trivial, is comeager in $\Xi_\Per(G)$. If $G$ has arbitrarily large finite quotients, then the isomorphism class of the action $G \actson \Gamma$ is comeager in $\Xi_\Pertr(G)$.
\end{cor}

%%%%%%%%%%%%%%%%%%%%%%%%%%%%%%%%%%%%%%%%%%%%%%%%%%

\section{Neighborhoods of actions of the free group}
\label{sec:neighb-free-group}

\subsection{Rauzy graphs}
\label{sec:rauzy-graphs}

Let $G$ be the free group freely generated by a finite set of generators $S_0$. In this section, we will give a description of neighborhood bases of the spaces $\Xi(G)$ and $\Xi_\Min(G)$. We let $S = S_0 \cup S_0^{-1}$.

\begin{defn}
  \label{df:Rauzy-graph}
  A \df{Rauzy graph (for $G$)} is a tuple $\cG = (V, E, \sigma, \rho, \ell, \bar{\cdot})$, where $V$ is a set of vertices, $E$ is a set of edges, $\sigma, \rho \colon E \to V$ are the source and range maps, $\ell \colon E \to S$ is a labeling of the edges, and $\bar{\cdot} \colon E \to E$ is an involution, satisfying the following for all $e \in E$:
  \begin{itemize}
  \item $\sigma(\bar e) = \rho(e)$, $\rho(\bar e) = \sigma(e)$;
  \item $\ell(\bar e) = \ell(e)^{-1}$;
  \item the map $E \to V \times V \times S$, $e \mapsto (\sigma(e), \rho(e), \ell(e))$ is injective;
  \item for every $v \in V$ and $s \in S$, there exists $e \in E$ with $\sigma(e) = v$ and $\ell(e) = s$.
  \end{itemize}
  If there is an edge labeled $s$ from $v_1$ to $v_2$, we will write $v_1 \edge{s} v_2$.
  A \df{morphism} between two Rauzy graphs $\cG_1$ and $\cG_2$ is a map $\pi \colon V(\cG_1) \to V(\cG_2)$ such that for all $v_1, v_2 \in V(\cG_1)$ and $s \in S$, if $v_1 \edge{s} v_2$, then $\pi(v_1) \edge{s} \pi(v_2)$. A morphism $\pi \colon \cG_1 \to \cG_2$ induces a map $E(\cG_1) \to E(\cG_2)$ (because the edge $\pi(v_1) \edge{s} \pi(v_2)$ is unique), which we will also denote by $\pi$. We will say that the morphism $\pi$ is \df{surjective} if the induced map on the edges is surjective. When the Rauzy graph $\cG$ is given, we denote by $V(\cG)$ and $E(\cG)$ its vertex and edge set, respectively. A \df{subgraph} $\cG'$ of $ \cG$ is given by $V' \subseteq V(\cG)$, $E' \subseteq E(\cG)$ such that $\cG'=(V', E', \sigma, \rho, \ell, \bar{\cdot})$ satisfies the axioms of a Rauzy graph. If $\cG_1$, $\cG_2$ are Rauzy graphs and $\pi \colon \cG_1 \to \cG_2$ is a morphism, then $\pi[\cG_1]= (\pi[V(\cG_1)],\pi[E(\cG_1)])$ is a subgraph of $\cG_2$.
  A Rauzy graph is \df{connected} if the underlying graph (forgetting the labeling) is connected.
  A Rauzy graph is \df{deterministic} if the map $E \to V \times S$, $e \mapsto (\sigma(e), \ell(e))$ is injective.
\end{defn}

We note that a deterministic Rauzy graph defines an action $G \actson V$ by $s \cdot v = \rho(e)$, where $e$ is the unique edge with $\sigma(e) = v$ and $\ell(e) = s^{-1}$. In that case, the Rauzy graph is often called the \df{Schreier graph} of the action.

An example of a deterministic Rauzy graph is the \df{Cayley graph of $G$} $\Cay(G)$ given by
\begin{itemize}
\item $V = G$, $E = G \times S$;
\item $\sigma(g, s) = g$, $\rho(g, s) = gs$, $\overline{(g, s)} = (gs, s^{-1})$;
\item $\ell(g, s) = s$.
\end{itemize}
The action defined by this graph is the right translation action $G \actson G$, $g \cdot x = xg^{-1}$.

Every finite Rauzy graph $\cG$ defines a non-empty SFT $X(\cG) \in \cS(V(\cG)^G)$ by
\begin{equation}
  \label{eq:defn-XG}
  X(\cG) = \set{\pi \colon G \to V(\cG) : \pi \text{ is a morphism } \Cay(G) \to \cG}.
\end{equation}
To see that $X(\cG)$ is an SFT, note that one can obtain it by forbidding the patterns $p \colon \set{1_G, s} \to V(\cG)$ (for $s \in S$) where there is no edge labeled $s$ from $p(1_G)$ to $p(s)$. To construct a point $x \in X(\cG)$, one can choose $x(1_G)$ arbitrarily and then inductively choose $x(gs)$ (for $g \in G$, $s \in S$) such that $x(g) \edge{s} x(gs)$.

Let now $A$ be a non-degenerate partition of $\Omega$. A \df{Rauzy graph on $A$} is a Rauzy graph with vertex set $A$. Every $\xi \in \Xi(G)$ defines a labeled graph $\cG(\xi, A)$ on $A$ by
\begin{equation}
  \label{eq:Rauzy-from-action}
  a \edge{s} b \iff a \cap \xi(s) \cdot b \neq \emptyset \quad \text{for } a, b \in A, s \in S.
\end{equation}
The following lemma is immediate.
\begin{lemma}
  \label{l:xi-defines-Rauzy}
  For every $\xi \in \Xi(G)$ and every non-degenerate partition $A$, $\cG(\xi, A)$ is a Rauzy graph.
\end{lemma}

If $B$ is a partition of $\Omega$ that refines $A$, then there is a surjective morphism $\pi \colon \cG(\xi, B) \to \cG(\xi, A)$ defined by
\begin{equation*}
  \pi(b) = a \iff b \sub a.
\end{equation*}

Every Rauzy graph $\cG$ on $A$ defines a clopen subset $\cN(\cG) \sub \Xi(G)$ by
\begin{equation*}
  \cN(\cG) = \set{\xi \in \Xi(G) : \cG(\xi, A) = \cG }.
\end{equation*}
These definitions give us an encoding of a basis of open sets for $\Xi(G)$.

\begin{prop}
  \label{p:nbhd-Rauzy-Xi}
  Each $\cN(\cG)$ is non-empty and the collection
  \begin{equation*}
    \set{\cN(\cG) : A \text{ is a non-degenerate partition of } \Omega \text{ and $\cG$ is a Rauzy graph on $A$}}
  \end{equation*}
  forms a basis of open sets for the topology of $\Xi(G)$. More precisely, for $\xi \in \Xi(G)$, the collection
  \begin{equation*}
    \set{\cN(\cG(\xi, A)) : A \text{ is a non-degenerate partition of } \Omega}
  \end{equation*}
  is a basis at $\xi$.
\end{prop}
\begin{proof}
  $\cN(\cG)$ is non-empty because (a conjugate of) the product of $X(\cG)$ with a trivial action on $\Omega$ belongs to it.

  Let now $\xi_0 \in \Xi(G)$ and consider its basic neighborhood $U$ consisting of all $\xi$ such that $a_i \cap \xi(s_i) \cdot b_i = \emptyset$ for $i = 0, \dots, n-1$, where the $s_i$ are elements of $S$ and the $a_i, b_i$ are elements of $\cB(\Omega)$ (see \autoref{eq:subbasis-Xi}). Let $A$ be the partition of $\Omega$ generated by the $a_i$ and the $b_i$. It is easy to check that $\cN(\cG(\xi_0, A)) \sub U$: the condition $a_i \cap \xi(s_i) \cdot b_i = \emptyset$ is witnessed by the absence of edges labeled $s_i$ between the elements of the partition contained in $a_i$ and the ones contained in $b_i$.
\end{proof}

Rauzy graphs can also be defined from subshifts as follows. If $A$ is a finite alphabet and $F$ is a finite subset of $G$ containing $1_G$, we can define the graph $\cG(A^G, A^F)$ on $A^F$ in a way similar to \autoref{eq:Rauzy-from-action}:
\begin{equation*}
  p_1 \edge{s} p_2 \iff \text{$p_1$ and $s \cdot p_2$ are compatible}, \quad \text{ for } p_1, p_2 \in A^F,
\end{equation*}
where $s \cdot p_2$ is the pattern $sF \to A$ defined by $(s \cdot p_2)(f) = p_2(s^{-1}f)$ and two patterns are \df{compatible} if they agree on the intersection of their supports.
If $X \in \cS(A^G)$, we can similarly define $\cG(X, A^F)$, which is a subgraph of $\cG(A^G, A^F)$, by
\begin{equation*}
  V(\cG(X, A^F)) = \set{p \in A^F : p \text{ occurs in } X}
\end{equation*}
and
\begin{equation*}
  p_1 \edge{s} p_2 \iff \text{$p_1$ and $s \cdot p_2$ are compatible and $p_1 \cup s \cdot p_2$ occurs in $X$}.
\end{equation*}

We have a natural equivariant map $\iota_F \colon X \big(\cG(A^G, A^F) \big) \to A^G$ defined by
\begin{equation}
  \label{eq:defn-iotaF}
  \iota_F(x)(g) = x(g)(1_G) \quad \text{ for } x \in X \big(\cG(A^G, A^F) \big), \ g \in G.
\end{equation}
\begin{lemma}
  \label{l:iotaF-injective}
  Suppose that $F^{-1}$ is a connected subset of $\Cay(G)$. Then:
  \begin{enumerate}
  \item \label{i:l:iinj:1} $\iota_F(x)|_F = x(1_G)$ for all $x \in X \big(\cG(A^G, A^F) \big)$.
  \item \label{i:l:iinj:2} The map $\iota_F$ is an isomorphism from $X \big(\cG(A^G, A^F) \big)$ to $A^G$. More generally, if $Z \subseteq A^G$ is an SFT and $F$ is a defining window for $Z$, then the map $\iota_F$ defines an isomorphism between $X\big(\cG(Z, A^F) \big)$ and $Z$.
  \end{enumerate}
\end{lemma}
\begin{proof}
  \autoref{i:l:iinj:1} Fix $f \in F$. We have to show that
  \begin{equation}
    \label{eq:x-pattern-agree}
    x(f)(1_G) = x(1_G)(f).
  \end{equation}
  Write $f = s_0 \cdots s_{n-1}$ as a reduced word with $s_i \in S$ for all $i$. By the connectedness of $F^{-1}$, $s_i \cdots s_{n-1} \in F$ for all $i \leq n-1$. Note that for all $i$,
  \begin{equation*}
    x(s_0 \cdots s_{i-1}) \edge{s_i} x(s_0 \cdots s_i),
  \end{equation*}
  so, by the definition of the graph $\cG(A^G, A^F)$,
  \begin{equation*}
    \begin{split}
      x(s_0 \cdots s_{i-1})(s_i \cdots s_{n-1}) &= \big(s_i \cdot x(s_0 \cdots s_i) \big) (s_i \cdots s_{n-1}) \\
                                           &=  x(s_0 \cdots s_i)(s_{i+1} \cdots s_{n-1}).
    \end{split}
  \end{equation*}
  Applying this consecutively for $i = 0, \ldots, n-1$, we obtain \autoref{eq:x-pattern-agree}, as desired.

  \autoref{i:l:iinj:2} Define $j_F \colon A^G \to X \big(\cG(A^G, A^F) \big)$ by
 \begin{equation*}
     j_F(x)(g)(f)= x(gf) \quad \text{ for } x \in A^G, \ g \in G,\ f\in F.
 \end{equation*}
 We will check that $j_F$ is an inverse of $\iota_F$. By definition, for any $x \in A^G$, one has $(\iota_F \circ j_F)(x)(g)= j_F(x)(g)(1_G)=x(g) $. Conversely, given $x \in X\big(\cG(A^G, A^F) \big)$, $g \in G$, and $f \in F$, using \autoref{i:l:iinj:1}, we have:
\[j_F \big(\iota_F(x) \big)(g)(f)=\iota_F(x)(gf)=\iota_F(g^{-1} \cdot x)(f)=(g^{-1} \cdot x)(1_G)(f)=x(g)(f). \]

 For any $Z \in \cS(A^G)$ we have  $j_F [Z] \subseteq X\big(\cG(Z, A^F) \big)$. If $Z$ is an SFT with a defining window $F$, then $\iota_F[ X\big(\cG(Z, A^F) \big)] \subseteq Z$ since no forbidden $F$-pattern can occur in $\iota_F[ X\big(\cG(Z, A^F) \big)]$. So, in this situation, $\iota_F[ X\big(\cG(Z, A^F) \big)] = Z$.
\end{proof}

This lemma allows us, for any SFT $Z \subseteq A^G$ with defining window $F$, to view subshifts of $X \big(\cG(Z, A^F) \big)$ as elements of $\cS(Z)$.

\subsection{Minimal Rauzy graphs}
\label{sec:minimal-rauzy-graphs}

\begin{defn}
  \label{df:Rauzy-minimal}
  Let $\cG$ be a Rauzy graph. A \df{reduced path} in $\cG$ is a sequence of edges $(e_0, \ldots, e_{n})$ such that $\rho(e_i) = \sigma(e_{i+1})$ and $\ell(e_{i}) \neq \ell(e_{i+1})^{-1}$ for all $i < n$. The graph $\cG$ is called \df{minimal} if for all edges $e, f \in E(\cG)$, there exists a reduced path whose first edge is $e$ and whose last edge is $f$ or $\bar f$.
\end{defn}

For fixed $\xi \in \Xi(G)$ and partition $A$ of $\Omega$, every $\omega \in \Omega$ defines a morphism $\Cay(G) \to \cG(\xi, A)$, $g \mapsto A(g^{-1} \cdot \omega)$. This gives a natural way to obtain reduced paths from elements of the free group. More precisely,
let $\xi \in \Xi(G)$, let $A$ be a non-degenerate partition of $\Omega$, and let $\cG = \cG(\xi, A)$. Let $g \in G$ and $\omega \in \Omega$. Write $g = s_1 \cdots s_n$ as a reduced word in $S$ and let $\omega_i = (s_1 \cdots s_i)^{-1} \cdot \omega$. Then the path
\begin{equation}
  \label{eq:reduced-path}
A(\omega) \edge{s_1^{-1}} A(\omega_1) \edge{s_2^{-1}} \cdots \edge{s_n^{-1}} A(g^{-1} \cdot \omega)
\end{equation}
is reduced.

\begin{lemma}
  \label{l:Rauzy-xi-min}
  For every $\xi \in \Xi_{\Min}(G)$ and every non-degenerate partition $A$ of $\Omega$, $\cG(\xi, A)$ is a minimal Rauzy graph.
\end{lemma}
\begin{proof}
  That $\cG(\xi, A)$ is a Rauzy graph follows from \autoref{l:xi-defines-Rauzy}. To check minimality, let $e, f \in E$ be given. Let $a = \sigma(e)$, $s = \ell(e)$, and $a_0 = a \cap \xi(s) \cdot \rho(e)$. Similarly, let $b = \sigma(f)$, $t = \ell(f)$, and $b_0 = b \cap \xi(t) \cdot \rho(f)$. By the definition of $\cG(\xi, A)$, $a_0 \neq \emptyset$ and $b_0 \neq \emptyset$. Let $\omega \in a_0$. Let
  \begin{equation*}
    L = \set{g \in G : g \text{ is a reduced word ending in } s^{-1}}
  \end{equation*}
  and note that $L$ is infinite and for every $g \in G$, the set $gL \sdiff L$ is finite. This implies that $L$ is left thick, so it intersects the set $\set{g \in G : g \cdot \omega \in b_0}$, which is left syndetic since $\xi$ is minimal. Let $g^{-1}$ be in this intersection and consider the reduced path constructed in \autoref{eq:reduced-path}, whose first edge is equal to $e$ (because $\omega \in a_0$ and $g^{-1} \in L$). Denoting by $s_n$ the label of the last edge of this path, there are two possibilities. If $s_n = t$, then the last edge of the path is $\bar f$ (because $g^{-1} \cdot \omega \in b_0$). If not, we can add $f$ to the end of the path and it will still be reduced.
\end{proof}

\begin{lemma}
  \label{l:right-synd-surj}
  Let $\cG$ be a Rauzy graph and let $x \colon \Cay(G) \to \cG$ be a morphism such that for every $e \in E(\cG)$, the set
  \begin{equation}
    \label{eq:preimage-edge}
    \set{\sigma(f) : f \in E(\Cay(G)), \ x(f) = e}
  \end{equation}
  is right syndetic in $G$. Then every $y \in \cl{G \cdot x}$ is a surjective morphism $\Cay(G) \to \cG$. (Here we view the morphism $x \colon \Cay(G) \to \cG$ as an element of the subshift $X(\cG)$ defined by \autoref{eq:defn-XG}.)
\end{lemma}
\begin{proof}
  Let $y \in \cl{G \cdot x}$.
  Let $e \in E(\cG)$ and let $L$ be the set defined in \autoref{eq:preimage-edge}. As $L$ is right syndetic, there is a finite $F \sub G$ such that $LF^{-1} = G$. Let $g \in G$ be such that $y|_{F \cup FS} = (g \cdot x)|_{F \cup FS}$. Then there is $h \in F$ such that $g^{-1} \in Lh^{-1}$, i.e., $g^{-1}h \in L$. Together with the fact that $y(h) = x(g^{-1}h)$ and $y(h \ell(e)) = x(g^{-1}h \ell(e))$, this implies that $y$ maps the edge $h \edge{\ell(e)} h \ell(e)$ to $e$.
\end{proof}

If $\cG$ is a minimal Rauzy graph on a partition $A$ of $\Omega$, we denote
\begin{equation*}
  \cN_{\Min}(\cG) = \cN(\cG) \cap \Xi_{\Min}(G).
\end{equation*}
\begin{theorem}
  \label{th:minimal-Rauzy}
  If $\cG$ is minimal, the set $\cN_{\Min}(\cG)$ is non-empty and the collection
  \begin{multline*}
    \set{\cN_{\Min}(\cG) : A \text{ is a non-degenerate partition of } \Omega  \\
      \text{ and $\cG$ is a minimal Rauzy graph on $A$}}
  \end{multline*}
  forms a basis of open sets of the space $\Xi_{\Min}(G)$.
\end{theorem}
\begin{proof}
  We start by showing that $\cN_\Min(\cG)$ is non-empty. For each $v \in V(\cG)$ and $s \in S$, let $P_{v, s}$ be a reduced path starting from $v$ with an edge labeled $s$, which contains $e$ or $\bar e$ for every $e \in E(\cG)$. Such a path is easy to construct, visiting each edge consecutively, using the minimality of $\cG$. We will construct by induction a morphism $x \colon \Cay(G) \to \cG$ which satisfies the hypothesis of \autoref{l:right-synd-surj}. Start with $x_0$ with $\dom x_0 = \set{1_G}$ and $x_0(1_G) = v_0$, where $v_0$ is an arbitrary vertex of $\cG$.  Enumerate all edges of $\Cay(G)$ as $e_0, e_1, \dots$ (taking only one edge of each pair $\set{e, \bar e}$). Suppose that $x_n$ has been constructed and let $\cT_n = \dom x_n$. Let $e_i$ be the first edge of the enumeration which is not in $E(\cT_n)$ but is adjacent to $\cT_n$. We may assume that $v \coloneqq \sigma(e_i) \in V(\cT_n)$. Let $(f_0, \ldots, f_k)$ be the edges of the path $P_{v, \ell(e_i)}$ and define $x_{n+1}$ to be equal to $x_n$ on $\cT_n$ and set
  \begin{equation*}
    x_{n+1}\big(v \ell(f_0) \cdots \ell(f_j) \big) = \rho(f_j) \quad \text{ for all } j = 0, \ldots, k.
  \end{equation*}
  Finally, set $x = \bigcup_n x_n$. It is clear that $x$ is defined everywhere. Let $N$ be the maximum length of the paths $P_{v, s}$, for $v \in V(\cG), s \in S$ and let $B$ be the ball in $G$ around $1_G$ of radius $N$. For $e \in E(\cG)$, let $L_e$ be the set defined in \autoref{eq:preimage-edge}. It follows from the construction that $L_eB = G$ for all $e$, so the hypothesis of \autoref{l:right-synd-surj} is satisfied for $x$.

  Let $Y$ be any minimal subshift of $\cl{G \cdot x}$ and consider first the case where $Y$ is perfect. Recall that $A$ is a non-degenerate partition of $\Omega$, $V(\cG) = A$ and thus $Y \sub A^G$. Denote by $\eta$ the shift action on $Y$. By \autoref{l:right-synd-surj}, every $y \in Y$ is surjective on $A$, so $\hat A(Y)$ is a non-degenerate partition of $Y$. Let $\phi \colon Y \to \Omega$ be a homeomorphism sending the partition $\hat{A}$ to the partition $A$. It is clear that $\phi \eta \phi^{-1} \in \cN_\Min(\cG)$. If $Y$ is finite, we can consider instead a minimal action $G \actson \Omega$ which admits $Y$ as a factor (see the proof of \autoref{p:minimal-subshift-Fraisse}).

For the second part, it follows from \autoref{p:nbhd-Rauzy-Xi} that the sets $\cN_\Min(\cG(\xi, A))$, for $A$ a non-degenerate partition of $\Omega$, form a basis at $\xi$ and by \autoref{l:Rauzy-xi-min}, the graphs $\cG(\xi, A)$ are minimal.
\end{proof}

\begin{figure}[h!]
  \begin{subfigure}[t]{0.3\textwidth}
    \centering
    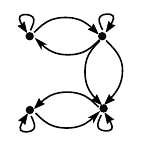
    \caption*{\autoref{i:rem:varRauzy:minimal} $\nRightarrow$ \autoref{i:rem:varRauzy:reduced-path-conn-edges}}
  \end{subfigure}%
  ~
  \begin{subfigure}[t]{0.5\textwidth}
    \centering
    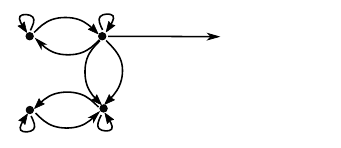
    \caption*{\autoref{i:rem:varRauzy:reduced-path-conn-vert} $\nRightarrow$ \autoref{i:rem:varRauzy:minimal}}
  \end{subfigure}
  \caption{Two examples of Rauzy graphs}
  \label{fig:RG}
\end{figure}

\begin{example}
  \label{ex:examples-Rauzy-graphs}
  We give two examples showing that some natural variants of \autoref{df:Rauzy-minimal} do not characterize minimality. We consider the following three conditions on a Rauzy graph $\cG$:

\begin{enumerate}
\item \label{i:rem:varRauzy:reduced-path-conn-vert} For any two vertices $v,w \in V(\cG)$, there exists a reduced path $(e_0,\ldots,e_n)$ with $\sigma(e_0)=v$ and $\rho(e_n)=w$.

\item \label{i:rem:varRauzy:minimal} $\cG$ is minimal in the sense of \autoref{df:Rauzy-minimal}.

\item \label{i:rem:varRauzy:reduced-path-conn-edges} For any two edges $e, f \in E(\cG)$, there exists a reduced path $(e_0,\ldots,e_n)$ with $e_0 = e$ and $e_n = f$.
\end{enumerate}

It is clear that \autoref{i:rem:varRauzy:reduced-path-conn-edges}  $\Rightarrow$ \autoref{i:rem:varRauzy:minimal} $\Rightarrow$ \autoref{i:rem:varRauzy:reduced-path-conn-vert}. The two examples in \autoref{fig:RG} (for the free group on two generators $s$, $t$) illustrate that both implications are strict.
\end{example}

%%%%%%%%%%%%%%%%%%%%%%%%%%%%%%%%%%%%%%%%%%%%%%%%%%

\section{Sofic minimal subshifts of the free group}
\label{sec:sofic-minim-subsh}

In this section, we set to prove that for free groups, sofic minimal subshifts are dense among all minimal subshifts. We retain the notation of the previous section.

We start by describing a device that produces sofic subshifts.
\begin{defn}
  \label{df:edge-selector}
  Let $\cG$ be a Rauzy graph for $G$. An \df{edge selector for $\cG$} is a triple $T = (v_0, T_0, T_1)$, where $v_0 \in V(\cG)$, $T_0 \colon S \to E(\cG)$, $T_1 \colon E(\cG) \times S \to E(\cG)$ are such that $\sigma(T_0(s)) = v_0$, $\ell(T_0(s)) = s$, $\sigma(T_1(e, s)) = \rho(e)$, and $\ell(T_1(e, s)) = s$ for all $(e, s) \in E(\cG) \times S$.
\end{defn}

If $T = (v_0, T_0, T_1)$ is an edge selector, the function $T_1$ can be extended inductively to a function $E(\cG) \times G \to E(\cG)$ by:
\begin{align*}
  T_1(e, 1_G) &= e \\
  T_1(e, ws) &= T_1(T_1(e, w), s) \quad \text{ for a reduced word } ws.
\end{align*}
An edge selector $T$ defines a morphism $x_T \colon \Cay(G) \to \cG$ by
\begin{align*}
  x_T(1_G) &= v_0 \\
  x_T(sw) &= \rho \big(T_1(T_0(s), w) \big) \quad \text{ for a reduced word } sw.
\end{align*}
This morphism defines a subshift $X(T) \sub X(\cG)$ by
\begin{equation*}
  X(T) = \cl{G \cdot x_T}.
\end{equation*}
We note that while for simplicity of notation, the map $T_1$ is defined everywhere on $E(\cG) \times S$, the values $T_1(e, s)$ with $s = \ell(e)^{-1}$ will never be used and can be arbitrary.

We also observe that edge selectors exist in abundance for any Rauzy graph $\cG$: in particular, there is one with $v_0 = v$ for every $v \in V(\cG)$.

\begin{prop}
  \label{p:edge-selector-sofic}
  Let $\cG$ be a Rauzy graph and let $T$ be an edge selector for $\cG$. Then the subshift $X(T)$ is projectively isolated in $\cS \big(V(\cG)^G \big)$ and, in particular, sofic.
\end{prop}
\begin{proof}
  Let $*$ be a new letter and let $B = E(\cG) \cup \set{*}$. Define a word $z_0 \in B^G$ as follows: $z_0(1_G) = *$, $z_0(s) = T_0(s)$ for $s \in S$, $z_0(ws) = T_1(z_0(w), s)$ for a reduced word $ws$. We note that
  \begin{equation}
    \label{eq:labels-z0}
    \ell(z_0(ws)) = s \quad \text{ for all reduced words } ws.
  \end{equation}
  Also, if $w_1 w_2$ is a reduced word, then
  \begin{equation}
    \label{eq:z0-determined}
    z_0(w_1 w_2) = T_1 \big(z_0(w_1), w_2 \big).
  \end{equation}

  Let $\phi \colon B \to V(\cG)$ be given by $\phi(*) = v_0$, $\phi(e) = \rho(e)$ for $e \in E(\cG)$, and let $\Phi \colon B^G \to V(\cG)^G$ be the corresponding $G$-map. It follows from the definitions of $x_T$ and $z_0$ that $\Phi(z_0) = x_T$. Now let $Z \in \cS(B^G)$ be the SFT defined by the following rules for patterns $p \colon S \cup \set{1_G} \to E(\cG)$ (see \autoref{fig:rules-Z}):
  \begin{enumerate}
  \item \label{i:edge-sel-rul:ran} $\ran p \sub \ran z_0$;
  \item \label{i:edge-sel-rul:reduced} if $p(1_G) \neq *$ and $p(s) \neq *$, then $\ell(p(s)) \neq \ell(p(1_G))^{-1}$ for $s \in S$;
  \item \label{i:edge-sel-rul:star} if $p(1_G) = *$, then $p(s) = T_0(s)$ for $s \in S$;
  \item \label{i:edge-sel-rul:not-star} if $\ell(p(1_G)) = s$, then for all $t \in S \sminus \set{s^{-1}}$, $p(t) = T_1(p(1_G), t)$, and $p(s^{-1}) = *$ or $p(1_G)=T_1(p(s^{-1}), s)$.
  \end{enumerate}

  \begin{figure}
    \begin{equation*}
      \begin{tikzcd}[row sep = large, column sep = large]
        & & & \bullet \\
        \bullet \arrow[r, "p(s^{-1})"', "?"] & \bullet \arrow[r, "s", "p(1_G)"'] & \bullet
        \arrow[ur, "s", "p(s)"', outer sep=-2pt]  \arrow[dr, "t^{-1}", "p(t^{-1})"', outer sep=-2pt]
        \arrow[r,  "t", "p(t)"'] &\bullet \\
        & & & \bullet
      \end{tikzcd}
    \end{equation*}
    \caption{The rules for $Z$ if $*$ does not appear}
    \label{fig:rules-Z}
  \end{figure}
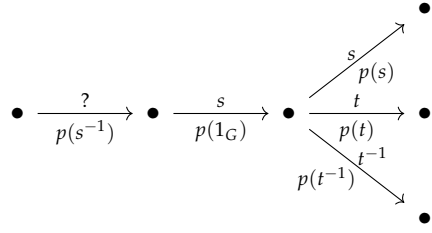

  We will show that $Z = \cl{G \cdot z_0}$. A direct verification shows that $z_0$ satisfies the rules, so we only have to prove that $G \cdot z_0$ is dense in $Z$. Let $z \in Z$. Suppose first that the symbol $*$ occurs in $z$. By translating by an appropriate element of $G$, we may assume that $z(1_G) = *$. But then, as $z_0$ is constructed inductively using rules \autoref{i:edge-sel-rul:star} and \autoref{i:edge-sel-rul:not-star}, we conclude that $z = z_0$. This shows that $z_0$ is an isolated point in $Z$.

  Next suppose that the symbol $*$ does not occur in $z$. Define inductively sequences $(s_n)_{n \in \N}$ of elements of $S$ and $(h_n)_{n \in \N}$ of elements of $G$ by:
\begin{equation*}
  h_0 = 1_G, \quad s_n = \ell(z(h_n))^{-1}, \quad h_{n+1} = h_n s_n.
\end{equation*}
It follows from rule \autoref{i:edge-sel-rul:reduced} that for all $n$, $s_n \neq s_{n+1}^{-1}$, and therefore, $h_n$ is a reduced word of length $n$. It also follows, using rule \autoref{i:edge-sel-rul:not-star}, that for all $n$ and for all $g \in G$ that do not start with $s_n$,
\begin{equation}
  \label{eq:z-determined}
  z(h_n g) = T_1(z(h_n), g).
\end{equation}
Let now $F \sub G$ be a finite set, let $n$ be larger than the length of all elements of $F$, and fix $f \in F$. The word $h_n^{-1}f$ starts with $s_{n-1}^{-1}$, so by \autoref{eq:z-determined},
\begin{equation*}
  z(f) = z\big(h_n(h_n^{-1}f)\big) = T_1 \big(z(h_n), h_n^{-1}f \big).
\end{equation*}
By rule \autoref{i:edge-sel-rul:ran}, there exists $g_0 \in G$ such that $z_0(g_0) = z(h_{n})$. It follows from \autoref{eq:labels-z0} that the last letter of $g_0$ is $s_{n}^{-1}$, so we can apply \autoref{eq:z0-determined} to the word $g_0(h_n^{-1}f)$ and obtain
\begin{equation*}
  z_0\big(g_0(h_n^{-1}f)\big) = T_1 \big(z_0(g_0), h_n^{-1}f \big) = T_1 \big(z(h_n), h_n^{-1}f \big) = z(f).
\end{equation*}
Thus $(h_ng_0^{-1}) \cdot z_0$ agrees with $z$ on $F$ and as $F$ was arbitrary, we conclude that $z \in \cl{G \cdot z_0}$.

Therefore $X(T) = \Phi[Z]$ and $Z$ is an isolated SFT by \autoref{l:isolated-point-SFT}. This implies that $X(T)$ is projectively isolated.
\end{proof}

\begin{theorem}[\cite{Doucha2022p}]
  \label{th:proj-isolated-free}
  Let $G$ be a finitely generated free group and let $A$ be a finite alphabet. Then projectively isolated subshifts are dense in $\cS(A^G)$.
\end{theorem}
\begin{proof}
  A basic open set in $\cS(A^G)$ is given by $\cS(Z) \cap \cU^+(P)$, where $Z \in \cS(A^G)$ is an SFT, $F \sub G$ is finite, and $P$ is the collection of $F$-patterns occurring in $Z$. We may further assume that $F^{-1}$ is connected. Let $\cG_0 = \cG(Z, A^F)$ and recall that $V(\cG_0) = P$. For every $p \in P$, let $T_p$ be an edge selector with $v_0 = p$. Let $X = \bigcup_{p \in P} X(T_p)$ and note that by \autoref{p:edge-selector-sofic} and \autoref{l:disjoint-union}, $X$ is projectively isolated. There is a natural map $\Phi \colon X \to A^G$, which is given by $\iota_F$ (as defined in \autoref{eq:defn-iotaF}) on each $X(T_p)$. It follows from the definition of $X(T_p)$ and \autoref{l:iotaF-injective} that $p$ occurs in $\Phi[X]$ for every $p \in P$. Thus $\Phi[X]$ belongs to the neighborhood $\cS(Z) \cap \cU^+(P)$ and it is projectively isolated.
\end{proof}

The following corollary is due to Kechris and Rosendal~\cite{Kechris2007a} for $G = \Z$ (see also \cite{Akin2008} for a different proof) and to Kwiatkowska~\cite{Kwiatkowska2012} in general. It implies that the group $\Homeo(\Omega)$ has ample generics, which has many further consequences. See \cite{Kwiatkowska2012} for more details.
\begin{cor}[\cite{Kwiatkowska2012}]
  \label{c:comeager-conj-class-free-group}
  Let $G$ be a finitely generated free group. Then the space $\Xi(G)$ has a comeager conjugacy class.
\end{cor}
\begin{proof}
  This follows from \autoref{th:proj-isolated-free} and \autoref{c:all-actions-comeager-ex}.
\end{proof}

Next we isolate a special type of edge selectors $T$ for which the shift $X(T)$ is minimal. Let $\cG$ be a Rauzy graph for $G$. A sequence of distinct edges $C = (e_0, \ldots, e_{n-1})$ of $\cG$ is called a \df{simple reduced cycle} if $(e_0, \ldots, e_{n-1}, e_0)$ is a reduced path. We denote by $\bar C$ the cycle $(\bar e_{n-1}, \ldots, \bar e_0)$. Sometimes, we will abuse notation and also use the letter $C$ for the set of edges of the cycle. We will say that an edge selector $T$ \df{follows} the cycle $C$ if
\begin{equation*}
  T_1 \big(e_i, \ell(e_{i+1}) \big) = e_{i+1} \quad \text{ for all } i < n.
\end{equation*}
Here and below all calculations with the indices of the cycle are done mod $n$.

\begin{defn}
  \label{df:recurrent-edge-sel}
  Let $\cG$ be a minimal Rauzy graph for $G$, let $T = (v_0, T_0, T_1)$ be an edge selector for $\cG$, and let $C = (e_0, \ldots, e_{n-1})$ be a simple reduced cycle. We say that $T$ is \df{recurrent for $C$} if the following conditions hold:
  \begin{enumerate}
  \item \label{i:res:follows} $T$ follows $C$ and $\bar C$.
  \item \label{i:res:v0} $v_0 = \sigma(e_0) = \rho(e_{n-1})$.
  \item \label{i:res:T0} $T_0\big(\ell(e_0)\big) = e_0$ and $T_0\big(\ell(e_{n-1})^{-1}\big) = \bar e_{n-1}$.
  \item \label{i:res:out-of-cycle-agree} $T_1(e_i, s) = T_1(\bar e_{i+1}, s)$ for all $i < n$ and $s \neq \ell(e_i)^{-1}, \ell(e_{i+1})$.
  \item \label{i:res:reachable} for every edge $e \in E(\cG)$, there exists a reduced word $w$ such that $\ell(e) w$ is reduced and $T_1(e, w) \in C \cup \bar C$.
  \end{enumerate}
\end{defn}

\begin{prop}
  \label{p:edge-selector-minimal}
  Let $\cG$ be a Rauzy graph, let $C = (e_0, \ldots, e_{n-1})$ be a simple reduced cycle in $\cG$, and let $T$ be an edge selector for $\cG$ that is recurrent for $C$. Then the subshift $X(T)$ is minimal.
\end{prop}
\begin{proof}
  Let $B_k$ denote the set of elements of $G$ that can be written as a reduced word of length at most $k$. As $X(T) = \cl{G \cdot x_T}$, it suffices to show that for every $k$, the set of $h \in G$ such that $(h^{-1} \cdot x_T)|_{B_{nk}} = x_T|_{B_{nk}}$ is right syndetic. To that end, let $k$ be given. Let $g_0 = s_0w_0$ be some non-identity element of $G$ (with $s_0 \in S$ and $s_0 w_0$ reduced). Let $e = T_1(T_0(s_0), w_0)$. By the definition of $x_T$, we have that $x_T(g_0) = \rho(e)$ and the last letter of $g_0$ is $\ell(e)$. By \autoref{df:recurrent-edge-sel} \autoref{i:res:reachable}, there exists $w_1$ such that $\ell(e) w_1$ is reduced and $T_1(e, w_1) \in C \cup \bar C$. By prolonging $w_1$ to follow the cycle if necessary, we may assume that $w_1$ is non-empty and $T_1(e, w_1) = e_0$ or $T_1(e, w_1) = \bar e_0$. Suppose that $T_1(e, w_1) = e_0$ (the other case is treated analogously). For $i < n$, let $t_i = \ell(e_i)$ and note that the last letter of $w_1$ is $t_0$, i.e., $w_1 = w_1' t_0$. Let $c = t_0 \cdots t_{n-1}$ and
  \begin{equation*}
    h = s_0 w_0 w_1' c^k
  \end{equation*}
  and note that the word on the right-hand side above is reduced.
  We claim that $(h^{-1} \cdot x_T)|_{B_{nk}} = x_T|_{B_{nk}}$. This is enough because the length of the word $w_1' c^k$ is uniformly bounded ($w_1'$ depends only on the edge $e$ and there are only finitely many edges).

  We aim to show that for every $u \in B_{nk}$, $x_T(u) = x_T(hu)$. First consider the case $u = 1_G$. Then
  \begin{equation}
    \begin{split}
    \label{eq:calc-xTh}
      x_T(h) &= T_1\big(T_0(s_0), w_0w_1'c^k\big) = T_1\big(e, w_1' c^k\big) \\
            &= T_1\big(e_0, (t_1 \cdots t_{n-1}) c^{k-1}\big) = v_0=x_T(1_G),
    \end{split}
  \end{equation}
  where the next-to-last equality uses the fact that $T$ follows $C$. The same argument shows that $x_T(hu) = x_T(u)$ if there is no cancellation in the concatenation of $h$ and $u$, i.e., if the first letter of $u$ is not $t_{n-1}^{-1}$. Suppose, finally, that there is cancellation. Recall that the length of $u$ is at most $nk$ and write $u = t_{n-1}^{-1} u_0 u_1$ as a reduced word so that
  \begin{equation*}
    hu = s_0 w_0 w_1' c^{q} t_0 \cdots t_{m-1} u_1
  \end{equation*}
  and the right-hand side of the equation above is a reduced word, where $0 \leq q < k$ and $0 \leq m < n$. Note that the last letter of $u_0$ is $t_{m}^{-1}$ and the first letter of $u_1$ is neither $t_{m}$ nor $t_{m-1}^{-1}$. On the one hand, using \autoref{df:recurrent-edge-sel} \autoref{i:res:T0}, we have
  \begin{equation}
    \label{eq:xTu}
    x_T(u) = \rho \Big( T_1\big(T_0(t_{n-1}^{-1}), u_0 u_1\big) \Big) = \rho \Big( T_1(\bar e_{n-1}, u_0 u_1) \Big)
    = \rho \Big( T_1(\bar e_{m}, u_1) \Big).
  \end{equation}
  On the other, using the same calculation as in \autoref{eq:calc-xTh}, we get
  \begin{equation}
    \label{eq:xThu}
    x_T(hu) = x_T(s_0 w_0 w_1' c^{q} t_0 \cdots t_{m-1} u_1) = \rho \Big( T_1(e_{m-1}, u_1) \Big).
  \end{equation}
  We conclude by noticing that, by \autoref{df:recurrent-edge-sel} \autoref{i:res:out-of-cycle-agree}, the expressions on the right in \autoref{eq:xTu} and \autoref{eq:xThu} are equal.
\end{proof}

Next we turn to the existence of recurrent edge selectors. \emph{From this point onward, we assume that $|S_0| \geq 2$, i.e., $G$ is non-abelian.}
\begin{lemma}
  \label{l:reduced-cycles-exist}
  Let $G$ be a non-abelian, finitely generated free group, let $\cG$ be a minimal Rauzy graph for $G$, and let $v \in V(\cG)$. Then there exists a simple reduced cycle $(e_0, \ldots, e_{n-1})$ with $\sigma(e_0) = \rho(e_{n-1}) = v$.
\end{lemma}
\begin{proof}
  First note that if we find a cycle that is reduced but not necessarily simple (i.e., one with possibly non-distinct edges), then we can convert it to a simple one by repeatedly replacing subcycles of the form $(e, \ldots, e)$ with the edge $e$. So for the rest of the proof, we will ignore the requirement that the edges be distinct.

  Let $s \in S$ and let $e, f \in E(\cG)$ with $\sigma(e) = \rho(f) = v$ and $\ell(e) = \ell(f) = s$. By minimality of $\cG$, there exists a reduced path $(e_0, \ldots, e_{n})$ with $e_0 = e$ and $e_{n} = f$ or $e_{n} = \bar f$. In the first case, $(e_0, \ldots, e_{n})$ is a reduced cycle and we are done. In the second, $\rho(e_{n-1}) = v$ and if $\ell(e_{n-1}) \neq s^{-1}$, then $(e_0, \ldots, e_{n-1})$ is a reduced cycle from $v$ to $v$. So we may assume that $\ell(e_{n-1}) = s^{-1}$.

  Let $t \in S \sminus \set{s, s^{-1}}$ and repeat the same argument as above to produce either a reduced cycle or a reduced path $(e_0', \ldots, e_{m-1}')$ from $v$ to $v$ with $\ell(e_0') = t$ and $\ell(e_{m-1}) = t^{-1}$. Then $(e_0, \ldots, e_{n-1}, e_0', \ldots, e_{m-1}')$ is a reduced cycle.
\end{proof}

\begin{lemma}
  \label{l:recurrent-selectors-exist}
  Let $\cG$ be a minimal Rauzy graph for $G$ and let $C$ be a simple reduced cycle in $\cG$. Then there exists an edge selector for $\cG$ recurrent for $C$.
\end{lemma}
\begin{proof}
  Write $C = (e_0, \ldots, e_{n-1})$. We will construct $T = (v_0, T_0, T_1)$ according to \autoref{df:recurrent-edge-sel}. First use \autoref{i:res:v0} and \autoref{i:res:T0} to define $v_0$ and $T_0$ (condition \autoref{i:res:T0} gives only partial constraints for $T_0$; for the other letters, define it arbitrarily). Next define $T_1$ on $C$ and $\bar C$ in order to satisfy conditions \autoref{i:res:follows} and \autoref{i:res:out-of-cycle-agree}.

  We will build collections of edges $E_0 \subset E_1 \subset \cdots$ and inductively extend the definition of $T_1$, so that at step $k$, condition \autoref{i:res:reachable} is satisfied for all edges in $E_k$ and the projection of $\dom T_1$ on $E(\cG)$ is contained in $E_k$.
  Set $E_0 = C \cup \bar C$ and note that for edges in $E_0$, \autoref{i:res:reachable} holds trivially (one can take $w$ to be the empty word). Suppose now that $E_k$ has been defined and that $E_k \neq E(\cG)$. Let $e \in E(\cG) \sminus E_k$ and using the minimality of $\cG$, let $(f_0, \ldots, f_{q})$ be a reduced path such that $f_0 = e$ and $f_{q} \in E_0$. Let $m$ be the least such that $f_m \in E_k$. By the inductive assumption, condition \autoref{i:res:reachable} is satisfied for $f_m$, so there exists a word $w$ such that $\ell(f_m) w$ is reduced and $T_1(f_m, w) \in E_0$. Moreover, $T_1$ is not defined on $f_i$ for any letter for any $i < m$. Now define $E_{k+1} = E_k \cup \set{f_i : i < m}$ and set $T_1 \big(f_i, \ell(f_{i+1}) \big) = f_{i+1}$ for $i < m$. Then for all $i < m$,
  \begin{equation*}
    T_1\big(f_i, \ell(f_{i+1}) \cdots \ell(f_m) w\big) = T_1(f_m, w) \in E_0
  \end{equation*}
  and the word $\ell(f_i) \ell(f_{i+1}) \cdots \ell(f_m) w$ is reduced, thus verifying \autoref{i:res:reachable} for the edges in $E_{k+1}$.

  As $E(\cG)$ is finite, for some $k$, $E_k = E(\cG)$. At this point, condition \autoref{i:res:reachable} is satisfied for all $e \in E(\cG)$ and we can extend $T_1$ arbitrarily to obtain a fully defined edge selector.
\end{proof}

\begin{theorem}
  \label{th:minimal-sofic-dense}
  Let $G$ be a non-abelian, finitely generated free group and let $A$ be a finite alphabet. Then the set of sofic minimal subshifts is dense in $\cS_\Min(A^G)$.
\end{theorem}
\begin{proof}
  Using \autoref{c:SFT-basis-minimal}, it suffices to find a sofic subshift in all neighborhoods of the form $\cS_\Min(Z)$, where $Z \in \cS(A^G)$ is an SFT. Let $F \sub G$ be a defining window for $Z$ with $1_G \in F$ and $F^{-1}$ connected. Let $X_0 \in \cS_\Min(Z)$ and let $\cG_0 = \cG(X_0, A^F)$. By \autoref{l:Rauzy-xi-min}, $\cG_0$ is minimal. Moreover, as $X_0 \sub Z$, the patterns in $V(\cG_0)$ are allowed in $Z$ and using the map $\iota_F$ from \autoref{l:iotaF-injective}, we can identify $X(\cG_0)$ with a subshift of $Z$.

Let $v \in V(\cG_0)$ be arbitrary and use \autoref{l:reduced-cycles-exist} to produce a simple reduced cycle starting at $v$. By \autoref{l:recurrent-selectors-exist}, there exists an edge selector $T$ for $\cG_0$ recurrent for $C$. Then by \autoref{p:edge-selector-sofic} and \autoref{p:edge-selector-minimal}, the subshift $X(T)$ is sofic and minimal. We have that $X(T) \sub X(\cG_0) \sub Z$ and thus $X(T) \in \cS_\Min(Z)$, completing the proof.
\end{proof}

\begin{remark}
  \label{rem:sofic-eq-proj-isolated-frgr}
It follows from Theorem \ref{th:minimal-sofic-dense} that for a non-abelian free group $G$, the minimal subshifts which are projectively isolated among minimal subshifts are precisely the sofic minimal subshifts. Indeed, let $X \sub A^G$ be projectively isolated among minimal subshifts; then there exists a non-empty open subset $\cU$ of some $\cS_\Min(B^G)$ and a factor map $\pi \colon B^G \to A^G$ such that $\pi[Y]=X$ for every $Y \in \cU$. Since $\cU$ contains a sofic subshift, $X$ is itself sofic.
\end{remark}

\begin{cor}
  \label{c:sofic-minimal-amalgamation}
  Let $G$ be a non-abelian, finitely generated free group. Then the sofic minimal subshifts of $G$ form a Fraïssé class.
\end{cor}
\begin{proof}
 This follows from \autoref{rem:sofic-eq-proj-isolated-frgr} and \autoref{c:proj-isolated-Fraisse}.
\end{proof}

\begin{cor}
  \label{c:minimal-free-comeager}
  Let $G$ be a non-abelian, finitely generated free group. Then the space $\Xi_\Min(G)$ has a comeager conjugacy class elements of which are isomorphic to the Fraïssé limit of the sofic minimal subshifts.
\end{cor}
\begin{proof}
 This follows from \autoref{th:minimal-sofic-dense}, \autoref{p:sofic-minimal-proj-isol}, and \autoref{c:min-actions-comeager-ex}.
\end{proof}

\begin{remark}
  \label{rem:also-true-for-Z}
  \autoref{th:minimal-sofic-dense}, \autoref{c:sofic-minimal-amalgamation}, and \autoref{c:minimal-free-comeager} also hold for $G = \Z$. One just has to notice that in this case, the sofic minimal subshifts are precisely the periodic transitive subshifts. See \autoref{c:min-Z}.
\end{remark}

%%%%%%%%%%%%%%%%%%%%%%%%%%%%%%%%%%%%%%%%%%%%%%%%%%

\section{Measure-preserving actions of the free group}
\label{sec:meas-pres-acti}

We keep the notation of the previous two sections; we allow again the case $G = \Z$.

\subsection{Measured Rauzy graphs}
\label{sec:meas-rauzy-graphs}

\begin{defn}
  \label{df:measured-Rauzy}
  A \df{measured Rauzy graph} is a Rauzy graph $\cG$, equipped with two functions $\mu \colon V(\cG) \to \R^+$ and $m \colon E(\cG) \to \R^+$ (where $\R^+$ is the set of non-negative reals) such that $m(e) = m(\bar e)$ for all $e \in E(\cG)$ and
  \begin{equation}
    \label{eq:measured-Rauzy_defn}
    \sum_{\set{e : \sigma(e) = v, \ell(e) = s}} m(e) = \sum_{\set{e : \rho(e) = v, \ell(e) = s}} m(e) = \mu(v) \quad \text{ for all } v \in V(\cG), s \in S.
  \end{equation}
  We say that a measured Rauzy graph has \df{full support} if $m(e) > 0$ for all $e \in E(\cG)$.
\end{defn}

If an action $\xi \in \Xi(G)$ admits a finite invariant measure $\mu$ and $A$ is a partition of $\Omega$, then the graph $\cG(\xi, A)$ becomes a measured Rauzy graph by defining
\begin{equation}
  \label{eq:measured-Rauzy}
  m(e) = \mu \big(\sigma(e) \cap \ell(e) \cdot \rho(e) \big) \quad \text{ for all } e \in E(\cG(\xi, A)).
\end{equation}
If $\xi$ is minimal, then $\cG(\xi, A)$ is minimal by \autoref{l:Rauzy-xi-min} and $m$ and $\mu$ are strictly positive because in a minimal, measure-preserving system, every non-empty open set has positive measure (as it is syndetic).

\begin{lemma}
  \label{l:integer-measured-Rauzy}
  Let $(\cG, \mu, m)$ be a measured Rauzy graph with full support. Then there exist $\mu'$ and $m'$, which take integer values, such that $(\cG, \mu', m')$ is again a measured Rauzy graph with full support.
\end{lemma}
\begin{proof}
  The equations in \autoref{eq:measured-Rauzy_defn} form a homogeneous system of linear equations with unknowns $\mu(v), m(e)$ for $v \in V(\cG), e \in E(\cG)$. Let $M$ be the matrix of this system and note that it has integer entries. It follows from Gaussian elimination that $\ker M$ has a rational basis, so rational solutions of the system are dense in all solutions. It follows that it has a strictly positive rational solution, which can be converted into an integer one by multiplying it by an appropriate integer.
\end{proof}

\subsection{Density of the periodic flows}
\label{sec:dens-periodic-flows}

\begin{lemma}
  \label{l:finite-actions-dense}
  Let $G$ be a finitely generated free group and let $(\cG, \mu, m)$ be a measured Rauzy graph for $G$ with full support. Then there exists an action $G \actson^\eta W$ with $W$ finite and a partition $A$ of $W$ such that $\cG(\eta, A) \cong \cG$.
\end{lemma}
\begin{proof}
  By \autoref{l:integer-measured-Rauzy}, we may assume that $m$ and $\mu$ take integer values. For an integer $n$, we denote $[n] \coloneqq \set{0, \ldots, n-1}$ and we let
  \begin{equation}
    \label{eq:defn-W}
    W = \bigsqcup_{v \in V(\cG)} \set{v} \times [\mu(v)].
  \end{equation}
  We define $\pi \colon W \to V(\cG)$ by $\pi(v, i) = v$ and our goal is to define a deterministic Rauzy graph with vertex set $W$ such that $\pi$ becomes a surjective morphism. A \df{partial edge relation} is a subset $E \sub W \times W \times S$ such that:
  \begin{enumerate}
  \item \label{i:partial-edge-atm1} for all $w_1 \in W, s \in S$, $|\set{w_2 \in W : (w_1, w_2, s) \in E}| \leq 1$;
  \item \label{i:partial-edge-inv} for all $w_1,w_2 \in W, s \in S$, $(w_1, w_2, s) \in E \implies (w_2, w_1, s^{-1}) \in E$;
  \item \label{i:partial-edge-morph} for all $w_1,w_2 \in W, s \in S$, $(w_1, w_2, s) \in E \implies \pi(w_1) \edge{s} \pi(w_2)$;
  \item \label{i:partial-edge-meas} for all $v_1, v_2 \in V(\cG), s \in S$,
    \begin{equation*}
      |\set{(w_1, w_2, s) \in E : \pi(w_1) = v_1, \pi(w_2) = v_2}| \leq m(v_1 \edge{s} v_2).
    \end{equation*}
  \end{enumerate}
  A partial edge relation is \df{total} if the inequalities in \autoref{i:partial-edge-meas} are equalities. For a total $E$ and $v \in V(\cG), s \in S$, we have that
  \begin{equation*}
    \begin{split}
    |\set{(w_1, w_2, s) \in E : \pi(w_1) = v}| & \leq |\set{w_1 \in W \colon \pi(w_1)=v}| \\
                                            & = \mu(v) \\
                                            &= \sum_{v' \in V(\cG)} m(v \edge{s} v') \\
                                            &= |\set{(w_1, w_2, s) \in E : \pi(w_1) = v}|,
    \end{split}
  \end{equation*}
  so we must have equality also in \autoref{i:partial-edge-atm1}.

  Thus a total edge relation makes $(W, E)$ into a deterministic Rauzy graph by defining $\sigma(w_1, w_2, s) = w_1, \rho(w_1, w_2, s) = w_2$, and $\ell(w_1, w_2, s) = s$. The map $\pi$ is a morphism because of \autoref{i:partial-edge-morph}. It is also surjective because of the equality in \autoref{i:partial-edge-meas} and the fact that $(\cG, \mu, m)$ has full support.

  To finish the proof, it suffices to check that to any partial edge relation which is not total, we can add an edge. Let $E$ be such a relation and let $v_1, v_2 \in V(\cG)$ and $s \in S$ be such that we have strict inequality in \autoref{i:partial-edge-meas}. Then
  \[\sum_{v' \in V(\cG)} m(v_1 \edge{s} v') > |\set{(w_1, w_2, s) \in E : \pi(w_1) = v_1}|\]
  and, using the same computations as above, we conclude that there exists $w_1 \in \pi^{-1}(v_1)$ such that $\set{w_2' \in W : (w_1, w_2', s) \in E} = \emptyset$. Using \autoref{i:partial-edge-inv} and applying the same reasoning to $\set{(w_2,w_1,s^{-1}) : \pi(w_2)=v_2}$, we find $w_2 \in \pi^{-1}(v_2)$ such that $\set{w_1' \in W : (w_1', w_2, s) \in E} = \emptyset$. Then $E \cup \set{(w_1, w_2, s), (w_2, w_1, s^{-1})}$ is still a partial edge relation.
\end{proof}

\begin{lemma}
  \label{l:finite-tr-actions-dense}
  Let $G$ be a finitely generated free group and let $(\cG, \mu, m)$ be a connected, measured Rauzy graph for $G$ with full support. Then there exists a transitive action $G \actson^\eta W$ with $W$ finite and a partition $A$ of $W$ such that $\cG(\eta, A) \cong \cG$.
\end{lemma}
\begin{proof}
  We define $W$ as in \autoref{eq:defn-W} and $\pi \colon W \to V(\cG)$ as above. Let $E$ be a total edge relation on $W$ as constructed in \autoref{l:finite-actions-dense}. Our goal is to modify $E$ to become a connected total edge relation.
  Note that for $E$ to be connected, it suffices that for every $v \in V(\cG)$, the vertices in $\pi^{-1}(v)$ are in the same $E$-connected component. Indeed, assume this holds and let $R$ be the equivalence relation on $W$ generated by $E$ and suppose that $R$ has at least two classes $C_1$ and $C_2$. Then, by assumption, $\pi^{-1}(\pi[C_1]) = C_1$ and $\pi^{-1}(\pi[C_2]) = C_2$, and by surjectivity of $\pi$, there is no edge in $\cG$ between $\pi[C_1]$ and $\pi[C_2]$. This contradicts the connectedness of $\cG$.

  Fix now $t \in S$ and let $\tau \colon W \to W$ be the bijection with graph $\set{(w_1, w_2) : (w_1, w_2, t) \in E}$. Suppose that there is $v \in V(\cG)$ such that $\pi^{-1}(v)$ intersects two distinct $\tau$-orbits, say in the points $w_1$, $w_2$. Define $\tau' \colon W \to W$ by $\tau'(w) = \tau(w)$ for all $w \neq w_1, w_2$, $\tau'(w_1) = \tau(w_2), \tau'(w_2) = \tau(w_1)$. Then $\tau'$ is a bijection which merges the $\tau$-orbits of $w_1$ and $w_2$. Moreover,
  \begin{equation*}
    \big(E \sminus (\tau \times \set{t} \cup \tau^{-1} \times \set{t^{-1}}) \big) \cup \big(\tau' \times \set{t} \cup \tau'^{-1} \times \set{t^{-1}}\big)
  \end{equation*}
  is a total edge relation. By iterating this procedure, we obtain a total edge relation such that for every $v$, all vertices in $\pi^{-1}(v)$ are in the same $t$-cycle. By our observation above, this is enough.
\end{proof}

\begin{theorem}
  \label{th:min-pmp-free-group}
  Let $G$ be a finitely generated free group. Then:
  \begin{enumerate}
  \item \label{i:min-pmp:xi-comeager-conj} $\Xi_\Mpmp(G)$ has a comeager conjugacy class whose elements are isomorphic to the left translation action of $G$ on its profinite completion.
  \item \label{i:min-pmp:subshifts-dense} For every $A$, $\cS_\Pertr(A^G)$ is dense in $\cS_\Mpmp(A^G)$.
  \end{enumerate}
\end{theorem}
\begin{proof}
  \autoref{i:min-pmp:xi-comeager-conj} By \autoref{c:profinite-actions}, it suffices to prove that $\Xi_\Pertr(G)$ is dense in $\Xi_\Mpmp(G)$. Let $\xi \in \Xi_\Mpmp(G)$ and let $\cN(\cG) \cap \Xi_{\Mpmp}(G)$ be a neighborhood of $\xi$, where $\cG$ is a minimal Rauzy graph on some partition $B$ of $\Omega$ (see \autoref{th:minimal-Rauzy}). As $\xi$ admits an invariant measure, we can use \autoref{eq:measured-Rauzy} to convert $\cG$ into a measured Rauzy graph $(\cG, \mu, m)$. As $\xi$ is minimal, $(\cG, \mu, m)$ is connected and has full support. Let $G \actson^\eta W$ and the partition $A$ of $W$ be as given by \autoref{l:finite-tr-actions-dense}. We can realize $\eta$ as a factor of a profinite action $\xi'$ of $G$ on $\Omega$, so that $A$ can be identified with a partition of $\Omega$. Now if $f \in \Homeo(\Omega)$ is such that $f \cdot A = B$, we have that $f \cdot \xi'$ is profinite and belongs to $\cN(\cG)$.

  \autoref{i:min-pmp:subshifts-dense} This follows from \autoref{i:min-pmp:xi-comeager-conj} and the facts that the map $\pi_A$ from $\Xi_\Mpmp(G)$ to $\cS_\Mpmp(A^G)$ defined by \autoref{eq:defn-piA} is continuous and surjective and that for a profinite $\xi$, $\pi_A(\xi)$ is periodic.
\end{proof}

\begin{remark}
  \label{rem:mpmp-convergence-fixed-measure}
  \autoref{l:finite-tr-actions-dense} only requires that the graph $\cG$ be connected, so the proofs above imply the following more precise result: if $\xi$ is a \emph{transitive} action of $G$ on $\Omega$ and $\mu$ is \emph{any} $G$-invariant measure with full support, then there exists a sequence of profinite actions $(\xi_n)_n$ such that $\xi_n \to \xi$ and $\mu_n \to \mu$, where $\mu_n$ denotes the unique $\xi_n$-invariant measure on $\Omega$.

  Moreover, it follows that, for measured Rauzy graphs with full support, the three conditions in \autoref{ex:examples-Rauzy-graphs} are all equivalent to being connected.
\end{remark}

Applying \autoref{th:min-pmp-free-group} to $G = \Z$ and using the fact that $\Xi_\Mpmp(\Z) = \Xi_\Min(\Z)$ and $\cS_\Mpmp(A^\Z) = \cS_\Min(A^\Z)$ (because $\Z$ is amenable), we obtain the following.
\begin{cor}[\cite{Hochman2008}]
  \label{c:min-Z}
  The following hold:
  \begin{enumerate}
  \item $\Xi_\Min(\Z)$ has a comeager conjugacy class whose elements are isomorphic to the universal odometer.
  \item For every $A$, $\cS_\Pertr(A^\Z)$ is dense in $\cS_\Min(A^\Z)$.
  \end{enumerate}
\end{cor}

\begin{remark}
  \label{rem:Piantadosi}
  Periodic subshifts of free groups have also been studied by Piantadosi~\cite{Piantadosi2008}. Among other results, he has a combinatorial characterization of the SFTs containing a periodic point. It is a consequence of the results of this section that SFTs defined by measured Rauzy graphs contain periodic points; it follows that measured Rauzy graphs satisfy the combinatorial condition of \cite{Piantadosi2008}*{Theorem~3.4}. Conversely, any SFT with a periodic point admits an invariant measure, so its corresponding graph can be made measured (but the measure may not have full support). We are grateful to Ville Salo for making us aware of the paper \cite{Piantadosi2008}.
\end{remark}

%%%%%%%%%%%%%%%%%%%%%%%%%%%%%%%%%%%%%%%%%%%%%%%%%%

\section{Non-finitely generated groups}
\label{sec:non-finit-gener}

\subsection{Co-induction}

If $G$ is a group and $H \leq G$, the \df{restriction functor} takes a $G$-flow and produces the $H$-flow which is the restriction of the action to $H$. Co-induction is the right adjoint functor to restriction. It takes an $H$-flow and it produces a $G$-flow whose restriction to $H$ is an extension of the given $H$-flow and which is universal with respect to this property. The construction is similar to the classical notion of induced representation in representation theory.

Co-induction is particularly useful for understanding properties in $\Xi(G)$ when $G$ is not finitely generated: the topology on $\Xi(G)$ is defined by looking at finitely many elements of $G$ at a time, which implies that actions co-induced from finitely generated subgroups are dense (see the proof of \autoref{p:non-fg-gen-n-transitive}).

We briefly recall the construction and some of its basic properties.
Let $H \actson Z$ be an $H$-flow. Then $G$ acts on $Z^G$ via left-shift on the coordinates. Define $\pi \colon Z^G \to Z$ by $\pi(\tilde{z}) = \tilde{z}(1_G)$ and consider the set
\[
  \tilde Z = \{ \tilde{z} \in Z^G : \forall g \in G \ \forall h \in H \quad \tilde{z}(g h)  = h^{-1} \cdot \tilde{z}(g)\},
\]
where the symbol $\cdot$ refers to the action of $H$ on $Z$.
By definition, $\tilde Z$ is a closed, $G$-invariant subset of $Z^G$; furthermore, $\pi \colon \tilde Z \to Z$ is $H$-equivariant. The \emph{co-induced action} of $H \actson Z$ (from $H$ to $G$) is the action $G \actson \tilde Z$.

Let $T$ be a transversal for the left cosets of $H$ in $G$. Then any $\tilde{z} \in \tilde Z$ is uniquely determined by its values on $T$, since for all $h \in H$, we have $\tilde{z}(th) = h^{-1} \cdot \tilde{z}(t)$; and conversely, any element $\tilde{z}$ of $Z^T$ can be extended to an element of $\tilde Z$ by the same formula. In particular, $\tilde Z$ is homeomorphic to $Z^T$.

\begin{lemma}
  \label{l:coind-not-minimal}
  Let $H \leq G$ and let $H \actson Z$ be a flow which is not minimal. Then the co-induced flow $G \actson \tilde{Z}$ is not minimal either.
\end{lemma}
\begin{proof}
  Let $U$ be an $H$-invariant, non-empty, open subset of $Z$ and let $z_0 \in Z \setminus U$. We claim that $G \cdot \pi^{-1}(U)$ is a proper subset of $\tilde{Z}$, where $\pi \colon \tilde{Z} \to Z$ denotes the natural projection. Let $T$ be a transversal for the left $H$-cosets in $G$ and set $\tilde{z}_0(th) = h^{-1} \cdot z_0$ for all $t \in T, h \in H$. Then it is clear that $\tilde{z}_0 \in \tilde{Z} \setminus G \cdot \pi^{-1}(U)$.
\end{proof}

It is straightforward to check that the co-induced action $\tilde X$ of an $H$-subshift $X \in \cS(A^H)$ is a $G$-subshift in the same alphabet. Indeed, in that case $\tilde X$ can be identified with
\[
  \{x \in A^G : \forall g \in G \ \big(h \mapsto x(gh) \big) \in X\}
\]
via the $G$-equivariant, continuous, injective map $\tau \colon \tilde X \to A^G$ defined by $\tau(\tilde{x})(g) = \tilde{x}(g)(1_G)$. In other words, elements of $\tilde X$ are obtained by copying independently elements of $X$ inside each left $H$-coset. It follows that the co-induced action of an SFT is an SFT (with the same forbidden patterns) and the co-induced action of a sofic subshift is sofic (because of functoriality).

\begin{defn}
Let $G$ be a group and $n \in \N$. A $G$-flow $G \actson Z$ is \emph{topologically $n$-transitive} (or simply \emph{$n$-transitive}) if the diagonal action $G \actson Z^n$ is transitive.
\end{defn}

\begin{prop}
  \label{p:coind-n-transitive}
Let $G$ be a group, let $H$ be a subgroup of $G$ of infinite index, and let $H \actson Z$ be an $H$-flow. Then the co-induced action $G \actson \tilde Z$ is $n$-transitive for all $n$.
\end{prop}
\begin{proof}
Fix $n$ and two non-empty open subsets $U$, $V$ of $\tilde Z^n$. Let $T$ be a transversal for the left $H$-cosets in $G$. Identifying $T$ with $G/H$, the transitive action $G \actson G/H$ gives us a transitive action $G \actson T$ (explicitly, $g \cdot t_1=t_2$ if $g t_1H= t_2H$).

We may assume that there exist a finite subset $F$ of $T$ and two finite families of open subsets $U_{i,f}$, $V_{i,f}$, for $i < n$, $f \in F$, of $X$ such that
\begin{align*}
U &= \{\tilde{z} \in \tilde Z^n \colon \forall f \in F \ \forall i < n \ \tilde{x}_i(f) \in U_{i,f} \} \\
V &= \{\tilde{z} \in \tilde Z^n \colon \forall f \in F \ \forall i < n \ \tilde{x}_i(f) \in V_{i,f} \}.
\end{align*}
Pick some $\tilde{y} \in U$ and $\tilde{z} \in V$. Since $G/H$ is infinite, Neumann's lemma gives us $g \in G$ such that $gFH \cap FH = \emptyset$. Then we can define $\tilde{w} \in \tilde Z^n$ by setting $\tilde{w}_i(f) = \tilde{y}_i(f)$ and $\tilde{w}_i(gf) = \tilde{z}_i(f)$ for all $f \in F$, and extending $\tilde{w}$ arbitrarily to the other cosets. By definition, $\tilde{w} \in U$ and $g^{-1} \cdot \tilde{w} \in V$.
\end{proof}

This has the following immediate consequence.
\begin{prop}
  \label{p:non-fg-gen-n-transitive}
Let $G$ be a countable group which is not finitely generated. Then a generic element of $\Xi(G)$ is $n$-transitive for all $n$.
\end{prop}
\begin{proof}
Being $n$-transitive is a $G_\delta$ condition, so we only need to prove that $n$-transitive actions are dense in $\Xi(G)$. Consider a basic open set $U$ in $\Xi(G)$ of the form
\[
  U = \{\xi \in \Xi(G) : \forall f \in F \ \forall a \in \cA \ \xi(f)a= \xi_0(f)a\}
\]
where $\xi_0 \in \Xi(G)$, $F$ is a finite subset of $G$, and $\cA$ is a finite subalgebra of $\cB(\Omega)$.

Let $\eta$ be the restriction of $\xi_0$ to the subgroup $H$ of $G$ generated by $F$, and denote by $\tilde \eta \colon G \actson \tilde \Omega$ the co-induced action of $\eta|_H$. Note that $\tilde \Omega$ is homeomorphic to $\Omega$. Denoting by $\pi \colon \tilde \Omega \to \Omega$ the $H$-equivariant factor map, we may pick a homeomorphism $\phi \colon \tilde \Omega \to \Omega$ such that $\phi \big(\pi^{-1}(\xi_0(f)a) \big)=\xi_0(f)a$ for all $a \in \cA$ and all $f \in F$.

Then for every $f \in F$ and every $a \in \cA$, we have $(\phi \tilde \eta \phi^{-1})(f)a =
\xi_0(f)a$, which shows that $\phi \tilde \eta \phi^{-1}$ belongs to $U$. On the other hand, by \autoref{p:coind-n-transitive}, $\tilde{\eta}$ is $n$-transitive for all $n$ and we are done.
\end{proof}

As pointed out in \cite{Frisch_Kechris_Shinko}*{Corollary~4.4.7}, \autoref{p:non-fg-gen-n-transitive} admits a converse.  Indeed, let $G$ be finitely generated and let $\set{a_1, a_2}$ be a non-degenerate partition of $\Omega$. Then the set of $\xi \in \Xi(G)$ satisfying the condition ``$a_1$ is $\xi$-invariant'' is open, disjoint from the set of transitive actions, and non-empty (because the trivial action belongs to it). Thus we obtain a dynamical characterization of finitely generated groups: a group $G$ is finitely generated iff a generic element of $\Xi(G)$ is not transitive.

\subsection{Density of minimality}
We saw above that the density of $\Xi_\Tr(G)$ in $\Xi(G)$ is characterized by $G$ not being finitely generated. It is natural to ask for which groups $G$, $\Xi_\Min(G)$ is dense in $\Xi_\Tr(G)$ (this problem was brought to our attention by A.~S.~Kechris). For instance, it is easy to see that this is the case when $G$ is a free group on infinitely many generators. While we do not have a general criterion, we provide an answer for the class of non-finitely generated amenable groups (see \autoref{c:character-amen-non-fg}). Before turning to that, we point out the following obstruction.

\begin{prop}
Let $G$ be a countable group such that its center $Z(G)$ contains an element of infinite order. Then $\Xi_\Min(G)$ is not dense in $\Xi(G)$.
\end{prop}
\begin{proof}
We fix a clopen partition $\{a_1,a_2,a_3\}$ of $\Omega$, with each $a_i$ non-empty, as well as a homeomorphism $\phi$ of $\Omega$ such that $\phi(a_1) \cap a_2 \ne \emptyset$, $\phi(a_2)=a_3$, and $\phi(a_3) \subset a_2$.

We first observe that the centralizer $C(\phi) \subset \Homeo(\Omega)$ does not act minimally on $\Omega$. To see this, consider $b = a_1 \cap \phi^{-1}(a_2)$, which is clopen and non-empty. Assume that $C(\phi)$ acts minimally; then there exists a finite $F \subseteq C(\phi)$ such that for all $\omega \in \Omega$ there exists $f \in F$ with $f \cdot \omega \in b$. Now pick some $\omega \in b$; since $F$ is finite, there exists $f \in F$ such that $f \cdot \phi^n(\omega) \in b$ for infinitely many $n \in \N$. Since $f \cdot \phi^n(x) = \phi^n(f \cdot x)$ for all $n$, we see that as soon as $\phi^n(f \cdot x) \in b$, we must have $\phi^m(f \cdot x) \in a_2 \cup a_3$ for all $m \ge n+1$, a contradiction.

Now let $h \in Z(G)$ be of infinite order and consider the action $\langle h \rangle \actson \Omega$ given by $h \cdot \omega = \phi(\omega)$. Using co-induction, we see that the open set $U$ of all actions $\xi \in \Xi(G)$ such that $\xi(h)(a_i)=\phi(a_i)$ for $i = 1,2,3$ is non-empty. For any $\xi \in U$, $\xi[G] \subseteq \Homeo(\Omega)$ is contained in the centralizer of $\xi(h)$, which does not act minimally by the argument given in the previous paragraph. Hence no element of $U$ acts minimally.
\end{proof}

Now we focus on amenable groups and we consider two cases depending on whether the group is locally finite or not. We start with the latter.

\begin{defn}
Let $G$ be a countable group and let $G \actson Z$ be a zero-dimensional flow. We say that $a \in \cB(Z)$ is \emph{shrinkable} if it can be equidecomposed with a proper subset of itself, i.e., if there exist a clopen partition $a_0, \ldots, a_{n-1}$ of $a$ and $g_0, \ldots, g_{n-1} \in G$ such that $\bigsqcup_i g_i \cdot a_i \subsetneq a$.

We say that the flow is \emph{shrinking} if there exists a shrinkable $a \in \cB(\Omega)$.
\end{defn}

We note that being shrinking is an open condition in $\Xi(G)$.

The shrinking condition is an obstruction to the existence of an invariant measure with full support. In particular, we have the following.
\begin{lemma}
  \label{l:not-min-shrinking-amenable}
Assume that $G$ is an amenable group and $G \actson Z$ is a minimal, zero-dimensional flow. Then it is not shrinking.
\end{lemma}
\begin{proof}
Suppose that $a \in \cB(Z)$ is shrinkable, as witnessed by $a_0, \ldots, a_{n-1}$ and $g_0, \ldots, g_{n-1}$. Then $b = a \setminus \bigsqcup g_i \cdot a_i$ is open, non-empty. Let $\mu$ be a $G$-invariant probability measure on $Z$. On the one hand, by invariance, $\mu(a) = \mu(\bigsqcup_i g_i \cdot a_i)$, so $\mu(b) = 0$. On the other, by minimality, finitely many translates of $b$ cover $Z$, so $\mu(Z) = 0$, contradiction.
\end{proof}

We also need the following well-known fact.
\begin{lemma}[K\H{o}nig]
  \label{l:inf-geodesic-path}
  Let $\Gamma$ be a connected, infinite, locally finite graph and let $v_0$ be a vertex of $\Gamma$. Then $\Gamma$ has an infinite ray starting at $v_0$.
\end{lemma}

\begin{prop}
  \label{p:generic_behavior_amenable_non_lf}
Let $G$ be a non-locally finite, countable group. Then the set of shrinking actions of $G$ is an open, dense subset of $\Xi(G)$.
If $G$ is moreover amenable, then $\Xi_\Min(G)$ is not dense in $\Xi(G)$.
\end{prop}
\begin{proof}
  The set of shrinking actions is open and conjugacy-invariant, so by \autoref{p:dense-conj-class-Xi}, we only need to build one shrinking action on a Cantor space. Let $H$ be a finitely generated, infinite subgroup of $G$ and fix a finite, symmetric generating set $S$ of $H$. Consider the left Cayley graph $\Gamma$ for $H$ and $S$: that is, the graph with vertex set $H$ and edges between $h$ and $sh$ for every $h \in H$ and $s \in S$. Then we can apply \autoref{l:inf-geodesic-path} to find an infinite ray $(h_n)_{n \in \N}$ in $\Gamma$.
Let $T = \{h_n : n \in \N\}$ and for $s \in S$, let $T_s = \{h_n \in T : h_{n+1}= s h_n\}$. Then $\bigsqcup_{s \in S} T_s = T$ and $sT_s = \set{h_{n+1} : h_{n+1} = s h_n}$, so $\bigsqcup_{s \in S} s T_s = T \setminus \{h_0\}$.

Now consider the flow $G \actson \beta G$, where $\beta G$ denotes the Stone--\v{C}ech compactification of $G$. It follows from the calculation above that $\cl{T}$, which is a clopen subset of $\beta G$, is shrinkable. Let $\cA$ be a countable, $G$-invariant subalgebra of $\cB(\beta G)$ which contains all $\cl{T_s}$. Then the Stone space of $\cA$ gives us a metrizable, shrinking $G$-flow $Z$, with the same witnesses as in $\beta G$. If $Z$ has isolated points, we can take the product of $Z$ with the trivial action on the Cantor space to obtain an action on $\Omega$ that is still shrinking.

The second claim follows from the first and \autoref{l:not-min-shrinking-amenable}.
\end{proof}

Next we turn to the locally finite case.
We start with a simple observation.
\begin{lemma}
  \label{l:ext-finite-groups}
  Let $K \leq H$ be finite groups and let $K \actson A$ be an action on a finite set with $|A| \leq [H : K]$. Then there exists a free, transitive action $H \actson B$ and a $K$-factor map $B \to A$.
\end{lemma}
\begin{proof}
  Set $B = H$ and let $T$ be a transversal for the right cosets of $K$ in $H$. Let $\pi \colon T \to A$ be an arbitrary surjection and extend it to all of $H$ by $K$-equivariance: define $\pi(kt) = k \cdot \pi(t)$ for all $k \in K, t \in T$.
\end{proof}

For locally finite groups, we also obtain information about the invariant measures for a generic action. In order to state the theorem, we recall that the \emph{clopen value set} of a Borel probability measure $\mu$ on $\Omega$ is the countable set
\[
  \left\{ \mu(a) : a \in \cB(\Omega) \right\} \sub [0, 1].
\]
Following Akin, we say that a probability measure $\mu$ on $\Omega$ is \df{good} if it is atomless, has full support, and is such that for any two $a, b \in \cB(\Omega)$ with $\mu(a) < \mu(b)$, there exists $c \in \cB(\Omega)$ such that $c \sub b$ and $\mu(c) = \mu(a)$. Using results of Akin and Giordano--Putnam--Skau, the clopen value sets of good invariant measures can be used to characterize orbit equivalence (see \autoref{c:generic-OE-loc-finite}).

\begin{theorem}
  \label{th:generic_behavior_locally_finite}
  Let $G$ be a countably infinite, locally finite group. Then a generic element of $\Xi(G)$ is minimal and uniquely ergodic, its unique invariant measure is good, and the clopen value set of the measure is equal to
\begin{equation*}
  V_G = \set[\big]{i/|K| : i \in \{0, \ldots, |K|\}, \  K \text{ is a finite subgroup of } G }.
\end{equation*}
\end{theorem}
\begin{proof}
  We first show that the set of $\xi \in \Xi(G)$ satisfying the condition
  \begin{enumerate}[label=($*$)]
  \item \label{i:condition-loc-fin}
    for any non-degenerate partition $A$ of $\Omega$ and any finite set $F \sub G$, there exist a finite subgroup $H \leq G$ containing $F$ and a non-degenerate partition $B$ refining $A$ such that $B$ is $\xi(H)$-invariant and $\xi|_H$ is free and transitive on $B$
  \end{enumerate}
  is dense $G_\delta$.

  By the Baire category theorem, it suffices to see that the condition \autoref{i:condition-loc-fin} for fixed $A$ and $F$ is dense (it is obviously open). Let $U$ be a non-empty open set in $\Xi(G)$. By refining $A$ if necessary, we may assume that there is a finite subgroup $K \leq G$ such that $A$ is $K$-invariant and $U$ is given by the action of $K$ on $A$. Let $H$ be a finite subgroup of $G$ containing $K$ and $F$ such that $[H : K] \geq |A|$. Apply \autoref{l:ext-finite-groups} to produce a free transitive action $H \actson B$ that factors onto $K \actson A$. Let $B_0$ be a partition of $\Omega$ refining $A$ that realizes the factor map $B \to A$. Let $\eta_0$ be the action of $G$ co-induced from $H \actson B$ on the space $\tilde{B}$. Let $\phi \colon \tilde{B} \to \Omega$ be a homeomorphism sending $B$ to $B_0$ (here we view $B$ as a partition of $\tilde{B}$ using the canonical factor map $\tilde{B} \to B$). Then $\phi \eta \phi^{-1}$ belongs to $U$ and satisfies \autoref{i:condition-loc-fin} for $A$ and $F$.

  We have established that set of $\xi \in \Xi(G)$ satisfying \autoref{i:condition-loc-fin} is dense $G_\delta$. Let now $\xi$ satisfy \autoref{i:condition-loc-fin} and let $\mu$ be any $\xi$-invariant measure. Let $a \in \cB(\Omega)$ be arbitrary and apply \autoref{i:condition-loc-fin} to the partition $\set{a, \Omega \setminus a}$  to obtain $H$ and $B$. As the action $H \actson B$ is transitive, we get that $H \cdot a = \Omega$, which proves the minimality of $\xi$. As this action is moreover free and $\mu$ is $H$-invariant, we must have $\mu(b) = 1/|H|$ for all $b \in B$. Since $a$ is a union of elements of $B$, this determines $\mu(a)$ uniquely and $\mu(a) \in V_G$. To see that every element of $V_G$ is realized, let $K$ and $i$ be given and apply \autoref{i:condition-loc-fin} to the trivial partition and $K$ to obtain $H \geq K$ and $B$. Then the union of $i|H|/|K|$ elements of $B$ has measure $i/|K|$.

  Finally, to check that $\mu$ is good, let $a, b \in \cB(\Omega)$ with $\mu(a) < \mu(b)$. Let $B$ and $H$ be given by \autoref{i:condition-loc-fin} such that $a$ and $b$ can be written as unions of elements of $B$. Now we can take $c$ to be the union of $|H| \mu(a)$ elements of $B$ contained in $b$.
\end{proof}

Akin~\cite{Akin2005} has shown that if $\mu_1$ and $\mu_2$ are two good measures with the same clopen value sets, then there is $f \in \Homeo(\Omega)$ with $f_* \mu_1 = \mu_2$. By the work of Giordano--Putnam--Skau, this is connected with the notion of orbit equivalence. Recall that two actions $G \actson \Omega$, $H \actson \Omega$ are \df{orbit equivalent} if there exists  $f \in \Homeo(\Omega)$ such that $f[G \cdot \omega] = [H \cdot f(\omega)]$ for all $\omega \in \Omega$.
\begin{cor}
  \label{c:generic-OE-loc-finite}
  Let $G$ be a countable, locally finite group. Then a generic pair of actions in $\Xi(G)$ are orbit equivalent. More generally, if $G$ and $H$ are locally finite, a generic action of $G$ is orbit equivalent to a generic action of $H$ iff $V_G = V_H$.
\end{cor}
\begin{proof}
  We prove the more general statement.
  One direction follows directly from \autoref{th:generic_behavior_locally_finite}. For the other, let $\xi$ and $\eta$ be actions of $G$ and $H$, respectively, satisfying condition \autoref{i:condition-loc-fin} in the proof of \autoref{th:generic_behavior_locally_finite}. In particular, they are free. To show that they are orbit equivalent, we will employ \cite{Giordano1995}*{Theorem~2.3}. Using Akin's result mentioned above and the conclusion of \autoref{th:generic_behavior_locally_finite}, we only need to check that the orbit equivalence relations of $\xi$ and $\eta$ are generated by AF actions in the terminology of \cite{Giordano1995}. We will do this for $\xi$, the argument for $\eta$ being analogous.

  Let $\oset{\xi}$ be the \df{topological full group} of $\xi$, i.e., the group of all homeomorphisms $f$ of $\Omega$ for which there exists a clopen partition $A$ of $\Omega$ such that for all $a \in A$, there is $g_a \in G$ with $f|_a = \xi(g_a)|_a$. It is clear that $\oset{\xi}$ generates the same orbit equivalence relation as $\xi$. As $\xi$ is free, the set $\set{\omega \in \Omega : f(\omega) = \omega}$ is clopen for every $f \in \oset{\xi}$. The only remaining condition to check is that $\oset{\xi}$ is locally finite. To simplify notation, we will identify $G$ with its copy in $\oset{\xi}$. Let $F \sub \oset{\xi}$ be finite and let $L = \gen{F}$. Let $A$ be a clopen partition of $\Omega$ and $K \leq G$ be finite such that for every $f \in F$ and $a \in A$, there exists $g \in K$ with $f|_a = g|_a$. Let $\cA$ be a finite $K$-invariant subalgebra of $\cB(\Omega)$ containing $A$ and note that $\cA$ is $L$-invariant. By \autoref{i:condition-loc-fin}, we may further assume that the action of $K$ on the atoms of $\cA$ is free. We claim that the map $L \to \Aut(\cA), f \mapsto f|_\cA$ is injective, which will complete the proof. Indeed, let $f$ belong to the kernel of this map and let $a$ be an atom of $\cA$. By the construction of $\cA$, there exists $g_a \in K$ such that $f|_a = g_a|_a$. In particular, $g_a(a) = f(a) = a$, and by the freeness of the action of $K$, we must have that $g_a = 1_G$, i.e., $f|_a = \id|_a$. As this holds for all atoms $a$, we are done.
\end{proof}

In view of \autoref{c:generic-OE-loc-finite}, it is natural to ask whether there is a comeager conjugacy class in $\Xi(G)$ for $G$ locally finite. This is not the case by \autoref{c:no_comeager_class_amenable_non_fg}.

Combining \autoref{th:generic_behavior_locally_finite} with \autoref{p:generic_behavior_amenable_non_lf} we obtain the following characterization.
\begin{cor}
  \label{c:character-amen-non-fg}
  Let $G$ be a countably infinite, amenable group. Then the following are equivalent:
  \begin{itemize}
  \item $G$ is locally finite;
  \item $\Xi_\Min(G)$ is dense in $\Xi(G)$.
  \end{itemize}
\end{cor}

\subsection{Comeager conjugacy classes in $\Xi(G)$}
It was asked in \cite{Doucha2022p} whether there exists a non-finitely generated group $G$ such that $\Xi(G)$ has a comeager conjugacy class. In the end of this subsection, we provide an answer for amenable groups. We know from Corollary \ref{c:all-actions-comeager-ex} that this is equivalent to the density of projectively isolated subshifts in $\cS(A^G)$ for every finite $A$.

\begin{theorem}
  \label{th:projectivly_isolated_iff_minimal_sofic}
Let $G$ be a countable group which is not finitely generated, let $A$ be a finite alphabet, and let $X \in \cS(A^G)$ be a subshift. Then $X$ is projectively isolated iff it is minimal and sofic.
\end{theorem}
\begin{proof}
  One implication is obvious, since minimal sofic subshifts are projectively isolated for any group. Similarly, a projectively isolated subshift is always sofic. So we only need to prove that a projectively isolated subshift in $\cS(A^G)$ is minimal.

  If $X \in \cS(A^G)$ and $F \sub G$ is finite, we denote by $X_F$ the collection of $F$-patterns that occur in $X$. For $H \leq G$, denote by $X|_H$ the image of the restriction map $A^G \to A^H$ and note that $X|_H \in \cS(A^H)$.

  Let $X \in \cS(A^G)$ be projectively isolated.
  Then there exist a finite set $F \sub G$, a finite alphabet $B$, a map $\pi \colon B \to A$ giving rise to a map $\Pi \colon B^G \to A^G$, and an SFT $Y \subseteq B^G$ such that $\Pi[Z] = X$ for any $Z \in \cS(B^G)$ with $Z_F = Y_F$. Suppose, towards a contradiction, that $X$ is not minimal. By enlarging $F$ if necessary, we may assume that there is an $F$-pattern $p_0$ occurring in $X$, and $x_0 \in X$ such that $p_0$ does not occur in $x_0$.

Let
\begin{equation*}
  Q = \set{q \in B^F : \pi \circ q = p_0}
\end{equation*}
and let $y_0 \in Y$ be such that $\Pi(y_0) = x_0$. No element of $Q$ occurs in $y_0$.

Let $H = \gen{F}$ and note that $Y$ is the subshift co-induced of $Y|_H$, i.e., elements of $Y$ are obtained by copying independently elements of $Y|_H$ on each left $H$-coset. As $\Pi$ is defined using a map on the alphabets, the same is also true for $X$.

 Let $Z \sub Y$ be the subshift consisting of all $y$ such that patterns in $Q$ occur in $y$ in at most one left $H$-coset. Explicitly,
\begin{equation*}
Z = \set[\big]{y \in Y : \forall g_1, g_2 \ \big( (g_1^{-1} \cdot y)|_F, (g_2^{-1} \cdot y)|_F \in Q \implies g_1H = g_2H \big)}.
\end{equation*}
For each $p \in Y_F$, pick some $y_p \in Y|_H$ such that ${y_p}|_F = p$. Let $T$ be a transversal for the left cosets of $H$ in $G$ with $1_G \in T$.
For each $p \in P$, define $z_p \in Y$ by setting, for $t \in T$ and $h \in H$,
\[
  z_p(th) =
  \begin{cases}
    y_p(h) & \text{ if } t = 1_G, \\
    y_0(h) & \text{ otherwise.}
  \end{cases}
\]
Since patterns in $Q$ do not occur in $y_0$, $z_p$ belongs to $Z$. Thus $Z$ realizes all the patterns in $Y_F$, whence $Z_F = Y_F$.

On the other hand, $\Pi[Z]$ is a proper subshift of $X$: indeed, let $x_1 \in X|_H$ be a point in which $p_0$ occurs and copy $x_1$ on every $H$-coset to obtain an element of $X$ which is not in $\Pi[Z]$ (note that as $H$ is finitely generated and $G$ is not, we must have $H \lneq G$). This contradicts our choice of neighborhood that projectively isolates $X$.
\end{proof}

\begin{cor}
  \label{c:a_generic_must_be_minimal}
Let $G$ be a countable group which is not finitely generated. If there exists a comeager conjugacy class in $\Xi(G)$, then the generic element of $\Xi(G)$ is minimal.
\end{cor}
\begin{proof}
A generic element of $\Xi(G)$, if it exists, is obtained as an inverse limit of projectively isolated subshifts (see \autoref{th:proj-atomic} and \autoref{sec:Fraisse-cantor-actions} or \cite{Doucha2022p}). We just proved that if $G$ is not finitely generated then projectively isolated subshifts must be minimal. An inverse limit of minimal flows is minimal.
\end{proof}

\begin{prop}
  \label{p:no_minimal_sofic_locally_finite_group}
Let $G$ be an infinite, locally finite group. Then the only minimal sofic $G$-subshifts are singletons.
\end{prop}
\begin{proof}
  Let $X$ be a minimal sofic subshift. As in the proof of \autoref{th:projectivly_isolated_iff_minimal_sofic} and adopting the notation thereof, we see that there exists a finitely generated (so, finite) subgroup $H \leq G$ such that $X$ is  co-induced of $X|_H$. Furthermore for any finite $K$ such that $H \leq K \leq G$, $X$ is co-induced of $X|_K$ (from $K$ to $G$) and $X|_K$ is co-induced of $X|_H$ (from $H$ to $K$). It follows from \autoref{l:coind-not-minimal} that $X|_K$ is minimal, which, as $K$ is finite, implies that $\abs[\big]{X|_K} \leq |K|$. We have that
  \begin{equation*}
    |K| \geq \abs[\big]{X|_K} = \abs[\big]{(X|_H)^{[K : H]} } = \abs[\big]{X|_H}^{\abs{K} / \abs{H}},
  \end{equation*}
  which, applied for big enough $K$, implies that $X|_H$, and therefore, $X$ is a singleton.
\end{proof}

\begin{cor}
  \label{c:no_comeager_class_amenable_non_fg}
Let $G$ be a countable amenable group which is not finitely generated. Then there is no comeager conjugacy class in $\Xi(G)$.
\end{cor}
\begin{proof}
If $G$ is not locally finite, then $\Xi_\Min(G)$ is not dense in $\Xi(G)$ by \autoref{p:generic_behavior_amenable_non_lf}, so \autoref{c:a_generic_must_be_minimal} applies and there is no comeager conjugacy class in $\Xi(G)$.

If $G$ is locally finite, then by \autoref{th:projectivly_isolated_iff_minimal_sofic} and \autoref{p:no_minimal_sofic_locally_finite_group}, projectively isolated subshifts are not dense, whence there is no comeager conjugacy class in $\Xi(G)$ by \autoref{c:all-actions-comeager-ex}.
\end{proof}

\begin{question}
  \label{q:sofic-min-exist}
  Does there exist a non-finitely generated group which admits:
  \begin{enumerate}
  \item non-trivial, minimal SFTs?
  \item non-trivial, minimal, sofic subshifts?
  \end{enumerate}
\end{question}

%%%%%%%%%%%%%%%%%%%%%%%%%%%%%%%%%%%%%%%%%%%%%%%%%%

\section{Transitive actions}
\label{sec:transitive-actions}

\subsection{Some general criteria}
\label{sec:some-gener-crit}

We turn to the study of the space $\Xi_\Tr(G)$ of transitive actions of a countable group $G$. It is claimed by Hochman in \cite{Hochman2012} that there is a dense conjugacy class in $\Xi_\Tr(G)$ for $G = \Z^d$ (see Proposition~4.4 of \cite{Hochman2012}). However, the argument given there is incorrect, and the statement is not true for $d \geq 2$ (see \autoref{th:polycyclic-no-dense-conj-class}). Hochman's approach, which is general and does not depend on particular properties of $\Z^d$, is as follows. Let $\set{\xi_n : n \in \N}$ be a dense set in $\Xi_\Tr(G)$ and let $x_n$ be a transitive point for $\xi_n$. Then the closure $Z$ of the orbit of $(x_0, x_1, \dots)$ in the product $\prod_n \xi_n$ is a transitive system whose conjugates can approximate all of the $\xi_n$. While this is indeed true, the problem is that the point $(x_0, x_1, \dots)$ may be isolated in $Z$ and even worse, sometimes $Z$ cannot even be realized as a factor of a flow on a perfect space. We are thus led to characterize the flows which can be realized in such a way.

If $G$ is a group, $G \actson Z$ is a $G$-flow, $z \in Z$, and $U \sub Z$ is open, we denote
\begin{equation}
  \label{eq:defn-return-time}
  \Ret(z, U) \coloneqq \set{g \in G : g \cdot z \in U}.
\end{equation}
We say that a point $z \in Z$ is \df{recurrent} if $\Ret(z, U)$ is infinite for every open $U \ni z$. We note that if $G$ is infinite, then every flow contains recurrent points because every point belonging to a minimal subflow is recurrent.
A basic observation is that a transitive point which is not isolated is necessarily recurrent (because all of its neighborhoods have infinitely many disjoint, non-empty, open subsets). This leads to the following criterion.

\begin{prop}
  \label{p:criterion_factor_transitive_Cantor}
  Let $G$ be a countable group, let $G \actson Z$ be a zero-dimensional, metrizable $G$-flow, and let $z_0 \in Z$ be a transitive point. Then the following are equivalent:
  \begin{enumerate}
  \item The flow $Z$ is a factor of a transitive flow on a Cantor space;
  \item The point $z_0$ is recurrent;
  \item The point $z_0$ is not isolated or it has infinite stabilizer.
  \end{enumerate}
\end{prop}
\begin{proof}
  \begin{cycprf}
  \item[\impnnext] Suppose that $\pi \colon \Omega \to Z$ is a factor map for some transitive flow $G \actson \Omega$. If $z_0$ is not isolated, we are done by the observation above. Suppose now that $z_0$ is isolated. Let $U_0 = \pi^{-1}(\set{z_0})$. Then $U_0$ is open in $\Omega$ and therefore contains a transitive point $\omega_0$. Then $\pi(\omega_0) = z_0$ and $\Ret(\omega_0, U_0)$ is infinite and is contained in the stabilizer $G_{z_0}$.

  \item[\impfirst] Assume that $z_0$ is isolated with infinite stabilizer $H$ (otherwise, there is nothing to prove). The shift action $H \actson 2^G$ is transitive; let $x_0 \in 2^G$ be a point with dense orbit. Let $Y=\overline{G \cdot(x_0,z_0)} \sub 2^G \times Z$. We prove that $Y$ is perfect, which amounts to showing that $(x_0,z_0)$ is not isolated in $Y$. Since $2^G$ is perfect and $x_0$ has a dense orbit, there exists a sequence $(h_n)_{n \in \N}$ of elements of $H$ such that $h_n \cdot x_0 \ne x_0$ for all $n$, and $h_n \cdot x_0$ converges to $x_0$. Then $h_n \cdot (x_0,z_0) = (h_n \cdot x_0, z_0)$ converges to $(x_0, z_0)$, and we are done.
  \end{cycprf}
\end{proof}

\begin{theorem}
  \label{th:dense-conj-transitive}
  Let $G$ be an infinite countable group. Then the following are equivalent:
  \begin{enumerate}
  \item \label{i:p:dct:1} $\Xi_\Tr(G)$ has a dense conjugacy class;
  \item \label{i:p:dct:2} For all SFTs $X_1 \in \cS(A_1^G)$, $X_2 \in \cS(A_2^G)$, and all patterns $p_1, p_2$, if there exist recurrent points $x_i \in X_i$ such that $p_i$ occurs in $x_i$ for $i = 1,2$, then there exist $x_i' \in X_i$ such that $p_i$ occurs in $x'_i$ for $i = 1,2$, and $(x_1', x_2')$ is recurrent in $X_1 \times X_2$.
  \end{enumerate}
\end{theorem}
\begin{proof}
  \begin{cycprf}
  \item[\impnext]
    Let $\xi_0 \in \Xi_\Tr(G)$ have a dense conjugacy class. Let $X_i, p_i, x_i$ for $i = 1,2$ be given. Identify $A_1$ and $A_2$ with some arbitrary non-degenerate partitions of $\Omega$. Let
    \begin{equation}
      \label{eq:p:dct-1}
        \cU_i = \set{\xi \in \Xi_\Tr(G) : \pi_{A_i}(\xi) \sub X_i, p_i \text{ occurs in } \pi_{A_i}(\xi) }, \quad i = 1,2,
    \end{equation}
    where the maps $\pi_{A_i}$ are defined as in \autoref{eq:defn-piA}, and note that $\cU_1$ and $\cU_2$ are open. They are also non-empty by the assumptions and \autoref{p:criterion_factor_transitive_Cantor}. Then there exist $\phi_i \in \Homeo(\Omega)$ such that $\phi_i \cdot \xi_0 \in \cU_i$. If $\omega_0$ is any transitive point for $\xi_0$, we can take $x_i' = \Pi^{A_i}_{\phi_i \cdot \xi_0} \big(\phi_i(\omega_0) \big)$, where the maps $\Pi$ are defined by \autoref{eq:defin-PiAxi}.

  \item[\impfirst]
    Let $\cU_1$ and $\cU_2$ be two non-empty open subsets of $\Xi_\Tr(G)$, which, we may assume are given by \autoref{eq:p:dct-1}. Let $\xi_i \in \cU_i$ and let $\omega_i$ be a transitive point for $\xi_i$. Then it follows from \autoref{p:criterion_factor_transitive_Cantor} that $\Pi^{A_i}_\xi(\omega_i)$ witnesses that the hypotheses of \autoref{i:p:dct:2} are satisfied. Let $x_1', x_2'$ be as in the conclusion of \autoref{i:p:dct:2}. By \autoref{p:criterion_factor_transitive_Cantor}, there exists an action $\xi \in \Xi_\Tr(G)$ which factors onto $\cl{G \cdot (x_1', x_2')}$. Then $\xi$ has conjugates in both $\cU_1$ and $\cU_2$.
  \end{cycprf}
\end{proof}

Next we study how this property behaves with respect to finite index subgroups.
\begin{lemma}
  \label{l:recurrent-finite-index}
  Let $G \actson Z$ be a $G$-flow and let $H \leq G$ be a subgroup of finite index. If the point $z \in Z$ is recurrent for $G$, then it is also recurrent for $H$.
\end{lemma}
\begin{proof}
  Without loss of generality, we may assume that $Z = \cl{G \cdot z}$ and apply \autoref{p:criterion_factor_transitive_Cantor}. If $z$ is isolated, then the stabilizer $G_z$ is infinite, and $H_z = G_z \cap H$ is also infinite. Suppose now that $z$ is not isolated. Then $Z$ is perfect. Let $T$ be a set of representatives for the left cosets of $H$. We have that $Z = \cl{TH \cdot z} = \bigcup_{t \in T} t \cl{H \cdot z}$, so $\cl{H \cdot z}$ has non-empty interior. Therefore $\cl{H \cdot z}$ contains a perfect open subset and $z$ cannot be isolated in $\cl{H \cdot z}$. It is therefore $H$-recurrent.
\end{proof}

\begin{prop}
  \label{p:dense-conj-class-finite-index}
  Let $G$ be an infinite countable group and let $H \leq G$ be a subgroup of finite index. If $\Xi_\Tr(H)$ has a dense conjugacy class, then so does $\Xi_\Tr(G)$.
\end{prop}
\begin{proof}
  We use \autoref{th:dense-conj-transitive}. Let $X_i \in \cS(A_i^G)$, $p_i$, and $x_i$ be as in \autoref{th:dense-conj-transitive}~\autoref{i:p:dct:2}. Let $T$ be a set of representatives for the right cosets of $H$ in $G$ and consider the map $\Phi \colon A_i^G \to (A_i^T)^H$ given by
  \begin{equation*}
    \Phi(x)(h)(t) = x(ht).
  \end{equation*}
  Note that $\Phi$ is an isomorphism of $H$-flows and it is not difficult to see that $\Phi[X_i]$ is an SFT as an $H$-shift (see, e.g., \cite{Bitar2024}*{Lemma~1.5.11}). It follows from \autoref{l:recurrent-finite-index} that $x_i$ is recurrent for $H$. Thus there exist $x_1', x_2'$ such that $p_i$ occurs in $x_i'$ for $i=1,2$ and $(x_1', x_2')$ is recurrent for $H$ and therefore also for $G$.
\end{proof}
We do not know whether \autoref{p:dense-conj-class-finite-index} works in the other direction: that is, whether the property of having a dense conjugacy class in $\Xi_\Tr(G)$ goes down to subgroups of finite index.

\subsection{The special symbol property and polycyclic groups}
\label{sec:spec-symb-prop}

\begin{defn}
  \label{df:special-symbol}
  Let $G$ be a group and let $H \leq G$. We denote by $\chi_H \in 2^G$ the characteristic function of $H$ and we say that $G$ has the \df{special symbol property relative to $H$} if the subshift $\cl{G \cdot \chi_H} \in \cS(2^G)$ is sofic. We say that $G$ has the \df{special symbol property} if $G$ has the special symbol property relative to the trivial subgroup.
\end{defn}

We note that the point $\chi_H$ is isolated in $\cl{G \cdot \chi_H}$ because it is the only point in its orbit which belongs to the open set $\set{x \in 2^G : x(1_G) = 1}$. Its stabilizer is equal to $H$ and therefore it is recurrent iff $H$ is infinite.

The special symbol property (in its non-relative version) was introduced by Dahmani and Yaman in \cite{Dahmani2008}; the equivalent formulation that we use here comes from \cite{Aubrun2017}. It was proved in \cite{Dahmani2008} that hyperbolic groups have the special symbol property and that this property is stable under group extensions. All groups that have the special symbol property are finitely generated.

The following gives a simple criterion for the non-existence of a dense conjugacy class in $\Xi_\Tr(G)$, which applies, for example, to hyperbolic groups (see \autoref{th:hyperbolic-groups}).
\begin{prop}
  \label{p:no-dense-class-2-subgroups}
  Let $G$ be a countable group and let $H_1, H_2$ be two infinite subgroups of $G$ such that:
  \begin{itemize}
  \item $G$ has the special symbol property relative to $H_1$ and to $H_2$;
  \item for every $g \in G$, $g H_1 g^{-1} \cap H_2$ is finite.
  \end{itemize}
  Then there is no dense conjugacy class in $\Xi_\Tr(G)$.
\end{prop}
\begin{proof}
  Let $y_i = \chi_{H_i}$ for $i = 1,2$. We verify that \autoref{th:dense-conj-transitive}~\autoref{i:p:dct:2} fails. We let $\pi_i \colon X_i \to \cl{G \cdot y_i}$ witness the soficity of $\cl{G \cdot y_i}$, where $X_i$ is an SFT. Consider the set $\pi_i^{-1}(\set{y_i}) \sub X_i$. It is closed and $H_i$-invariant and, as $H_i$ is infinite, it must contain a recurrent point $x_i$. Let $p_i$ be a pattern which occurs in $x_i$ and such that if $p_i$ occurs in $x_i'$, then $\pi_i(x_i') \in G \cdot y_i$, and in particular, $\pi_i(x_i')$ is isolated (here we use that $y_i$ is isolated in $\cl{G \cdot y_i}$). Now $X_i, x_i, p_i$ witness the hypothesis of \autoref{th:dense-conj-transitive}~\autoref{i:p:dct:2}. On the other hand, if $p_i$ occurs in $x_i'$ for $i = 1,2$, then $\Ret \big(x_i', \pi_i^{-1}(\set{\pi_i(x_i')}) \big)$ is equal to the stabilizer of $\pi_i(x_i')$, which is a conjugate of $H_i$. As the intersection of a conjugate of $H_1$ and a conjugate of $H_2$ is always finite by hypothesis, we conclude that $(x_1', x_2')$ cannot be recurrent.
\end{proof}

Next we give a different sufficient condition for the non-existence of a dense conjugacy class in $\Xi_\Tr(G)$ that applies to some group extensions.

\begin{lemma}
  \label{l:ssp-relative-quotient}
  Let $G$ be a group and let $H \unlhd G$ be finitely generated. If $G/H$ has the special symbol property, then $G$ has the special symbol property relative to $H$.
\end{lemma}
\begin{proof}
  Let $X \sub A^{G/H}$ be an SFT that witnesses the soficity of $\cl{(G/H) \cdot \chi_{\set{1_{G/H}}}}$ via a map $A^{G/H} \to 2^{G/H}$ induced from a map $\pi \colon A \to \set{0, 1}$. Define the embedding $\iota \colon A^{G/H} \to A^G$ by $\iota(x)(g) = x(gH)$ and let $X' = \iota[X]$. It is clear that the map $A^G \to 2^G$ induced by $\pi$ sends $X'$ onto $\cl{G \cdot \chi_H}$, so we only have to check that $X'$ is an SFT. This follows from the fact that $H$ is finitely generated: we just have to copy the rules for $X$ (using arbitrary representatives of the cosets) and add the rules that $x(s) = x(1_G)$ for all and all $s$ from a finite generating set for $H$, thus ensuring that all $x \in X$ are constant on $H$-cosets.
\end{proof}

The following fact is a relativized version of \cite{Dahmani2008}*{Proposition~4.2}.
\begin{prop}
  \label{p:relative_ssp_stable_under_extension}
Assume that $K \le H \unlhd G$ are finitely generated groups such that $H$ has the special symbol property relative to $K$ and the quotient $G/H$ has the special symbol property. Then $G$ has the special symbol property relative to $K$.
\end{prop}
\begin{proof}
Let $X_1 \sub A_1^{H}$ be an SFT that witnesses the soficity of $\cl{H \cdot \chi_K}$ via a map $\Phi_1 \colon A_1^{H} \to 2^{H}$ induced from a map $\phi_1 \colon A_1 \to \set{0, 1}$. Let $E_1$ be a finite defining window for $X_1$; we may assume that $E_1$ generates $H$ and contains $1_G$.

Similarly, let $X_2 \sub A_2^{G/H}$ be an SFT witnessing the soficity of $\cl{(G/H) \cdot \chi_{\set{1_{G/H}}}}$ via a map $\Phi_2 \colon A_2^{G/H} \to 2^{G/H}$ induced from a map $\phi_2 \colon A_2 \to \set{0, 1}$. Let $E_2$ be a finite defining window for $X_2$ which generates $G/H$ and contains $1_{G/H} = H$. We pick a map $r \colon E_2 \to G$ such that $r(H)=1_G$ and for every $e \in E_2$, one has $e = r(e)H$.

We set $B = A_1 \times A_2$ and denote by $\alpha_1$, $\alpha_2$ the coordinate projections.
Consider the SFT $Z \sub B^G$, with defining window $\set{gh  : g \in r[E_2], \, h \in E_1}$ for which a pattern $p$ is allowed iff it satisfies the following conditions:
\begin{itemize}
    \item For every $g \in r[E_2]$, the pattern $h \mapsto \alpha_1 \circ p(gh)$ (with support equal to $E_1$) is an allowed pattern for $X_1$.
    This ensures that for any $z \in Z$ and any $g \in G$ the map $h \mapsto \alpha_1 \circ z(gh)$ belongs to $X_1$.

    \item For every $h \in E_1$, $\alpha_2 \circ p(h)= \alpha_2 \circ p (1_G)$.
This implies that for any $z \in Z$ the map $gH \mapsto \alpha_2 \circ z(g)$ is well-defined.

    \item The pattern $f \mapsto \alpha_2 \circ p(r(f))$ (with support $E_2$) is an allowed pattern for $X_2$.
Along with the previous condition, this guarantees that for all $z \in Z$ and all $g \in G$ the map $gH \mapsto \alpha_2 \circ z(g)$ is an element of $X_2$.
 \end{itemize}

We define $\phi \colon B \to \set{0,1}$ by $\phi(a_1,a_2)=\phi_1(a_1) \phi_2(a_2)$ and denote by $\Phi \colon B^G \to 2^G$ the associated factor map. We now prove that $\Phi[Z] = \cl{G \cdot \chi_K}$.

Choose $x_1 \in X_1$ such that $\Phi_1(x_1)= \chi_K$ and $x_2 \in X_2$ such that $\Phi_2(x_2)= \chi_{1_{G/H}}$. Let $S \sub G$ be a transversal for the $H$-cosets of $G$ which contains $1_G$ and define $z \in B^G$ by setting
\[
  \alpha_1 \circ z(s h) = x_1(h) \And \alpha_2 \circ z(sh)=x_2(sH) \quad \text{ for all } s \in S, h \in H.
\]
It is straightforward to check that $z \in Z$ and $\Phi(z)= \chi_K$.

Finally, we prove that $\Phi[Z] \subseteq \cl{G \cdot \chi_K}$. Pick $z \in Z$. Since $gH \mapsto \alpha_2 \circ z(gH)$ belongs to $X_2$, there is at most one $H$-coset on which $\Phi(z)$ is non-zero. If $\Phi(z)$ is the constant $\bZero$, necessarily $\bZero$ must belong to $\cl{H \cdot \chi_K}$ in $2^H$ or $\bZero$ belongs to $\cl{(G/H) \cdot \chi_{\set{1_{G/H}}}}$, hence $K$ has infinite index in $G$ and $\bZero$ belongs to $\cl{G \cdot \chi_K}$ in $2^G$.
We are left with the case where there exists a unique coset $gH$ on which $\Phi(z)$ is not identically $0$. Since $h \mapsto \alpha_1 \circ z(gh) \in X_1$, $h \mapsto \Phi(g^{-1} \cdot z)(h)$ belongs to the closure of $H \cdot \chi_K$ in $2^H$. Since $g^{-1} \cdot \Phi(z)$ is equal to $0$ on all $H$-cosets other than $H$, it follows that $g^{-1} \cdot  \Phi(z)$ belongs to the closure of $H \cdot \chi_K$ in $2^G$. Hence $\Phi(z) \in \cl{G \cdot \chi_K}$.
\end{proof}

\begin{prop}
  \label{p:suff-cond-nodense-special-symbol}
  Let $G$ be a group and let $H \leq G$ satisfy the following conditions:
  \begin{itemize}
  \item $H$ is normal, finitely generated, infinite, and of infinite index;
  \item $H$ and $G/H$ have the special symbol property.
  \end{itemize}
  Then there is no dense conjugacy class in $\Xi_\Tr(G)$.
\end{prop}
\begin{proof}
  We check the failure of \autoref{th:dense-conj-transitive}~\autoref{i:p:dct:2}. First, using \autoref{l:ssp-relative-quotient}, as in the proof of \autoref{p:no-dense-class-2-subgroups}, we find an SFT $X_1 \in \cS(A_1^G)$, a recurrent point $x_1 \in X_1$ and a pattern $p_1$ which occurs in $x_1$ such that for all $x_1' \in X_1$ in which $p_1$ occurs, there is an open set $U_1$ such that $\Ret(x_1', U_1)$ is a subset of a conjugate of $H$. As $H$ is normal, we can replace this by simply
  \begin{equation}
    \label{eq:p:scnss:1}
    \Ret(x_1', U_1) \sub H.
  \end{equation}

  Next let $Y = \cl{H \cdot \chi_{\set{1_H}}}$ and let $X_2$ be an SFT witnessing the soficity of $Y$. Let $\tilde Y$ and $\tilde X_2$ be the corresponding subshifts co-induced to $G$ and recall that $\tilde X_2$ is an SFT. The subshift $\tilde Y$ has a simple description: it contains all $\tilde y \in 2^G$ which have at most one occurrence of the symbol $1$ in each left $H$-coset. Let $\tilde{\pi} \colon \tilde X_2 \to \tilde Y$ be the map co-induced from $X_2 \to Y$. Let $p_2$ be a pattern occurring in $\tilde X_2$ such that if $p_2$ occurs in $\tilde x \in \tilde X_2$, then $1$ occurs in $\tilde{\pi}(\tilde x)$. By the remarks before \autoref{l:coind-not-minimal} and \autoref{p:coind-n-transitive}, $\tilde X_2$ is perfect and transitive, so recurrent points are dense, and in particular, there exists a recurrent point $\tilde x_2 \in \tilde X_2$ in which $p_2$ occurs. Now let $\tilde x_2'$ be any point in $\tilde X_2$ in which $p_2$ occurs. Then there exists $g \in G$ such that $\tilde{\pi}(\tilde x_2')(g) = 1$. Let $U_2 = \tilde{\pi}^{-1} \big(\set{\tilde y \in \tilde Y : \tilde y(g) = 1} \big)$. We claim that
  \begin{equation}
    \label{eq:p:scnss:2}
    \Ret(\tilde{x}_2', U_2) \cap H = \set{1_G}.
  \end{equation}
  Indeed, let $h \in H$ be such that $h \cdot \tilde x_2' \in U_2$. Then $\tilde{\pi}(\tilde x_2')(h^{-1} g) = 1$. However, $\tilde{\pi}(\tilde x_2')(g) = 1$ and there is at most one occurrence of $1$ in each right $H$-coset (as $H$ is normal, left and right cosets coincide). This implies that $h = 1_G$.

  Combining \autoref{eq:p:scnss:1} and \autoref{eq:p:scnss:2}, we see that whenever $p_1$ occurs in $x_1'$ and $p_2$ occurs in $\tilde x_2'$, the point $(x_1', \tilde x_2')$ cannot be recurrent, which concludes the proof.
\end{proof}

\begin{theorem}
  \label{th:polycyclic-no-dense-conj-class}
Let $G$ be a virtually polycyclic, infinite group. Then there is a dense conjugacy class in $\Xi_\Tr(G)$ iff $G$ is  virtually cyclic.
\end{theorem}
\begin{proof}
  One direction follows from the fact that $\Xi_\Tr(\Z)$ has a dense conjugacy class (see \cite{Hochman2008}) and \autoref{p:dense-conj-class-finite-index}.

  For the other, note that, by \cite{Dahmani2008}*{Proposition~4.2}, virtually polycyclic groups have the special symbol property. By \cite{Segal1983}*{p.~16, Lemma~6}, there exists a short exact sequence
  \begin{equation*}
    1 \to \Z^d \to G \to Q \to 1
  \end{equation*}
  with $d \geq 1$. We have two cases. If $Q$ is infinite, then we are done by \autoref{p:suff-cond-nodense-special-symbol}. So assume that $Q$ is finite. As $G$ is not virtually cyclic, we have that $d \geq 2$. Let $a = (1, 0, \dots) \in \Z^d$ and let $\set{a_0, \dots, a_{n-1}} \sub \Z^d$ be the (finite) conjugacy class of $a$ in $G$. As $d \geq 2$, $\Z^d$ is not the union of finitely many cyclic subgroups, so there exists $b \in \Z^d$ which does not belong to $\bigcup_{i < n} \gen{a_i}$. We claim that the subgroups $\gen{a}$ and $\gen{b}$ of $G$ satisfy the hypotheses of \autoref{p:no-dense-class-2-subgroups}, which allows us to conclude. First, by \autoref{l:ssp-relative-quotient}, $\Z^d$ has the special symbol property relative to both $\gen{a}$ and $\gen{b}$ and, by \autoref{p:relative_ssp_stable_under_extension}, so does $G$. Second, $\gen{b} \cap \gen{a_i} = \set{1_G}$ for all $i$ because $a_i$ is the image of $a$ by an automorphism of $\Z^d$, so if $b^k = a_i^n$ for some $k, n \in \Z$, then $k | n$ and $b \in \gen{a_i}$, contradicting the choice of $b$.
\end{proof}

\subsection{Hyperbolic groups}
\label{sec:hyperbolic-groups}

We recall some facts about finite automata and hyperbolic groups.
Given a finite alphabet $S$, we denote by $S^*$ the set of all finite words on $S$; for a word $w \in S^*$, we denote its length by $|w|$.

\begin{defn}
  \label{df:automaton}
  A \df{finite automaton} is a tuple $\mathcal{A}=(\Sigma, S, \mu, e, P)$, where $\Sigma$ is a finite set of \df{states}, $S$ is a finite alphabet, $\mu \colon \Sigma \times S \to \Sigma$ is the \emph{transition function}, $e\in A$ is the \emph{initial state}, and $P \subseteq \Sigma$ is the set of \emph{accepting states}.

  The transition function $\mu$ extends to a map $\Sigma \times S^* \to \Sigma$, which we still denote by $\mu$, defined by induction as follows:
  \begin{equation*}
    \mu(p, \emptyset) = p; \quad \mu(p, ws) = \mu \bigl(\mu(p, w), s \bigr), \text{ for } p \in \Sigma, w \in S^*, s \in S.
  \end{equation*}
  A word $w \in S^*$ is \emph{accepted} by $\mathcal{A}$ if $\mu(e, w)\in P$.
  A language $L \sub S^*$ is called \df{regular} if it is the set of words accepted by some finite automaton.
\end{defn}

\begin{defn}
  Let $(S, \leq)$ be a finite linearly ordered set. The \emph{shortlex order} $\preceq$ on $S^*$ is defined by setting $u \preceq v$ iff $|u| < |v|$ or $|u| = |v|$ and $u$ is smaller than $v$ for the lexicographic order on $S^{|u|}$ associated to $\leq$.
\end{defn}

Note that, as a linear order, $(S^*, \preceq)$ is isomorphic to $(\N, \leq)$.

Let $G$ be a finitely generated group and fix a finite, symmetric generating set $S \sub G$ which does not contain $1_G$.
There is a natural surjective \emph{evaluation map} $\pi \colon S^* \to G$, mapping the word $s_1 \cdots s_n \in S^*$ to the product $s_1 \cdots s_n \in G$ (and $\emptyset$ to $1_G$).
To avoid clutter, we will often suppress the map $\pi$ from the notation.
For $g \in G$, we denote by $|g|$ the shortest length of a word $w \in S^*$ such that $\pi(w) = g$. In particular, $|1_G|=0$ and $|g^{-1}| = |g|$ for every $g \in G$.
We view $G$ as a metric space endowed with the left-invariant word metric $d$ given by $d(g, h) = |g^{-1}h|$.

Let now $S$ be endowed with some arbitrary fixed linear order.
Given $g \in G$, there are in general many words from $S^*$ that evaluate to $g$; however, there is always a least one with respect to the shortlex order, which we denote by $\overline{g}$. Notice that for every $g\in G$, $|\overline{g}|=|g|$, so the word $\overline{g}$ represents a shortlex minimal geodesic connecting $1_G$ and $g$ in the Cayley graph of $G$.
We will say that a word $w \in S^*$ is \df{shortlex} if $w = \overline{g}$ for some $g \in G$.
It follows from the definition that any subword of a shortlex word is also shortlex.

The triple $(G, S, \leq)$, where $G$ is a group, $S$ is a finite, symmetric generating set for $G$, and $\leq$ is a linear order on $S$, is called \df{shortlex automatic} if the language $\set{\overline{g} : g \in G}$ is regular.

Recall that the finitely generated group $G$ is \emph{hyperbolic} if there exists a \df{hyperbolicity constant} $\delta \ge 0$ such that each side of any geodesic triangle in $G$ is in the $\delta$-neighborhood of the union of the other two sides.
The property of being hyperbolic does not depend on the generating set $S$; we refer to \cite[Chapter III.$\Gamma$]{BriHaebook} for an introduction to hyperbolic groups.
Typical examples are free groups as well as the fundamental groups of surfaces of genus at least $2$ and many $3$-manifolds.
By \cite{Can} (see also \cite{HRbook}*{Section 5.8} for a modern treatment), $(G,S,\leq)$ is shortlex automatic whenever $G$ is hyperbolic, for any generating set $S$ and any linear order $\leq$ on $S$.

\begin{defn}
Let $G$ be a hyperbolic group generated by a finite, symmetric set $S$. A subgroup $H \le G$ is \emph{quasi-convex} if there exists a constant $K$ such that for every $h \in H$, the geodesic represented by $\overline{h}$ is in the $K$-neighborhood of $H$. This property does not depend on the choice of $S$.
\end{defn}

The following generalizes \cite{Dahmani2008}*{Proposition~4.1}, but the proof is different.
\begin{prop}
  \label{p:SSP-quasi-convex}
  Let $G$ be a hyperbolic group and let $H \leq G$ be a quasi-convex subgroup. Then $G$ has the special symbol property relative to $H$.
\end{prop}

To prove \autoref{p:SSP-quasi-convex}, we will require some preparation. We fix a finite, symmetric, generating set $S$ for $G$ and a linear order on $S$.
By \cite{GS}*{Theorem 2.2} (see also \cite{HRbook}*{Section 8.1}), the language $\{\overline{h} : h \in H\}$ is regular.
So we fix automata $\mathcal{A} = (\Sigma, S, \mu, e, P)$ and $\mathcal{A}' = (\Sigma', S, \mu', e', P')$ which accept the languages $\{\overline{g} : g \in G\}$ and $\{\overline{h} : h \in H\}$, respectively. It follows from the properties of shortlex words that $e \in P$ and for all $p, q \in P$, if $\mu(p, w) = q$, then $w$ is shortlex.

Set $B = S \times P \times \Sigma'$ and $C = B \cup \{*\}$, where $* = (\dag, e, e')$ and $\dag$ is a new symbol which does not belong to $S$.
For $c\in C$, we write $c=(c_1, c_2, c_3)$ with $c_1 \in S \cup \set{\dag}$, $c_2 \in P$ and $c_3 \in \Sigma'$.

In order to prove the special symbol property, we define a point $x \in C^G$ such that $\cl{G \cdot x} \sub C^G$ is an SFT of which $\cl{G \cdot \chi_H}$ is a factor.
The word $x$ has the property that for every $g \in G$, the sequence $\bigl(x(\overline{g}|_i) \bigr)_{i < |g|}$ records the runs of the automata $\cA$ and $\cA'$ on the word $\overline{g}$.
More formally, set $x(1_G) = *$ and for $g\in G \setminus \{1_G\}$ such that $\overline{g} = ws$ with $w \in S^*, s \in S$, define
\begin{equation}
  \begin{split}
    \label{eq:xg-local-rules}
    x(g)_1 &= s^{-1}, \\
    x(g)_2 &= \mu\bigl(x(w)_2, s\bigr), \\
    x(g)_3 &= \mu'\bigl(x(w)_3, s\bigr).
  \end{split}
\end{equation}
It follows that $x(g)_2 = \mu(e, \overline{g})$ and $x(g)_3 = \mu'(e', \overline{g})$ for all $g \in G$.
Note also that $x(g)_2 \in P$ for all $g$ because $\overline{g}$ is a word accepted by $\mathcal{A}$.

\begin{lemma}
  \label{l:Gx-SFT}
  The subshift $\cl{G \cdot x}$ is an SFT.
\end{lemma}
\begin{proof}
Let $\delta \geq 1$ be a hyperbolicity constant for $(G, S)$ and let
\[
  D = \{g\in G : |g| \leq \delta \}.
\]
In particular, $S \cup \set{1_G} \sub D$. We denote by $X \subseteq C^G$ the SFT defined by the $D$-patterns that appear in $x$. In other words, the finite set of forbidden patterns for $X$ is $C^D\setminus \{(g \cdot x) |_D\colon g\in G\}$.

Our goal is to show that $X = \cl{G \cdot x}$.
The inclusion $\supseteq$ being clear, we prove the other. We start with two claims that allow us to make calculations along shortlex geodesics in elements of $X$.
\begin{claim}
  \label{cl:Gx-SFT-back}
  Let $y\in X$, $g\in G$, and let $s = y(g)_1 \in S$. Then:
  \begin{equation*}
    y(g)_2 = \mu \bigl(y(gs)_2, s^{-1} \bigr) \quad \And \quad y(g)_3 = \mu' \bigl(y(gs)_3, s^{-1} \bigr).
  \end{equation*}
\end{claim}
\begin{proof}
  By the definition of $X$, there exists $g'\in G \setminus \set{1_G}$ such that $y(g) = x(g')$ and $y(gs) = x(g's)$. Denoting $\overline{g'} = wt$ with $w \in S^*, t \in S$, we observe that $s = y(g)_1 = x(g')_1 = t^{-1}$. We also have $\mu \bigl(x(w)_2, t \bigr) = x(g')_2$. Since
  \[
    y(gs)_2 = x(g's)_2 = x(g't^{-1})_2 = x(w)_2,
  \]
  we get $y(g)_2 = \mu \bigl(y(gs)_2,s^{-1} \bigr)$. The argument for $y(g)_3$ is analogous.
\end{proof}

\begin{claim}
  \label{cl:Gx-SFT-forw}
  Let $y \in X, g \in G, s \in S$, and suppose that $\mu \bigl(y(g)_2, s \bigr) \in P$. Then:
  \begin{equation*}
    y(gs)_1 = s^{-1}, \quad y(gs)_2 = \mu \bigl(y(g)_2, s \bigr), \quad y(gs)_3 = \mu' \bigl(y(g)_3, s \bigr).
  \end{equation*}
\end{claim}
\begin{proof}
  Choose $g'\in G$ such that $y(g) = x(g')$ and $y(gs) = x(g's)$. Since $\mu \bigl(y(g)_2, s \bigr) \in P$ and $x(g')_2 = y(g)_2$, the word $\overline{g'} s$ is shortlex, so $\overline{g's} = \overline{g'} s$. Thus, by the definition of $x$, $x(g's)_1 = s^{-1}$, $x(g's)_2 = \mu \bigl(y(g)_2, s \bigr)$, and $x(g's)_3 = \mu' \bigl(y(g)_3, s \bigr)$. Since $y(gs) = x(g's)$, the claim follows.
\end{proof}

Let now $y\in X$ in order to show that $y \in \cl{G \cdot x}$.
Suppose first that there is $g \in G$ such that $y(g) = *$. Then it is enough to consider the case $g = 1_G$ and by applying \autoref{cl:Gx-SFT-back} and \autoref{eq:xg-local-rules} inductively, we conclude that $y = x$ in this case.

So suppose that the symbol $*$ does not occur in $y$ and let $F = \set{g \in G : |g| \leq k}$ for some $k \in \N$.
Our goal is to find the pattern $y|_F$ in some translate of $x$.
The word $x$ is set up in such a way that when $y \in \cl{G \cdot x}$ and $*$ does not occur in $y$, we may think of $*$ as pushed to the boundary of $G$ along a geodesic.
To see this, we define inductively an infinite word $\alpha \in S^\N$ and the corresponding ray $\gamma \in G^\N$.
Set $\gamma_{0} = 1_G$. Assuming that $\gamma_{i}$ has been defined, we set $\alpha_{i} = y(\gamma_{i})_1$ and $\gamma_{i+1} = \gamma_{i} \alpha_{i}$.
By \autoref{cl:Gx-SFT-back}, $\mu \bigl(y(\gamma_{i+1})_2, \alpha_{i}^{-1} \bigr) = y(\gamma_{i})_2$ for all $i$.
It follows that $\mu \bigl(y(\gamma_{i+1})_2, \alpha_{i}^{-1} \cdots \alpha_0^{-1} \bigr) = y(1_G)_2$.
As $y(\gamma_{i+1})_2, y(1_G)_2 \in P$, the word $\alpha_{i}^{-1} \cdots \alpha_0^{-1}$ is shortlex.
In particular, $|\gamma_{i+1}^{-1}| = i + 1$, and hence also $|\gamma_{i+1}| = i + 1$. We conclude that the ray $\gamma$ is geodesic.

Let $n \geq k + \delta + 1$.
By the definition of $X$, there exists $g_0 \in G$ such that $y(\gamma_n d)=x(g_0 d)$ for all $d \in D$.
Let $\phi \colon G \to G$ be the isometry defined by $\phi(g) = g_0 \gamma_n^{-1} g$ for $g \in G$.
Let $F' = \phi[F]$ and note that $\phi[\gamma_nD] = g_0D$.

We prove by induction, going backwards from $n$ to $0$, that
\begin{equation}
  \label{eq:l:gs:1}
  y(\gamma_i) = x(\phi(\gamma_i)) \quad \text{for all } i = n, \dots, 0.
\end{equation}
For $i = n$, this is true by the choice of $g_0$ (as $1_G \in D$).
Suppose it is true for $i+1$ in order to prove it for $i$.
By the definition of $\alpha$, $y(\gamma_i)_1 = \alpha_i$.
Then by \autoref{cl:Gx-SFT-back}, $y(\gamma_i)_2 = \mu(y(\gamma_{i+1})_2, \alpha_i^{-1})$, and similarly, for $y(\gamma_i)_3$.

In particular, setting $\ell = \phi(1_G)$, we get $x(\ell) = y(1_G)$.
Consider the shortlex geodesic $\overline{\ell}$ connecting $1_G$ and $\ell$.
Let $m = |\ell|$, define $t_0 = x(\ell)_1$ and then, inductively, $t_{i+1} = x(\ell t_0 \cdots t_i)_1$ for all $i < m$.
By the definition of the automaton $\mathcal{A}$ and $x$, $\overline{\ell} = t_{m-1}^{-1}\ldots t_0^{-1}$.

We prove by induction that $\phi(\gamma_i) = \ell t_0 \cdots t_{i-1}$ and $t_i = \alpha_i$ for all $i < n$.
For $i = 0$, we have that $\phi(\gamma_0) = \phi(1_G) = \ell$ and $t_0 = x(\ell)_1 = y(1_G)_1 = \alpha_0$.
Suppose it is true for $i$ in order to prove it for $i+1$.
Then
\begin{equation*}
  \phi(\gamma_{i+1}) = \phi(\gamma_i \alpha_i) = \phi(\gamma_i)\alpha_i = \ell t_0 \cdots t_{i-1} \alpha_i = \ell t_0 \cdots t_i,
\end{equation*}
and $t_{i+1} = x(\ell t_0 \cdots t_i)_1 = x(\phi(\gamma_i)) = y(\gamma_i)$, by \autoref{eq:l:gs:1}.
In particular, $g_0 = \phi(\gamma_n)$ lies on the geodesic $\overline{\ell}$.

Let $f\in F$ be arbitrary and set $f' = \phi(f)$.
Consider the geodesic triangle with vertices $1_G$, $\ell$, and $f'$, where the geodesic between $1_G$ and $\ell$ is the shortlex word $\overline{\ell}$, the one between $1_G$ and $f'$ is the shortlex word $\overline{f'}$, and the one between $\ell$ and $f'$ is the image by $\phi$ of the shortlex geodesic between $1_G$ and $f$.
In particular, the latter one is contained in $\phi[F]$.
Note that, by the triangle inequality, $d(g_0, \phi[F]) = d(\gamma_n, F) \ge n - k \geq \delta + 1$.
It follows from the $\delta$-hyperbolicity of $G$ that $g_0$ must be $\delta$-close to some $h'$ lying on the geodesic $\overline{f'}$.

Pick such an $h'$.
By the definition of $D$ and the word metric, we have $h' \in g_0D$.
Let $h = \phi^{-1}(h')$.
The word $w \coloneq \overline{(h')^{-1}f'}$ is a shortlex geodesic connecting $h'$ to $f'$ as well as $h$ to $f$.
Denoting $w = s_1 \cdots s_m$, we prove by induction on $i$ that
\begin{equation}
  \label{eq:l:gs:2}
  y(h s_1 \cdots s_i) = x(h' s_1 \cdots s_i) \quad \text{for all } i \leq m.
\end{equation}
For $i = 0$, note that $h \in \gamma_nD$ (because $h' \in g_0D$), so, by the choice of $g_0$, $y(h) = x(h')$.
Suppose that \autoref{eq:l:gs:2} has been proved for $i$ in order to prove it for $i+1$.
We have that $\mu \bigl(y(h s_1 \cdots s_i)_2, s_{i+1} \bigr) = \mu \bigl(x(h' s_1 \cdots s_i)_2, s_{i+1} \bigr) \in P$.
So, by \autoref{cl:Gx-SFT-forw}, $y(h s_1 \cdots s_{i+1})_1 = s_{i+1}^{-1} = x(h' s_1 \cdots s_{i+1})_1$,
\begin{equation*}
  \begin{split}
    y(h s_1 \cdots s_{i+1})_2 & = \mu \bigl(y(h s_1 \cdots s_i)_2, s_{i+1} \bigr) \\
                      & = \mu \bigl(x(h' s_1 \cdots s_i)_2, s_{i+1} \bigr) = x(h' s_1 \cdots s_{i+1}),
  \end{split}
\end{equation*}
and similarly for $y(h s_1 \cdots s_{i+1})_3$.
Applying \autoref{eq:l:gs:2} for $i = m$, we conclude that $y(f)= x(g_0\gamma_n^{-1}f)$. As $f \in F$ was arbitrary, this completes the proof.
\end{proof}

\begin{proof}[Proof of \autoref{p:SSP-quasi-convex}]
  By definition of the automaton $\mathcal{A}'$, for every $g \in G$, we have $g \in H \iff x(g)_3 \in P'$. Therefore, the map $\phi \colon C \to \{0,1\}$ defined by $\phi(c) = 1 \iff c_3\in P'$, induces a continuous equivariant map $\Phi \colon C^G \to \{0,1\}^G$ such that $\Phi(x) = \chi_H$. In particular, $\Phi[\cl{G \cdot x}] = \cl{G \cdot \chi_H}$.
  By \autoref{l:Gx-SFT}, $\cl{G \cdot x}$ is an SFT and therefore $\cl{G \cdot \chi_H}$ is sofic.
\end{proof}

We recall that a subgroup of a hyperbolic group $G$ is called \df{elementary} if it is virtually cyclic.
It is \df{maximal elementary} if it is not properly contained in any other elementary subgroup.
The following is well-known and we include its proof for completeness.
\begin{lemma}
  \label{l:hyperbolic-non-conj-cyclic}
  Let $G$ be a non-elementary hyperbolic group. Then there exist two maximal elementary subgroups $H_1, H_2 \leq G$ such that $g H_1 g^{-1} \cap H_2$ is finite for every $g \in G$.
\end{lemma}
\begin{proof}
  It follows from \cite[Corollary~8.2.G]{Gromov87} that there exist maximal elementary subgroups of $G$ which are not conjugate.
  Let $H_1$ and $H_2$ be two such subgroups.
  Suppose, towards a contradiction, that $gH_1g^{-1} \cap H_2$ is infinite for some $g \in G$.
  Then the intersection contains an element $\gamma$ of infinite order and by \cite[Corollary 8.2.C]{Gromov87}, $\gamma$ is contained in a unique maximal elementary subgroup.
  This implies that $g H_1 g^{-1} = H_2$, contradicting the fact that $H_1$ and $H_2$ are not conjugate.
\end{proof}

\begin{theorem}
  \label{th:hyperbolic-groups}
  Let $G$ be an infinite hyperbolic group. Then there is a dense conjugacy class in $\Xi_\Tr(G)$ iff $G$ is virtually cyclic.
\end{theorem}
\begin{proof}
  We only have to prove that if $G$ is not virtually cyclic, then there is no dense conjugacy class in $\Xi_\Tr(G)$.
  Apply \autoref{l:hyperbolic-non-conj-cyclic} to find $H_1, H_2 \leq G$ maximal elementary such that $gH_1g^{-1} \cap H_2$ is finite for all $g \in G$.
  If $\gamma_i \in H_i$ is of infinite order, then the centralizer of $\gamma_i$ is contained in $H_i$ and has finite index therein; it follows from \cite[Proposition 3.9]{BriHaebook} and the fact that quasi-convexity is preserved under finite extensions that
  $H_1$ and $H_2$ are quasi-convex in $G$.
  By \autoref{p:SSP-quasi-convex}, $G$ has the special symbol property relative to both $H_1$ and $H_2$.
  Now we can conclude by \autoref{p:no-dense-class-2-subgroups}.
\end{proof}

%%%%%%%%%%%%%%%%%%%%%%%%%%%%%%%%%%%%%%%%%%%%%%%%%%

\appendix

\section{Topological Fraïssé classes}
\label{sec:append-g_delta-isom}

The classical Fraïssé theorem ensures that every countable amalgamation class $\cF$ has a \df{Fraïssé limit}, i.e., a homogeneous structure with age $\cF$. It is well known that the isomorphism class of this Fraïssé limit is comeager in the space of all countable structures with age included in $\cF$ (in an appropriate Polish topology). Here, we establish a criterion for the existence of a comeager isomorphism class in the case where $\cF$ is no longer countable but is equipped with a natural Polish topology. The theorem below is inspired by classical results about model-theoretic forcing and atomic models, however, it differs in two important points. First, we do not assume compactness, and second, we work exclusively with quantifier-free types. The lack of compactness assumption is important because the spaces of minimal subshifts, to which we apply the theorem, are not compact. A version of this criterion for general Cantor actions was proved in \cite{Doucha2022p}.

We quickly recall some definitions from model theory. A \df{language} $\cL$ is a collection of function and relation symbols, together with their arities. An $\cL$-structure is a set equipped with interpretations for the $\cL$-symbols. A \df{term} is an expression using variables and function symbols. An \df{atomic formula} $\phi$ is an expression of the form $R(t_0, \ldots, t_{k-1})$, where $R$ is a relation symbol of arity $k$ and $t_0, \ldots, t_{k-1}$ are terms. A \df{quantifier-free formula} (or simply a \df{formula} in the sequel) is a Boolean combination of atomic formulas. If $\phi$ is a formula and $x$ is a set of variables, we will write $\phi(x)$ to mean that all variables that appear in $\phi$ are in $x$.
Two formulas $\phi_1(x)$ and $\phi_2(x)$ are \df{equivalent} if for all $\cL$-structures $M$ and tuples $a \in M^x$, $M \models \phi_1(a) \iff M \models \phi_2(a)$. Formulas in the variables $x$ modulo equivalence form a Boolean algebra and the compact space of ultrafilters of this algebra, denoted by $\tS_x(\cL)$, is called the space of \df{quantifier-free $\cL$-types} (which we will call simply \df{types} in what follows). A type in $\tS_x(\cL)$ can be viewed as an isomorphism type of an $\cL$-structure with generators $x$. A type $p \in \tS_x(\cL)$ \df{satisfies} a formula $\phi(x)$ (in symbols: $p \models \phi$) if $\phi \in p$. A clopen basis of the topology on $\tS_x(\cL)$ is given by sets of the form
\begin{equation*}
  \oset{\phi} = \set{p \in \tS_x(\cL) : p \models \phi}.
\end{equation*}
If $M$ is an $\cL$-structure and $a \in M^x$, the \df{type of $a$}, denoted by $\tp a$, is the element of $\tS_x(\cL)$ defined by
\begin{equation*}
  \tp a \models \phi \iff M \models \phi(a), \quad \text{ for all formulas } \phi(x).
\end{equation*}

If $\cL$ is countable, the type spaces are Polish. If $t(x)$ is a $y$-tuple of terms, then there is a natural projection $\pi_{t} \colon \tS_{x}(\cL) \to \tS_y(\cL)$ defined by
\begin{equation*}
  \pi_t(p) \models \phi(y) \iff p \models \phi(t(x)).
\end{equation*}
We will use this most often when $y \sub x$ and $t$ is the inclusion map, and then we will denote the projection by $\pi_y$.

Let $\cL$ be a countable language. A \df{hereditary $\cL$-class} $\cF$ is a collection of non-empty subsets $\tS_x(\cF) \sub \tS_x(\cL)$, for all finite $x$, such that for all $x$, $y$, all $y$-tuples of terms $t(x)$, and all $p \in \tS_x(\cF)$, we have that $\pi_{t}(p) \in \tS_y(\cF)$. If $u$ is a countable set of variables, we will denote by $\tS_u(\cF)$ the collection of types $p \in \tS_u(\cL)$ such that $\pi_x(p) \in \tS_x(\cF)$ for all finite $x \sub u$.
A \df{model of $\cF$} is an $\cL$-structure $M$ such that $\tp a \in \tS_x(\cF)$ for all finite $x$ and $a \in M^x$. Every type $p \in \tS_x(\cF)$ defines a structure generated by $x$, which is a model of $\cF$. In the case where the sets $\tS_x(\cF)$ are Borel (which is the case which will interest us), a hereditary class is given by a universal $L_{\omega_1 \omega}(\cL)$-theory. For a formula $\phi(x)$, we will denote $\oset{\phi}_\cF \coloneqq \oset{\phi} \cap \tS_x(\cF)$.
A formula $\phi(x)$ is \df{consistent with $\cF$} if $\oset{\phi}_\cF$ is non-empty.

\begin{defn}
  \label{df:top-Fraisse}
  Let $\cL$ be a countable language and let $\cF$ be a hereditary $\cL$-class. We say that $\cF$ is a \df{topological Fraïssé class} if the following conditions are satisfied:
  \begin{enumerate}
  \item \label{i:tF:amalg} $\cF$ has the \df{amalgamation property}: for all disjoint sets of variables $x, y_1, y_2$, and all $p \in \tS_x(\cF)$, $q_1 \in \tS_{xy_1}(\cF), q_2 \in \tS_{xy_2}(\cF) $ with $\pi_x(q_1) = \pi_x(q_2) = p$, there exists $s \in \tS_{xy_1y_2}(\cF)$ with $\pi_{xy_1}(s) = q_1$, $\pi_{xy_2}(s) = q_2$. (Here, as usual, $xy_1$ is a shortcut for $x \cup y_1$, etc.)

  \item \label{i:tF:Gdelta} For all finite $x$, $\tS_x(\cF)$ is a $G_\delta$ subset of $\tS_x(\cL)$ (and therefore a Polish space).

  \item \label{i:tF:cond-pi} For all finite $x \sub y$ and all open $U \sub \tS_{y}(\cF)$, the set $\pi_x[U]$ is $F_\sigma$ in $\tS_x(\cF)$.
  \end{enumerate}
\end{defn}

\begin{remark}
  \label{rem:top-Fraisse}
  In condition \autoref{i:tF:amalg} above, we allow $x = \emptyset$, so the amalgamation property includes the joint embedding property. For a classical Fraïssé class, one also requires that each $\tS_x(\cF)$ is countable. Each classical Fraïssé class can be made into a topological one by modifying the language, so that each $\tS_x(\cF)$ becomes discrete. Finally, condition \autoref{i:tF:cond-pi} is automatically satisfied if $\pi_x$ is open or closed, but also in other situations. In particular, if $\cF$ is given by a universal, first-order theory (i.e., a collection of finitary forbidden configurations) then $\tS_x(\cF)$ is compact and the only remaining condition is amalgamation.
\end{remark}

The notion of an existentially closed model is inspired from algebra and is key in Robinson's theory of model-theoretic forcing.
\begin{defn}
  \label{df:exist-closed}
  Let $\cF$ be a topological Fraïssé class. A model $M \models \cF$ is \df{existentially closed} if for all formulas
  $\phi(x, y)$ and $a \in M^x$, if there exist $N \supseteq M$, $N \models \cF$, and $b \in N^y$ such that $N \models \phi(a, b)$, then there exists $b' \in M^y$ such that $M \models \phi(a, b')$.
\end{defn}

The following proposition summarizes some of the basic properties of existentially closed models.
\begin{prop}
  \label{p:ec-models}
  Let $\cF$ be a topological Fraïssé class. Then:
  \begin{enumerate}
  \item \label{i:p:ec:model} Let $M \models \cF$, $a \in M^x$, and $q \in \tS_{xy}(\cF)$ be such that $a \models \pi_x(q)$. Then there exist $N \supseteq M$, $N \models \cF$, and $b \in N^y$ such that $(a, b) \models q$.
  \item \label{i:p:ec:char} A model $M$ is existentially closed iff for every formula $\phi(x, y)$, for every $a \in M^x$, and every $q \in \tS_{xy}(\cF)$ such that $\pi_x(q) = \tp(a)$ and $q \models \phi$, there exists $b \in M^y$ satisfying $M \models \phi(a, b)$.
  \item \label{i:p:ec:extension} Every model of $\cF$ embeds into an existentially closed one.
  \end{enumerate}
\end{prop}
\begin{proof}
  \autoref{i:p:ec:model} This follows from repeated amalgamation and the fact that models of $\cF$ are closed under unions of chains.

  \autoref{i:p:ec:char} Suppose that $M$ is existentially closed and let $a, q$, and $\phi$ be given. By \autoref{i:p:ec:model}, there is a model $N \supseteq M$ and $b \in N^y$ with $(a, b) \models q$. Now by the fact that $M$ is existentially closed, we have that there is $b' \in M^y$ with $M \models \phi(a, b')$. The other direction is obvious.

  \autoref{i:p:ec:extension} This follows by a standard argument from \autoref{i:p:ec:model}, \autoref{i:p:ec:char}, and taking unions of chains.
\end{proof}

Next we turn to different ways to code models of $\cF$ as elements of a Polish space. Let $u$ be a countable set of variables.
Define
\begin{multline}
  \label{eq:defn-Xi0}
  \Xi_0(\cL) = \set{\xi \in \tS_u(\cL) : \text{ for every finite } x \sub u \text{ and function symbol } F(x), \\
      \text{there exists } v \in u \  \xi \models F(x) = v }.
\end{multline}
Every $\xi \in \Xi_0(\cL)$ defines an $\cL$-structure as follows. Let $a \in M^u$ be a realization of $\xi$ (where $M$ is arbitrary). Then $\set{a(v) : v \in u}$ is a substructure of $M$ whose isomorphism class depends only on $\xi$ and which we denote by $M_\xi$. We let
\begin{equation*}
  \Xi_0(\cF) = \Xi_0(\cL) \cap \tS_u(\cF),
\end{equation*}
so that each element of $\Xi_0(\cF)$ codes a model of $\cF$. We note that if $\cF$ is a topological Fraïssé class, then $\Xi_0(\cF)$ is a $G_\delta$ subset of $\tS_u(\cL)$ and therefore a Polish space. In an abstract setting, the elements of $\Xi_0(\cF)$ are perhaps the most natural way to code models of $\cF$. However, in concrete situations, sometimes different codings are desirable. The framework we propose below is quite flexible and includes most codings that appear in the literature.

\begin{defn}
  \label{df:admissible-coding}
  Let $\cF$ be a topological Fraïssé class.
  We say that a set $\Xi \sub \Xi_0(\cF)$ is an \df{admissible coding for models of $\cF$} if the following conditions are satisfied:
  \begin{enumerate}
  \item \label{i:adm-cod:Gdelta} $\Xi$ is a non-empty $G_\delta$ subset of $\Xi_0(\cF)$.
  \item \label{i:adm-cod:proj-open} For every finite $x \sub u$, $\pi_x[\Xi]$ is an open subset of $\tS_x(\cF)$.
  \item \label{i:adm-cod:ec} If $x \sub u$ is finite, $y$ is a finite set of variables, and $\phi(x)$ and $\psi(x, y)$ are formulas such that $\oset{\phi}_\cF \sub \pi_x[\Xi]$ and $\phi(x) \land \psi(x, y)$ is consistent, then there is $\xi \in \Xi$ and $y' \sub u$ with $|y'| = |y|$ such that $\xi \models \phi(x) \land \psi(x, y')$.
  \end{enumerate}
\end{defn}

The following is clear.
\begin{lemma}
  \label{l:full-coding-admissible}
  Let $\cF$ be a topological Fraïssé class. Then $\Xi = \Xi_0(\cF)$ is an admissible coding.
\end{lemma}

For an admissible coding $\Xi$, we denote
\begin{equation*}
  \Xi_\ec \coloneqq \set{\xi \in \Xi : M_\xi \text{ is existentially closed}}.
\end{equation*}

Note that we do not require that the models coded by $\Xi$ meet every isomorphism class. However, it follows from the following proposition that there are many existentially closed models in $\Xi$.
\begin{prop}
  \label{p:ec-dense-Gdelta}
  Let $\cF$ be a topological Fraïssé class. If $\Xi$ is an admissible coding for models of $\cF$, then $\Xi_\ec$ is a dense $G_\delta$ subset of $\Xi$.
\end{prop}
\begin{proof}
  By \autoref{p:ec-models} \autoref{i:p:ec:char}, we have that
  \begin{multline*}
    \xi \in \Xi_\ec \iff \forall x \sub u \ \forall \phi(x, y) \big( \exists q \in \oset{\phi}_\cF  \ \pi_x(\xi) = \pi_x(q) \big) \\
    \implies \exists y' \sub u \ \xi \models \phi(x, y').
  \end{multline*}
  This condition is $G_\delta$ because the left side of the implication can be rewritten as
  \begin{equation*}
    \pi_x(\xi) \in \pi_x \big[\oset{\phi}_\cF \big],
  \end{equation*}
  which is an $F_\sigma$ set by condition \autoref{i:tF:cond-pi} in \autoref{df:top-Fraisse}. To verify density, by the Baire category theorem, it suffices to check that for fixed $\phi$ and $x$, the set of $\xi \in \Xi$ satisfying the implication is dense. Let $z \sub u$ be finite and let $\psi(x, z)$ be a formula such that $\oset{\psi} \cap \Xi$ is non-empty.
  Let $\xi_0 \in \oset{\psi} \cap \Xi$. We may assume that there is $q \in \oset{\phi}_\cF$ with $\pi_x(\xi_0) = \pi_x(q)$ (otherwise $\xi_0$ satisfies the implication). By \autoref{df:admissible-coding} \autoref{i:adm-cod:proj-open}, we may also assume that $\oset{\psi}_\cF \sub \pi_{yz}[\Xi]$. By amalgamation, the formula $\psi(x, z) \land \phi(x, y)$ is consistent, so by \autoref{df:admissible-coding} \autoref{i:adm-cod:ec}, there is $\xi \in \Xi$ and $y' \sub u$ such that $\xi \models \psi(x, z) \land \phi(x, y')$.
\end{proof}

\begin{defn}
  \label{df:proj-isolated}
  Let $\cF$ be a hereditary $\cL$-class. A type $p \in \tS_x(\cF)$ is called \df{isolated} if $p$ is an isolated point in $\tS_x(\cF)$ (in the relative topology). It is called \df{projectively isolated} if there exists a set of variables $y$ and a formula $\phi(x, y)$ such that $\pi_x\big[\oset{\phi}_\cF\big] = \set{p}$. In that case, we will say that $\phi$ \df{isolates} $p$. A model $M \models \cF$ is called \df{projectively atomic} if for all finite $x$ and all $a \in M^x$, $\tp a$ is projectively isolated in $\tS_x(\cF)$.
\end{defn}

We recall that a model $M$ is \df{homogeneous} if for all $n$ and $a, b \in M^n$ with $\tp a = \tp b$, there exists an automorphism $f$ of $M$ such that $f(a) = b$.

In classical model theory (where formulas are allowed to have quantifiers), the type projection maps $\pi_x$ are open and being projectively isolated is equivalent to being isolated. In that case, it follows from the omitting types theorem that the generic model (if it exists) is the \df{atomic} one, that is, the model that only realizes isolated types. Moreover, atomic models are always homogeneous.
The following theorem can be considered a generalization of these facts to the quantifier-free setting. Being existentially closed provides a sufficient ``quantifier elimination'' condition.
\begin{theorem}
  \label{th:proj-atomic}
  Let $\cL$ be a countable language, let $\cF$ be a topological Fraïssé $\cL$-class, and let $\Xi$ be an admissible coding for models of $\cF$. Then:
  \begin{enumerate}
  \item \label{i:th:pa:existence-proj-atomic} An existentially closed, projectively atomic model of $\cF$ exists iff the projectively isolated types are dense in $\tS_x(\cF)$ for every $x$.
  \item \label{i:th:pa:uniq-homog} If it exists, the existentially closed, projectively atomic model of $\cF$ is unique up to isomorphism and it is homogeneous.
  \item \label{i:th:pa:comeager} Let $M_0 \models \cF$. The following are equivalent:
    \begin{itemize}
    \item $M_0$ is existentially closed, projectively atomic;
    \item $\set{\xi \in \Xi : M_\xi \cong M_0}$ is dense $G_\delta$ in $\Xi$;
    \item $\set{\xi \in \Xi : M_\xi \cong M_0}$ is non-meager in $\Xi$.
    \end{itemize}
  \end{enumerate}
\end{theorem}
\begin{proof}
  \autoref{i:th:pa:uniq-homog} Let $M_1$ and $M_2$ be existentially closed, projectively atomic. We will show that
  \begin{equation*}
    \set{(a, b) \in M_1^n \times M_2^n : \tp a = \tp b}, \quad n \in \N
  \end{equation*}
  is a back-and-forth system, thus proving both statements simultaneously. Suppose that $a \in M_1^x$, $b \in M_2^x$, $\tp a = \tp b$, and let $c \in M_1$. As $\tp ac$ is projectively isolated, there is a formula $\phi(x, y, z)$ such that for every $q \models \phi$, $\pi_{xy}(q) = \tp ac$. As $M_2$ is existentially closed, by \autoref{p:ec-models} \autoref{i:p:ec:char}, there are $d \in M_2$, $d' \in M_2^z$ such that $M_2 \models \phi(b, d, d')$. Then $\tp bd = \tp ac$ and we are done.

  The direction from left to right of \autoref{i:th:pa:existence-proj-atomic} follows from the fact that the types realized in any existentially closed model are dense (because of \autoref{p:ec-models} \autoref{i:p:ec:char} applied to $x = \emptyset$). For the other direction and the first implication of \autoref{i:th:pa:comeager}, we will show that if projectively isolated types are dense, then the set of $\xi$ such that $M_\xi$ is existentially closed, projectively atomic is dense $G_\delta$. For existentially closed, this is given by \autoref{p:ec-dense-Gdelta}. To complete the proof using the Baire category theorem, it suffices to check that for every finite $x \sub u$, the set
  \begin{equation}
    \label{eq:proj-isolated}
    \set{\xi \in \Xi_\ec : \pi_x(\xi) \text{ is projectively isolated}}
  \end{equation}
  is open dense in $\Xi_\ec$. To see that it is open, let $\xi_0$ belong to this set and let $\phi(x, y)$ be a formula that isolates $\pi_x(\xi_0)$. As $M_{\xi_0}$ is existentially closed, there exists $y' \sub u$ with $|y'| = |y|$ such that $\xi_0 \models \phi(x, y')$. Then every $\xi$ satisfying $\xi \models \phi(x, y')$ belongs to the set \autoref{eq:proj-isolated}. We finally check that the set \autoref{eq:proj-isolated} is dense. Let $y \sub u$ and let $\psi(x, y)$ be a formula such that $\oset{\psi} \cap \Xi$ is non-empty. By \autoref{df:admissible-coding} \autoref{i:adm-cod:proj-open}, we may assume that $\oset{\psi}_\cF \sub \pi_{xy}(\Xi)$. Let $q \in \oset{\psi}_\cF$ be projectively isolated and let $\phi(x, y, z)$ be a formula that isolates $q$. Then, in particular, $\phi(x, y, z)$ implies $\psi(x, y)$. By \autoref{df:admissible-coding} \autoref{i:adm-cod:ec}, there is $z' \sub u$ with $|z'| = |z|$ and $\xi \in \Xi$ such that $\xi \models \phi(x, y, z')$. By \autoref{p:ec-dense-Gdelta}, we may also assume that $\xi \in \Xi_\ec$. Then $\pi_x(\xi)$ is projectively isolated and $\xi \in \oset{\psi} \cap \Xi_\ec$.

  The second implication of \autoref{i:th:pa:comeager} being trivial, we prove the third. It follows from \autoref{p:ec-dense-Gdelta} that $M_0$ is existentially closed. We will show that if $x \sub u$ is finite and $p \in \tS_x(\cF)$ is not projectively isolated, then the closed set
  \begin{equation*}
    \set{\xi \in \Xi : \pi_{x}(\xi) = p}
  \end{equation*}
  is nowhere dense in $\Xi$. Suppose to the contrary that it has non-empty interior. Then there exist $y \sub u$ and a formula $\phi(x, y)$ such that $\Xi \cap \oset{\phi} \neq \emptyset$ and for all $\xi \in \Xi \cap \oset{\phi}$, $\pi_x(\xi) = p$. By \autoref{df:admissible-coding} \autoref{i:adm-cod:proj-open}, we may assume that $\oset{\phi}_\cF \sub \pi_{xy}[\Xi]$. We claim that $\phi$ isolates $p$. Indeed, let $q \in \oset{\phi}_\cF$. Then there exists $\xi \in \Xi$ with $\pi_{xy}(\xi) = q$ and it follows that $\pi_x(q) = \pi_x(\xi) = p$. This is a contradiction.

  The same argument shows that for any $z \sub u$ with $|z| = |x|$, the set $\set{\xi \in \Xi : \pi_{z, x}(\xi) = p}$ is nowhere dense. Thus the set $\set{\xi \in \Xi : M_\xi \text{ realizes } p}$ is meager and it follows that $M_0$ does not realize $p$. As this is true for all non-projectively isolated $p$, we conclude that $M_0$ is projectively atomic.
\end{proof}

\begin{cor}
  \label{c:proj-isolated-equivalence}
  Let $\cF$ be a topological Fraïssé class and let $\Xi$ be an admissible coding for models of $\cF$. Then the following are equivalent:
  \begin{enumerate}
  \item The set $\Xi$ admits a comeager isomorphism class;
  \item The set of projectively isolated types in $\tS_x(\cF)$ is dense for every $x$.
  \end{enumerate}
\end{cor}

\begin{cor}
  \label{c:proj-isolated-Fraisse}
  Let $\cF$ be a topological Fraïssé class. If the projectively isolated types are dense, then they form a classical Fraïssé class and their Fraïssé limit is the projectively atomic model of $\cF$.
\end{cor}
\begin{proof}
  It is clear that the set of projectively isolated types is countable (as there are only countably many formulas) and they form an amalgamation class by the classical Fraïssé theorem (as the projectively atomic model is homogeneous).
\end{proof}

\begin{example}[Admissible codings and expansions]
  \label{ex:expansions}
  A common way to code structures is as \emph{expansions} of a fixed infinite homogeneous structure $M_0$. Let $\cL_0$ be the language of $M_0$ and let $\cF_0$ be its age. Assume that $\tS_x(\cF_0)$ is discrete for every finite $x$. Let $\cL \supseteq \cL_0$ and let $\cF$ be a topological Fraïssé class with the property that every type in $\tS_x(\cF)$ can be realized in a model whose $\cL_0$-reduct is isomorphic to $M_0$. Let $u$ be a countable set of variables and let $c \in M_0^u$ be a tuple that bijectively enumerates $M_0$. Then
  \begin{equation*}
    \Xi = \set{\xi \in \Xi_0(\cF) : \xi|_{\cL_0} = \tp c}
  \end{equation*}
  is an admissible coding for $\cF$. Via the bijection $c$, we can think of the variables $u$ as indexed by $M_0$. Then the group $\Aut(M_0)$ acts on $\Xi$ by permuting the variables and its orbits are exactly the isomorphism classes.

  The simplest special case of this is where $\cL_0$ is the language consisting only of equality and $M_0$ is an infinite set. Then $\Xi$ is the collection of all bijective enumerations of models of $\cF$ and it is equipped with an $S_\infty$-action.
\end{example}

We conclude the appendix with two concrete examples.

\subsection{Groups}
\label{sec:appendix-groups}

Let $\cL$ be the language of groups, containing function symbols for multiplication, inverse, and the identity, and let  $\cF$ be the hereditary class given by the (universal, first-order) theory of groups. Then $\tS_x(\cF)$ is the compact space of marked groups with generators $x$. The class $\cF$ has the amalgamation property by the construction of free product with amalgamation. Also, the maps $\pi_x$ are closed (because $\tS_x(\cF)$ is compact), so $\cF$ is a topological Fraïssé class.

The following characterization of projectively isolated groups is essentially due to Rips and it is based on a theorem of Boone and Higman which states that every finitely generated group with a solvable word problem embeds into a simple subgroup of a finitely presented group.
\begin{prop}[Rips]
  \label{p:solvable-word-problem}
  Let $G$ be a finitely generated marked group (that is, an element of $\tS_x(\cF)$ for some finite $x$). The following are equivalent:
  \begin{enumerate}
  \item $G$ is projectively isolated;
  \item $G$ has a solvable word problem.
  \end{enumerate}
\end{prop}
\begin{proof}
  This follows from the proof of the main theorem of \cite{Rips1982}. One only needs to observe that the formula $P(\bar x)$ that isolates $G$ can be taken to be existential.
\end{proof}

Miller~III~\cite{Miller1981} has constructed a finitely presented group $G_0 = \gen{x}$ such that all of its non-trivial quotients have an unsolvable word problem. Thus the set of quotients of $G_0$ forms a non-empty open neighborhood in $\tS_x(\cF)$ which contains no projectively isolated group. (This example has already been used in \cite{Cornulier2007} to show that isolated groups are not dense.)  Applying \autoref{c:proj-isolated-equivalence}, we recover the result of \cite{Goldbring2023} that there is no generic isomorphism class in the space of countable groups. The proof of \cite{Goldbring2023} uses the construction of \cite{Miller1981}; this is perhaps not surprising, as both \autoref{p:solvable-word-problem} and \autoref{c:proj-isolated-equivalence} are equivalences, so computability theory seems intrinsic to the problem. Similar questions were also considered by Ivanov and Majcher in \cite{Ivanov2024p}.

\subsection{$G$-sets}
\label{sec:appendix-g-sets}

  Let $G$ be a fixed countable group and let $\cL$ be the language consisting of unary function symbols for each element of $G$. Let $\cF$ be the hereditary class given by the (universal, first-order) theory that says that the function symbols are interpreted by bijections that compose according to the law of $G$. The models of $\cF$ are the $G$-sets. The type of a singleton $a$ is determined by the stabilizer of $a$ in $G$ and one can identify the type space $\tS_1(\cF)$ with the compact space $\Sub(G)$ of subgroups of $G$. The amalgamation property is easy to check. The projection maps $\pi_x$ are open, so a type is projectively isolated iff it is isolated. Finally, it is easy to see that a type $p \in \tS_x(\cF)$ is isolated iff $\pi_y(p)$ is isolated for all $y \sub x, |y| = 1$. Combining all of this with \autoref{c:proj-isolated-equivalence}, we recover the result of \cite{Glasner2016} that the group $G$ admits a generic action iff isolated subgroups are dense in $\Sub(G)$.

\bibliography{comeager-cantor-actions}
\end{document}